\documentclass[a4paper]{amsart}

\usepackage{hyperref}
\usepackage{enumerate}
\usepackage{array} 
\usepackage{amsmath}             
\usepackage{amssymb}              
\usepackage{lmodern,bm}              
\usepackage{longtable}
\usepackage{tikz}
\usepackage[all]{xy}
\usepackage{dsfont}
\usepackage{pdflscape}
\usepackage{bigdelim}
\usepackage{multirow}
\usepackage{wasysym}
\usepackage{adjustbox}

\allowdisplaybreaks
\newtheorem{theorem}{Theorem}[section]

\newtheorem{lemma}[theorem]{Lemma}
\newtheorem{proposition}[theorem]{Proposition}
\newtheorem{corollary}[theorem]{Corollary}

\theoremstyle{definition}

\newtheorem{definition}[theorem]{Definition}
\newtheorem{convention}[theorem]{Convention}
\newtheorem{conjecture}[theorem]{Conjecture}

\newtheorem{example}[theorem]{Example}
\newtheorem{remark}[theorem]{Remark}

\numberwithin{equation}{theorem}

\newcommand{\Ver}{\mathcal{V}}
\newcommand{\Edg}{\mathcal{E}}
\newcommand{\BigWedge}{\mathord{\adjustbox{valign=B,totalheight=.6\baselineskip}{$\bigwedge$}}}

\newcommand{\brB}[1]{\Bigl[ #1 \Bigr]}
\newcommand{\INV}[1]{\bigl| #1 \bigr|}

\newcommand{\iE}[1]{\mathtt{#1}}
\newcommand{\iiE}[1]{\boldsymbol{#1}}
\newcommand{\iiiE}[1]{\mathfrak{#1}}
\newcommand{\ivE}[1]{\mathit{#1}}

\newcommand{\onee}[2]{
\draw[line width=0.5pt, draw=black] (#1,0)-- (#1,-0.38);
\node[font=\tiny,align=center] at (#1,-0.5) {$\iE{#2}$};
\draw[black, fill=black] (#1,0) circle [radius=1.5pt];
}

\newcommand{\oneeu}[2]{
\draw[line width=0.5pt, draw=black] (#1,0)-- (#1,0.38);
\node[font=\tiny,align=center] at (#1,0.5) {$\iE{#2}$};
\draw[black, fill=black] (#1,0) circle [radius=1.5pt];
}

\newcommand{\twoe}[2]{
\begin{scope}
\clip (#1-0.2,-0.2) rectangle (#1+0.2,0.6);
\draw[line width=0.5pt, draw=black]  (#1,0) .. controls (#1-0.6,0.5) and (#1+0.6,0.5) .. (#1,0) ;
\node[font=\tiny,align=center] at (#1,0.5) {$\iiE{#2}$};
\draw[black, fill=black] (#1,0) circle [radius=1.5pt];
\end{scope}
}

\newcommand{\twoep}[3]{
\begin{scope}
\clip (#1-0.2,#2-0.2) rectangle (#1+0.2,#2+0.6);
\draw[line width=0.5pt, draw=black]  (#1,#2) .. controls (#1-0.6,#2+0.5) and (#1+0.6,#2+0.5) .. (#1,#2) ;
\node[font=\tiny,align=center] at (#1,#2+0.5) {$\iiE{#3}$};
\draw[black, fill=black] (#1,#2) circle [radius=1.5pt];
\end{scope}
}

\newcommand{\twoepd}[3]{
\begin{scope}
\clip (#1-0.2,#2+0.2) rectangle (#1+0.2,#2-0.6);
\draw[line width=0.5pt, draw=black]  (#1,#2) .. controls (#1-0.6,#2-0.5) and (#1+0.6,#2-0.5) .. (#1,#2) ;
\node[font=\tiny,align=center] at (#1,#2-0.5) {$\iiE{#3}$};
\draw[black, fill=black] (#1,#2) circle [radius=1.5pt];
\end{scope}
}

\newcommand{\twoetl}[3]{
\begin{scope}
\clip (#1+0.2,#2-0.2) rectangle (#1-0.6,#2+0.2);
\draw[line width=0.5pt, draw=black]  (#1,#2) .. controls (#1-0.5,#2-0.6) and (#1-0.5,#2+0.6) .. (#1,#2) ;
\node[font=\tiny,align=center] at (#1-0.5,#2) {$\iiE{#3}$};
\draw[black, fill=black] (#1,#2) circle [radius=1.5pt];
\end{scope}
}

\newcommand{\twoetr}[3]{
\begin{scope}
\clip (#1-0.2,#2-0.2) rectangle (#1+0.6,#2+0.2);
\draw[line width=0.5pt, draw=black]  (#1,#2) .. controls (#1+0.5,#2-0.6) and (#1+0.5,#2+0.6) .. (#1,#2) ;
\node[font=\tiny,align=center] at (#1+0.5,#2) {$\iiE{#3}$};
\draw[black, fill=black] (#1,#2) circle [radius=1.5pt];
\end{scope}
}

\newcommand{\twoed}[2]{
\begin{scope}
\clip (#1-0.2,0.2) rectangle (#1+0.2,-0.6);
\draw[line width=0.5pt, draw=black]  (#1,0) .. controls (#1-0.6,-0.5) and (#1+0.6,-0.5) .. (#1,0) ;
\node[font=\tiny,align=center] at (#1,-0.5) {$\iiE{#2}$};
\draw[black, fill=black] (#1,0) circle [radius=1.5pt];
\end{scope}
}

\newcommand{\twoeAB}[3]{
\begin{scope}
\clip (#1-0.2,0.2) rectangle (#2+0.2,-0.6);
\draw[line width=0.5pt, draw=black]  (#1,0) .. controls (#1*0.5+#2*0.5-0.6,-0.5) and (#1*0.5+#2*0.5+0.6,-0.5) .. (#2,0) ;
\node[font=\tiny,align=center] at (#1*0.5+#2*0.5,-0.5) {$\iiE{#3}$};
\draw[black, fill=black] (#1,0) circle [radius=1.5pt];
\draw[black, fill=black] (#2,0) circle [radius=1.5pt];
\end{scope}
}

\newcommand{\twoeABfu}[4]{
\begin{scope}
\clip (#1-0.2,1) rectangle (#2+0.2,-0.6);
\draw[line width=0.5pt, draw=black]  (#1,0) .. controls (#1*0.5+#2*0.5-#4,#4) and (#1*0.5+#2*0.5+#4,#4) .. (#2,0) node[font=\tiny, midway, above=1pt] {$\iiE{#3}$};
\draw[black, fill=black] (#1,0) circle [radius=1.5pt];
\draw[black, fill=black] (#2,0) circle [radius=1.5pt];
\end{scope}
}

\newcommand{\twoeABfd}[4]{
\begin{scope}
\clip (#1-0.2,0.6) rectangle (#2+0.2,-0.6);
\draw[line width=0.5pt, draw=black]  (#1,0) .. controls (#1*0.5+#2*0.5-#4,-#4) and (#1*0.5+#2*0.5+#4,-#4) .. (#2,0) ;
\node[font=\tiny,align=center] at (#1*0.5+#2*0.5,#4*#4-#4*1.5) {$\iiE{#3}$};
\draw[black, fill=black] (#1,0) circle [radius=1.5pt];
\draw[black, fill=black] (#2,0) circle [radius=1.5pt];
\end{scope}
}

\newcommand{\twoeABS}[3]{
\draw[line width=0.5pt, draw=black]  (#1,0) -- (#2,0) ;
\node[font=\tiny,align=center] at (#1*0.5+#2*0.5,0.1) {$\iiE{#3}$};
\draw[black, fill=black] (#1,0) circle [radius=1.5pt];
\draw[black, fill=black] (#2,0) circle [radius=1.5pt];
}

\newcommand{\twoeARB}[6]{
\draw[line width=0.5pt, draw=black]  (#1,#2) -- (#3,#4) node[font=\tiny, midway, #6] {$\iiE{#5}$} ;
\draw[black, fill=black] (#1,#2) circle [radius=1.5pt];
\draw[black, fill=black] (#3,#4) circle [radius=1.5pt];
}

\newcommand{\thre}[2]{
\begin{scope}
\clip (#1-0.2,-0.2) rectangle (#1+0.2,0.6);
\draw[line width=0.5pt, draw=black] (#1,0)-- (#1,0.38);
\draw[line width=0.5pt, draw=black]  (#1,0) .. controls (#1-0.6,0.5) and (#1+0.6,0.5) .. (#1,0) ;
\node[font=\tiny,align=center] at (#1,0.5) {$\iiiE{#2}$};
\draw[black, fill=black] (#1,0) circle [radius=1.5pt];
\end{scope}
}

\newcommand{\thred}[2]{
\begin{scope}
\clip (#1-0.2,#2+0.2) rectangle (#1+0.2,#2-0.5);
\draw[line width=0.5pt, draw=black] (#1,#2)-- (#1,#2-0.38);
\draw[line width=0.5pt, draw=black]  (#1,#2) .. controls (#1-0.6,#2-0.5) and (#1+0.6,#2-0.5) .. (#1,#2) ;
\draw[black, fill=black] (#1,#2) circle [radius=1.5pt];
\end{scope}
}

\newcommand{\threu}[2]{
\begin{scope}
\clip (#1-0.2,#2-0.2) rectangle (#1+0.2,#2+0.5);
\draw[line width=0.5pt, draw=black] (#1,#2)-- (#1,#2+0.38);
\draw[line width=0.5pt, draw=black]  (#1,#2) .. controls (#1-0.6,#2+0.5) and (#1+0.6,#2+0.5) .. (#1,#2) ;
\draw[black, fill=black] (#1,#2) circle [radius=1.5pt];
\end{scope}
}

\newcommand{\thretr}[3]{
\begin{scope}
\clip (#1-0.2,#2-0.2) rectangle (#1+0.6,#2+0.2);
\draw[line width=0.5pt, draw=black]  (#1,#2) .. controls (#1+0.5,#2-0.6) and (#1+0.5,#2+0.6) .. (#1,#2) ;
\node[font=\tiny,align=center] at (#1+0.5,#2) {$\iiiE{#3}$};
\draw[black, fill=black] (#1,#2) circle [radius=1.5pt];
\draw[line width=0.5pt, draw=black] (#1,#2)-- (#1+0.38,#2);
\end{scope}
}

\newcommand{\thretl}[3]{
\begin{scope}
\clip (#1-0.6,#2-0.2) rectangle (#1+0.2,#2+0.2);
\draw[line width=0.5pt, draw=black]  (#1,#2) .. controls (#1-0.5,#2-0.6) and (#1-0.5,#2+0.6) .. (#1,#2) ;
\node[font=\tiny,align=center] at (#1-0.5,#2) {$\iiiE{#3}$};
\draw[black, fill=black] (#1,#2) circle [radius=1.5pt];
\draw[line width=0.5pt, draw=black] (#1,#2)-- (#1-0.38,#2);
\end{scope}
}

\newcommand{\threAB}[3]{
\begin{scope}
\clip (#1-0.2,0.2) rectangle (#2+0.2,-0.6);
\draw[line width=0.5pt, draw=black] (#2,0)-- (#2,-0.38);
\draw[line width=0.5pt, draw=black]  (#1,0) .. controls (#2-0.6,-0.5) and (#2+0.6,-0.5) .. (#2,0) ;
\node[font=\tiny,align=center] at (#2,-0.5) {$\iiiE{#3}$};
\draw[black, fill=black] (#1,0) circle [radius=1.5pt];
\draw[black, fill=black] (#2,0) circle [radius=1.5pt];
\end{scope}
}

\newcommand{\threABC}[7]{
\draw[line width=0.5pt, draw=black] (#1,#2)-- (#1*0.3333+#3*0.3333+#5*0.3333,#2*0.3333+#4*0.3333+#6*0.3333);
\draw[line width=0.5pt, draw=black] (#3,#4)-- (#1*0.3333+#3*0.3333+#5*0.3333,#2*0.3333+#4*0.3333+#6*0.3333);
\draw[line width=0.5pt, draw=black] (#5,#6)-- (#1*0.3333+#3*0.3333+#5*0.3333,#2*0.3333+#4*0.3333+#6*0.3333) node[font=\tiny, near end, right=2pt] {$\iiiE{#7}$};
\draw[black, fill=black] (#1,#2) circle [radius=1.5pt];
\draw[black, fill=black] (#3,#4) circle [radius=1.5pt];
\draw[black, fill=black] (#5,#6) circle [radius=1.5pt];
}

\newcommand{\threABCT}[7]{
\draw[line width=0.5pt, draw=black] (#1,#2)-- (#1*0.5+#3*0.5,#2*0.5+#4*0.5);
\draw[line width=0.5pt, draw=black] (#3,#4)-- (#1*0.5+#3*0.5,#2*0.5+#4*0.5);
\draw[line width=0.5pt, draw=black] (#5,#6)-- (#1*0.5+#3*0.5,#2*0.5+#4*0.5) node[font=\tiny, near end, right=2pt] {$\iiiE{#7}$};
\draw[black, fill=black] (#1,#2) circle [radius=1.5pt];
\draw[black, fill=black] (#3,#4) circle [radius=1.5pt];
\draw[black, fill=black] (#5,#6) circle [radius=1.5pt];
}

\newcommand{\threABARB}[7]{
\draw[line width=0.5pt, draw=black] (#1,#2)--(#3,#4)  node[font=\tiny, midway, above=3pt] {$\iiiE{#5}$};
\draw[line width=0.5pt, draw=black]  (#1*0.5+#3*0.5,#2*0.5+#4*0.5).. controls (#1*0.5+#3*0.5+#6,#2*0.5+#4*0.5+#7)..(#3,#4);
\draw[black, fill=black] (#1,#2) circle [radius=1.5pt];
\draw[black, fill=black] (#3,#4) circle [radius=1.5pt];
}

\newcommand{\threABS}[3]{
\begin{scope}
\clip (#1-0.2,0.5) rectangle (#2+0.2,-0.5);
\draw[line width=0.5pt, draw=black] (#1,0)-- (#2,0);
\draw[line width=0.5pt, draw=black]  (#1*0.5+#2*0.5,0) .. controls (#1*0.25+#2*0.75-0.1,-0.15) and (#1*0.25+#2*0.75+0.1,-0.15) .. (#2,0);
\node[font=\tiny,align=center] at (#1*0.5+#2*0.5,0.1) {$\iiiE{#3}$};
\draw[black, fill=black] (#1,0) circle [radius=1.5pt];
\draw[black, fill=black] (#2,0) circle [radius=1.5pt];
\end{scope}
}

\newcommand{\threBA}[3]{
\begin{scope}
\clip (#1-0.2,0.2) rectangle (#2+0.2,-0.6);
\draw[line width=0.5pt, draw=black] (#1,0)-- (#1,-0.38);
\draw[line width=0.5pt, draw=black]  (#1,0) .. controls (#1-0.6,-0.5) and (#1+0.6,-0.5) .. (#2,0) ;
\node[font=\tiny,align=center] at (#1,-0.5) {$\iiiE{#3}$};
\draw[black, fill=black] (#1,0) circle [radius=1.5pt];
\draw[black, fill=black] (#2,0) circle [radius=1.5pt];
\end{scope}
}

\newcommand{\foureu}[3]{
\begin{scope}
\clip (#1-0.2,#2-0.2) rectangle (#1+0.2,#2+0.6);
\draw[line width=0.5pt, draw=black]  (#1,#2) .. controls (#1-0.6,#2+0.5) and (#1+0.6,#2+0.5) .. (#1,#2) node[font=\tiny, midway, above=1pt] {$\ivE{#3}$};
\draw[line width=0.5pt, draw=black]  (#1,#2) .. controls (#1-0.3,#2+0.5) and (#1+0.3,#2+0.5) .. (#1,#2);
\draw[black, fill=black] (#1,#2) circle [radius=1.5pt];
\end{scope}
}

\newcommand{\Vaab}[1]{
\begin{scope}
\clip (#1-0.2,-0.5) rectangle (#1+0.2,0.5);
\twoe{#1}{}
\draw[line width=0.5pt, draw=black] (#1,0)-- (#1-0.2,-0.38);
\draw[line width=0.5pt, draw=black] (#1,0)-- (#1+0.2,-0.38);
\draw[black, fill=black] (#1,0) circle [radius=1.5pt];
\end{scope}
}

\newcommand{\gra}[1]{
\draw[line width=0.5pt, draw=black] (#1,0)-- (#1,0.38);
\draw[line width=0.5pt, draw=black] (#1,0)-- (#1+0.1,0.2);
\draw[line width=0.5pt, draw=black] (#1,0)-- (#1-0.1,0.2);
\node[circle, fill=white,draw,inner sep=1pt]  at (#1,0.38) {\small{$\Gamma$}};
\draw[black, fill=black] (#1,0) circle [radius=1.5pt];
}

\newcommand{\grb}[2]{
\draw[line width=0.5pt, draw=black] (#1,0)-- (#1,0.38);
\draw[line width=0.5pt, draw=black] (#1,0)-- (#1+0.1,0.2);
\draw[line width=0.5pt, draw=black] (#1,0)-- (#1-0.1,0.2);
\node[circle, fill=white,draw,inner sep=0.5pt]  at (#1,0.38) {\small{$#2$}};
\draw[black, fill=black] (#1,0) circle [radius=1.5pt];
}

\newcommand{\grc}[1]{
\draw[line width=0.5pt, draw=black] (#1,0)-- (#1+0.1,0.2);
\draw[line width=0.5pt, draw=black] (#1,0)-- (#1-0.1,0.2);
\node[circle, fill=white,draw,inner sep=1pt]  at (#1,0.38) {\small{$\Gamma$}};
\draw[black, fill=black] (#1,0) circle [radius=1.5pt];
}

\newcommand{\vera}[3]{
\node[circle, fill=white,draw,inner sep=1pt]  at (#1,#2) {\tiny{$#3$}};
}

\newcommand{\gA}[4]{
\begin{tikzpicture}[baseline=(A),outer sep=0pt,inner sep=0pt]

\node (A) at (0,0) {};

\node[font=\tiny,align=center] at (0,0.5) {$\iE{#1}$};

\node[font=\tiny,align=left] at (0.5,0) {$\iE{#2}$};
\node[font=\tiny,align=right] at (-0.5,0) {$\iE{#4}$};

\draw[line width=0.5pt, draw=black] (0.38,0)-- (-0.38,0);

\node[font=\tiny,align=center] at (0,-0.5) {$\iE{#3}$};
\draw[line width=0.5pt, draw=black] (0,0.38)-- (0,-0.38);

\draw[black, fill=black] (0,0) circle [radius=1.5pt];

\end{tikzpicture}
}

\newcommand{\gB}[4]{
\begin{tikzpicture}[baseline=(A),outer sep=0pt,inner sep=0pt]

\node (A) at (0,0) {};

\node[font=\tiny,align=center] at (0,-0.5) {$\iiiE{#4}$};

\draw[line width=0.5pt, draw=black] (0,0)-- (0,-0.38);

\draw[black, fill=black] (0,0) circle [radius=1.5pt];

\draw[black, fill=black] (0.7,0) circle [radius=1.5pt];

\draw[black, fill=black] (-0.7,0) circle [radius=1.5pt];

\thre{-0.7}{#1}

\thre{0}{#2}

\thre{0.7}{#3}

\draw[line width=0.5pt, draw=black]  (0.7,0) .. controls (0.6,-0.5) and (-0.6,-0.5) .. (-0.7,0) ;

\end{tikzpicture}
}

\newcommand{\gCa}[1]{
\begin{tikzpicture}[baseline=(A),outer sep=0pt,inner sep=0pt]
\node (A) at (0,0) {};
\twoe{0}{#1}
\twoed{0}{#1}
\draw[black, fill=black] (0,0) circle [radius=1.5pt];
\end{tikzpicture}
}

\newcommand{\gCb}[2]{
\begin{tikzpicture}[baseline=(A),outer sep=0pt,inner sep=0pt]
\node (A) at (0,0) {};
\twoe{0}{#1}
\twoed{0}{#2}
\draw[black, fill=black] (0,0) circle [radius=1.5pt];
\end{tikzpicture}
}

\newcommand{\gD}[6]{
\begin{tikzpicture}[baseline=(A),outer sep=0pt,inner sep=0pt]

\node (A) at (0,0) {};

\node[font=\tiny,align=center] at (0,-0.5) {$\iiE{#6}$};

\node[font=\tiny,align=center] at (-0.35,0.1) {$\iiE{#2}$};
\node[font=\tiny,align=center] at (0.35,0.1) {$\iiE{#4}$};

\draw[line width=0.5pt, draw=black] (-0.7,0)-- (0.7,0);

\draw[black, fill=black] (0,0) circle [radius=1.5pt];

\draw[black, fill=black] (0.7,0) circle [radius=1.5pt];

\draw[black, fill=black] (-0.7,0) circle [radius=1.5pt];

\twoe{-0.7}{#1}

\twoe{0}{#3}

\twoe{0.7}{#5}

\draw[line width=0.5pt, draw=black]  (0.7,0) .. controls (0.6,-0.5) and (-0.6,-0.5) .. (-0.7,0) ;
\end{tikzpicture}
}

\newcommand{\gDA}[6]{
\begin{tikzpicture}[baseline=(A),outer sep=0pt,inner sep=0pt]

\node (A) at (0,0) {};

\node[font=\tiny,align=center] at (0,-0.5) {$\iiE{#6}$};

\node[font=\tiny,align=center] at (-0.35,0.1) {$\iiE{#2}$};
\node[font=\tiny,align=center] at (0.35,0.1) {$\iiE{#4}$};

\draw[line width=0.5pt, draw=black] (-0.7,0)-- (0,0);
\draw[line width=0.5pt, draw=black, dotted] (0.7,0)-- (0,0);

\draw[black, fill=black] (0,0) circle [radius=1.5pt];

\draw[black, fill=black] (0.7,0) circle [radius=1.5pt];

\draw[black, fill=black] (-0.7,0) circle [radius=1.5pt];

\twoe{-0.7}{#1}

\twoe{0}{#3}

\twoe{0.7}{#5}

\draw[line width=0.5pt, draw=black]  (0.7,0) .. controls (0.6,-0.5) and (-0.6,-0.5) .. (-0.7,0) ;
\end{tikzpicture}
}

\newcommand{\gDf}[8]{
\begin{tikzpicture}[baseline=(A),outer sep=0pt,inner sep=0pt]

\node (A) at (0,0) {};

\node[font=\tiny,align=center] at (0.35,-0.5) {$\iiE{#8}$};

\node[font=\tiny,align=center] at (-0.35,0.1) {$\iiE{#2}$};
\node[font=\tiny,align=center] at (0.35,0.1) {$\iiE{#4}$};

\node[font=\tiny,align=center] at (1.05,0.1) {$\iiE{#6}$};

\draw[line width=0.5pt, draw=black] (1.4,0)-- (-0.7,0);

\draw[black, fill=black] (0,0) circle [radius=1.5pt];

\draw[black, fill=black] (1.4,0) circle [radius=1.5pt];

\draw[black, fill=black] (0.7,0) circle [radius=1.5pt];

\draw[black, fill=black] (-0.7,0) circle [radius=1.5pt];

\twoe{-0.7}{#1}

\twoe{0}{#3}

\twoe{0.7}{#5}
\twoe{1.4}{#7}

\draw[line width=0.5pt, draw=black]  (1.4,0) .. controls (0.95,-0.5) and (-0.25,-0.5) .. (-0.7,0) ;
\end{tikzpicture}
}

\newcommand{\gE}[2]{
\begin{tikzpicture}[baseline=(A),outer sep=0pt,inner sep=0pt]
\node (A) at (0,0) {};
\thre{0}{#1}
\node[font=\tiny,align=center] at (0,-0.5) {$\iE{#2}$};
\draw[line width=0.5pt, draw=black] (0,0)-- (0,-0.38);
\draw[black, fill=black] (0,0) circle [radius=1.5pt];
\end{tikzpicture}
}

\newcommand{\gF}[3]{
\begin{tikzpicture}[baseline=(A),outer sep=0pt,inner sep=0pt]
\node (A) at (0,0) {};
\twoe{0}{#1}
\node[font=\tiny,align=center] at (-0.2,-0.5) {$\iE{#2}$};
\node[font=\tiny,align=center] at (0.2,-0.5) {$\iE{#3}$};
\draw[line width=0.5pt, draw=black] (0,0)-- (-0.2,-0.38);
\draw[line width=0.5pt, draw=black] (0,0)-- (0.2,-0.38);
\draw[black, fill=black] (0,0) circle [radius=1.5pt];
\end{tikzpicture}
}

\newcommand{\gGa}[4]{
\begin{tikzpicture}[baseline=(A),outer sep=0pt,inner sep=0pt]
\node (A) at (0,0) {};
\twoe{0}{#1}
\node[font=\tiny,align=center] at (0.35,0.1) {$\iiE{#2}$};
\twoe{0.7}{#3}
\node[font=\tiny,align=center] at (0,-0.5) {$\iE{#4}$};
\node[font=\tiny,align=center] at (0.7,-0.5) {$\iE{#4}$};
\draw[line width=0.5pt, draw=black] (0,0)-- (0.0,-0.38);
\draw[line width=0.5pt, draw=black] (0.7,0)-- (0.7,-0.38);
\draw[line width=0.5pt, draw=black] (0.7,0)-- (0,0);
\draw[black, fill=black] (0,0) circle [radius=1.5pt];
\draw[black, fill=black] (0.7,0) circle [radius=1.5pt];
\end{tikzpicture}
}

\newcommand{\gGb}[5]{
\begin{tikzpicture}[baseline=(A),outer sep=0pt,inner sep=0pt]
\node (A) at (0,0) {};
\twoe{0}{#1}
\node[font=\tiny,align=center] at (0.35,0.1) {$\iiE{#2}$};
\twoe{0.7}{#3}
\node[font=\tiny,align=center] at (0,-0.5) {$\iE{#4}$};
\node[font=\tiny,align=center] at (0.7,-0.5) {$\iE{#5}$};
\draw[line width=0.5pt, draw=black] (0,0)-- (0.0,-0.38);
\draw[line width=0.5pt, draw=black] (0.7,0)-- (0.7,-0.38);
\draw[line width=0.5pt, draw=black] (0.7,0)-- (0,0);
\draw[black, fill=black] (0,0) circle [radius=1.5pt];
\draw[black, fill=black] (0.7,0) circle [radius=1.5pt];
\end{tikzpicture}
}

\newcommand{\gGc}[7]{
\begin{tikzpicture}[baseline=(A),outer sep=0pt,inner sep=0pt]
\node (A) at (0,0) {};
\twoe{0}{#1}
\node[font=\tiny,align=center] at (0.35,0.1) {$\iiE{#2}$};
\twoe{0.7}{#3}
\node[font=\tiny,align=center] at (1.05,0.1) {$\iiE{#4}$};
\twoe{1.4}{#5}
\node[font=\tiny,align=center] at (0,-0.5) {$\iE{#6}$};
\node[font=\tiny,align=center] at (1.4,-0.5) {$\iE{#7}$};
\draw[line width=0.5pt, draw=black] (0,0)-- (0.0,-0.38);
\draw[line width=0.5pt, draw=black] (1.4,0)-- (1.4,-0.38);
\draw[line width=0.5pt, draw=black] (1.4,0)-- (0,0);
\draw[black, fill=black] (0,0) circle [radius=1.5pt];
\draw[black, fill=black] (0.7,0) circle [radius=1.5pt];
\draw[black, fill=black] (1.4,0) circle [radius=1.5pt];
\end{tikzpicture}
}

\newcommand{\gGA}[7]{
\begin{tikzpicture}[baseline=(A),outer sep=0pt,inner sep=0pt]
\node (A) at (0,0) {};
\twoe{0}{#1}
\node[font=\tiny,align=center] at (0.35,0.1) {$\iiE{#2}$};
\twoe{0.7}{#3}
\node[font=\tiny,align=center] at (1.05,0.1) {$\iiE{#4}$};
\twoe{1.4}{#5}
\node[font=\tiny,align=center] at (0,-0.5) {$\iE{#6}$};
\node[font=\tiny,align=center] at (1.4,-0.5) {$\iE{#7}$};
\draw[line width=0.5pt, draw=black] (0,0)-- (0.0,-0.38);
\draw[line width=0.5pt, draw=black] (1.4,0)-- (1.4,-0.38);
\draw[line width=0.5pt, draw=black] (0.7,0)-- (0,0);
\draw[line width=0.5pt, draw=black, dotted] (0.7,0)-- (1.4,0);
\draw[black, fill=black] (0,0) circle [radius=1.5pt];
\draw[black, fill=black] (0.7,0) circle [radius=1.5pt];
\draw[black, fill=black] (1.4,0) circle [radius=1.5pt];
\end{tikzpicture}
}

\newcommand{\gHa}[3]{
\begin{tikzpicture}[baseline=(A),outer sep=0pt,inner sep=0pt]
\node (A) at (0,0) {};
\thre{0}{#2}
\node[font=\tiny,align=center] at (0.35,0.1) {$\iiE{#1}$};
\thre{0.7}{#3}
\draw[line width=0.5pt, draw=black] (0.7,0)-- (0,0);
\draw[black, fill=black] (0,0) circle [radius=1.5pt];
\draw[black, fill=black] (0.7,0) circle [radius=1.5pt];
\end{tikzpicture}
}

\newcommand{\gHb}[4]{
\begin{tikzpicture}[baseline=(A),outer sep=0pt,inner sep=0pt]
\node (A) at (0,0) {};
\thre{0}{#4}
\node[font=\tiny,align=center] at (0.35,0.1) {$\iiE{#1}$};
\twoe{0.7}{#2}
\node[font=\tiny,align=center] at (1.05,0.1) {$\iiE{#3}$};
\thre{1.4}{#4}
\draw[line width=0.5pt, draw=black] (1.4,0)-- (0,0);
\draw[black, fill=black] (0,0) circle [radius=1.5pt];
\draw[black, fill=black] (0.7,0) circle [radius=1.5pt];
\draw[black, fill=black] (1.4,0) circle [radius=1.5pt];
\end{tikzpicture}
}

\newcommand{\gHc}[5]{
\begin{tikzpicture}[baseline=(A),outer sep=0pt,inner sep=0pt]
\node (A) at (0,0) {};
\thre{0}{#4}
\node[font=\tiny,align=center] at (0.35,0.1) {$\iiE{#1}$};
\twoe{0.7}{#2}
\node[font=\tiny,align=center] at (1.05,0.1) {$\iiE{#3}$};
\thre{1.4}{#5}
\draw[line width=0.5pt, draw=black] (1.4,0)-- (0,0);
\draw[black, fill=black] (0,0) circle [radius=1.5pt];
\draw[black, fill=black] (0.7,0) circle [radius=1.5pt];
\draw[black, fill=black] (1.4,0) circle [radius=1.5pt];
\end{tikzpicture}
}

\newcommand{\gHA}[5]{
\begin{tikzpicture}[baseline=(A),outer sep=0pt,inner sep=0pt]
\node (A) at (0,0) {};
\thre{0}{#4}
\node[font=\tiny,align=center] at (0.35,0.1) {$\iiE{#1}$};
\twoe{0.7}{#2}
\node[font=\tiny,align=center] at (1.05,0.1) {$\iiE{#3}$};
\thre{1.4}{#5}
\draw[line width=0.5pt, draw=black] (0.7,0)-- (0,0);
\draw[line width=0.5pt, draw=black, dotted] (1.4,0)-- (0.7,0);
\draw[black, fill=black] (0,0) circle [radius=1.5pt];
\draw[black, fill=black] (0.7,0) circle [radius=1.5pt];
\draw[black, fill=black] (1.4,0) circle [radius=1.5pt];
\end{tikzpicture}
}

\newcommand{\gHd}[7]{
\begin{tikzpicture}[baseline=(A),outer sep=0pt,inner sep=0pt]
\node (A) at (0,0) {};
\thre{0}{#6}
\node[font=\tiny,align=center] at (0.35,0.1) {$\iiE{#1}$};
\twoe{0.7}{#2}
\node[font=\tiny,align=center] at (1.05,0.1) {$\iiE{#3}$};
\twoe{1.4}{#4}
\node[font=\tiny,align=center] at (1.75,0.1) {$\iiE{#5}$};
\thre{2.1}{#7}
\draw[line width=0.5pt, draw=black] (2.1,0)-- (0,0);
\draw[black, fill=black] (0,0) circle [radius=1.5pt];
\draw[black, fill=black] (0.7,0) circle [radius=1.5pt];
\draw[black, fill=black] (1.4,0) circle [radius=1.5pt];
\draw[black, fill=black] (2.1,0) circle [radius=1.5pt];
\end{tikzpicture}
}

\newcommand{\gIa}[4]{
\begin{tikzpicture}[baseline=(A),outer sep=0pt,inner sep=0pt]
\node (A) at (0,0) {};
\twoe{0}{#1}
\node[font=\tiny,align=center] at (0.35,0.1) {$\iiE{#2}$};
\thre{0.7}{#3}
\node[font=\tiny,align=center] at (0,-0.5) {$\iE{#4}$};
\draw[line width=0.5pt, draw=black] (0,0)-- (0.0,-0.38);
\draw[line width=0.5pt, draw=black] (0.7,0)-- (0,0);
\draw[black, fill=black] (0,0) circle [radius=1.5pt];
\draw[black, fill=black] (0.7,0) circle [radius=1.5pt];
\end{tikzpicture}
}

\newcommand{\gIb}[6]{
\begin{tikzpicture}[baseline=(A),outer sep=0pt,inner sep=0pt]
\node (A) at (0,0) {};
\twoe{0}{#1}
\node[font=\tiny,align=center] at (0.35,0.1) {$\iiE{#2}$};
\twoe{0.7}{#3}
\node[font=\tiny,align=center] at (1.05,0.1) {$\iiE{#4}$};
\thre{1.4}{#5}
\node[font=\tiny,align=center] at (0,-0.5) {$\iE{#6}$};
\draw[line width=0.5pt, draw=black] (0,0)-- (0.0,-0.38);
\draw[line width=0.5pt, draw=black] (1.4,0)-- (0,0);
\draw[black, fill=black] (0,0) circle [radius=1.5pt];
\draw[black, fill=black] (0.7,0) circle [radius=1.5pt];
\draw[black, fill=black] (1.4,0) circle [radius=1.5pt];
\end{tikzpicture}
}

\newcommand{\gIA}[6]{
\begin{tikzpicture}[baseline=(A),outer sep=0pt,inner sep=0pt]
\node (A) at (0,0) {};
\twoe{0}{#1}
\node[font=\tiny,align=center] at (0.35,0.1) {$\iiE{#2}$};
\twoe{0.7}{#3}
\node[font=\tiny,align=center] at (1.05,0.1) {$\iiE{#4}$};
\thre{1.4}{#5}
\node[font=\tiny,align=center] at (0,-0.5) {$\iE{#6}$};
\draw[line width=0.5pt, draw=black] (0,0)-- (0.0,-0.38);
\draw[line width=0.5pt, draw=black] (0.7,0)-- (0,0);
\draw[line width=0.5pt, draw=black, dotted] (1.4,0)-- (0.7,0);
\draw[black, fill=black] (0,0) circle [radius=1.5pt];
\draw[black, fill=black] (0.7,0) circle [radius=1.5pt];
\draw[black, fill=black] (1.4,0) circle [radius=1.5pt];
\end{tikzpicture}
}

\newcommand{\gEq}{
\begin{tikzpicture}[baseline=(A),outer sep=0pt,inner sep=0pt]
\node (A) at (0,0) {$\simeq$};
\end{tikzpicture}
}

\newcommand{\gSu}[2]{
\begin{tikzpicture}[baseline=(A),outer sep=0pt,inner sep=-1pt]
\node (A) at (0,0) {$\sum\limits_{#1}^{#2}$};
\end{tikzpicture}
}

\newcommand{\gS}[1]{
\begin{tikzpicture}[baseline=(A),outer sep=0pt,inner sep=-3pt]
\node (A) at (0,0) {$#1$};
\end{tikzpicture}
}

\newcommand{\gES}[1]{
\begin{tikzpicture}[baseline=(A),outer sep=0pt,inner sep=-3pt]
\node (A) at (0,0) {};
\node  at (0,-0.5) {$#1$};
\end{tikzpicture}
}

\def\sgn{{\rm sgn}}

\def\SL{{\rm SL}}

\def\CC{{\mathbb C}}

\def\NN{{\mathbb N}}

\def\MMM{{\mathcal M}}

\def\eff{{\rm eff}}

\title[Invariants of fundamental representations of $\SL_n$ and hypergraphs]{Invariant rings of sums of fundamental representations of $\SL_n$ and colored hypergraphs}

\author[Lukas Braun]{Lukas Braun}

\address{Mathematisches Institut, Universit\"at T\"ubingen,
Auf der Morgenstelle 10, 72076 T\"ubingen, Germany}
\email{braun@math.uni-tuebingen.de}


\begin{document}

\begin{abstract}
The fundamental representations of the special linear group $\SL_n$ over the complex numbers are the exterior powers of $\CC^n$. We consider the invariant rings of sums of arbitrary many copies of these $\SL_n$-modules. The symbolic method for antisymmetric tensors developed by Grosshans, Rota and Stein is used, but instead of brackets, we associate colored hypergraphs to the invariants. This approach allows us to use results and insights from graph theory. In particular, we determine  (minimal) generating sets of the invariant rings in the case of $\SL_4$ and $\SL_5$, as well as syzygies for $\SL_4$. Since the invariants constitute incidence geometry of linear subspaces of the projective space $\mathbb{P}_{n-1}$, the generating invariants provide (minimal) sets of geometric relations that are able to describe all others.
\end{abstract}

\maketitle

\section{Introduction}

Classical invariant theory deals with invariants of linear reductive groups and their syzygies. From its beginnings in the nineteenth century, it was not only a forerunner for modern invariant theory but, for example through \emph{Hilberts Nullstellensatz} and \emph{Basissatz}, for modern algebraic geometry and algebra in general. The literature on the subject is vast, we refer to~\cite{GW, KP, PO, Weyl} for an overview.

One of the most important fields of research in the nineteenth century was the theory of invariants of \emph{binary forms}, so to say of $\SL_2(\CC)$-invariants of  symmetric tensors. This is still an active area of research, see~\cite{BP1, BP2, MD, LO}. 
Besides  the case of binary forms, among others, invariants and syzygies have been found for \emph{systems of vectors and covectors} of the classical groups, see~\cite[ \S 9.3, 9.4]{PV}. At least the theory of binary forms relies heavily on the \emph{symbolic method}, using \emph{brackets} to denote complete contractions of tensors made up of the relevant symmetric ones and the covariant tensor $\det$. 

Weitzenb\"ock applied the symbolic method to \emph{antisymmetric} tensors in~\cite{Wei1, Wei2}, but it was not until 1987, when Grosshans, Rota, and Stein in~\cite{GRS}, see also~\cite{RS}, formulated a rigorous symbolic method for both symmetric and antisymmetric tensors using \emph{superalgebras}. Nevertheless, besides findings on invariants of $(\BigWedge^2(\CC^n))^1$ and $(\BigWedge^2(\CC^4))^{n_2}$, see~\cite[\S 5.4]{GRS},~\cite[Th. 34.9]{Gur}  and~\cite{crapo, mcmillan, RStu, vazzdiss, vazz, White} as well as invariants of up to four linear subspaces of projective space~\cite{HH, Hu1}, and some statements on invariants of $\BigWedge^3(\CC^n)$ for small $n$, see~\cite[\S 35]{Gur} and~\cite{chan}, the only progress in finding generators for the ring of invariants has been made by Rosa Huang, a student of Rota, in the case of $(\BigWedge^2(\CC^4))^{n_2}$, see~\cite{Hu}. But the system of generators she found was by no means minimal. As vanishing of invariants describes the geometry of linear subspaces of projective space, see~\cite{Stu} and~\cite[\S 11]{dolga} for a discussion, determination of a \textit{minimal} generating set in this context means nothing less than finding a minimal set of geometric relations that are able to describe all the others.

Gurevich~\cite[\S 35]{Gur} as well as Sturmfels~\cite[p. 173]{Stu} and Procesi~\cite[\S 6.8]{Pro} consider this a complicated and involved problem. 

The aim of the present paper is first to develop a method for approaching this problem and second, to demonstrate the power of this method by finding generators of the ring of invariants for $\SL_4$ and $\SL_5$ as well as relations for $\SL_4$. 

We fix some notation. Let $n \geq 2$ be a fixed integer. We work over the field $\CC$ of complex numbers and denote by $\SL_n$ the special linear group of degree $n$ over $\CC$. This group acts on $V=\CC^n$ by multiplication from the left, which is the standard representation. This induces an action of $\SL_n$ on $\BigWedge^i V$ for any $i \in \{1,\ldots,n-1\}$, the fundamental representations. Now set $V_{i,j}:=\BigWedge^i V$ and
$$
W:=  \bigoplus_{i=1}^{n-1} \bigoplus_{j=1}^{n_i} V_{i,j}
$$
for fixed $n_1,\ldots,n_{n-1}\geq 0$. We call the induced action of $\SL_n$ on $W$ the \emph{action on sums of fundamental representations}. The special case $n_i=0$ for $i\neq 1,n-1$ is equivalent to the action on vectors and covectors.

We informally describe the symbolic method from~\cite{GRS} now, ignoring signs. Let $m$, $n_{i,j}$ be nonnegative integers. A bracket monomial is a product of $m$ brackets, where every bracket contains $n$ out of the following letters: to every $V_{i,j}$ associate letters $a_{i,j,k}$, so that for every $1 \leq k \leq n_{i,j}$, the letter  $a_{i,j,k}$ turns up $i$ times in the bracket monomial. Now consider the mapping of an element $\sum t_{i,j}$ of $W$ to 
$$
\left(\bigotimes\limits_{i,j}t_{i,j}^{\otimes n_{i,j}}\right) \otimes {\det}^{\otimes m},
$$
followed by the complete contraction, where two indices of the $k$-th appearance of $t_{i,j}$ and the $l$-th appearance of $\det$ in the tensor are contracted if and only if the letter $a_{i,j,k}$ turns up in the $l$-th bracket. This is the invariant associated to the bracket monomial. 

We have two fundamental statements: first, all \emph{invariants} come from bracket polynomials and second, all \emph{relations} between bracket polynomials come from the \emph{Pl\"ucker relation}
$$
\sum_{(u_1,u_2)\vdash(u)} \left[u_1\right] \left[ u_2 w \right] 
\ = \
0,
$$
where $u$ is a word of length $n+1$, $w$ a word of length $n-1$, and we sum over all partitions of $u$ in two subwords $u_1$ of length $n$ and $u_2$ of length one. In the case $n_i=0$ for $i\neq 1$, this gives the standard Pl\"ucker relations (without sign, which is due to the nature of the involved superalgebras, see Section~\ref{sec:brainv}). 

This sets the starting point for our method. We associate to each bracket monomial a \emph{colored hypergraph} defined as follows: for each bracket, we have a vertex, and for each letter $a_{i,j,k}$, we have an $i$-edge of \emph{color} $j$ and \emph{shading} $k$. We ignore the shading for a moment, as it just affects sign. Now if the letter $a_{i,j,k}$ turns up in a bracket, the respective $i$-edge has a connection to the vertex associated to the bracket. For a similar (somewhat dual) approach involving directed graphs in the case of binary forms see~\cite{Ograph} and also~\cite{olive, PO, PV}. This approach has just recently been applied to determine the ideal of relations of several points on the projective line, see~\cite{HMSV1, HMSV2, HMSV3, HMSV4, kempe}.

At first sight, the problem is much more involved in our case: excluding the already settled cases $n=2,3$, we deal with vertices of degree $n\geq 4$ and $i$-edges with $i\leq n-1$. But it turns out that the Pl\"ucker relations from above can be used to substantially simplify the involved hypergraphs, which in turn allows us to effectively use combinatorial and graph theoretical results. We develop suitable techniques in Sections~\ref{sec:bragra},~\ref{sec:SL4}, and~\ref{sec:SL5}. One of the greatest advantages of the approach is the self-containedness, making it comprehensible for everyone with basic knowledge on combinatorics.

In the case of $\SL_4$, we explicitly list a minimal set of generators of $\CC[W]^{\SL_4}$ in the following theorem. Here and throughout the paper, for the colors of $1$-edges we use typewriter font, for those of $2$-edges bold letters and for $3$-edges Fraktur letters. We denote the standard coordinate functions on $V_{1,j}$ by $x_{\iE{j}1},\ldots,x_{\iE{j}4}$, those on $V_{2,j}$ by $y_{\iiE{j}12},\ldots,y_{\iiE{j}34}$, and on $V_{3,j}$ by $z_{\iiiE{j}123},\ldots,z_{\iiiE{j}234}$. Moreover, to ease notation, colors $1,2,\ldots,N$ of $i$-edges stand representatively for arbitrary but ascending colors $1\leq k_1,k_2,\ldots,k_N\leq n_i$.

\begin{theorem}\label{th:fftSL4}
Let $\SL_4$ act on an arbitrary sum of fundamental representations $W$. Then a minimal generating set of $\CC[W]^{\SL_4}$ is the following, where if the respective invariant is too big, only the number of monomials is given. In these cases, consult the appendix~\cite{app} for the actual invariant.

\setlength{\tabcolsep}{-0.2mm}
\setlength{\arraycolsep}{0.1mm}

\begin{longtable}{ccc}
Graph & Invariant & Symbol
\\
\hline
\\[-9pt]
\gA{1}{2}{3}{4}
&
 $\left| \begin{array}{ccc}
x_{\iE{1}1} & \cdots & x_{\iE{4}1} \\
\vdots &\ddots& \vdots \\
x_{\iE{1}4} & \cdots & x_{\iE{4}4}
\end{array}\right|$
&
$\INV{{}_{\iE{1234}}}$
\vspace{2pt}
\\
\hline
\\[-9pt]
\gB{1}{2}{3}{4}
&
$
\left| \begin{array}{ccc}
z_{\iiiE{1}234} & \cdots & z_{\iiiE{4}234} \\
\vdots &\ddots& \vdots \\
z_{\iiiE{1}123} & \cdots & z_{\iiiE{4}123}
\end{array}\right|
$
&
$\INV{{}^{\iiiE{1234}}}$
\vspace{2pt}
\\
\hline
\\[-9pt]
\gCa{1}
&
$\begin{array}{r}
2(y_{\iiE{1}12}y_{\iiE{1}34}-y_{\iiE{1}13}y_{\iiE{1}24}+y_{\iiE{1}14}y_{\iiE{1}23})
\end{array}$
&
$\INV{\iiE{11}}$
\vspace{2pt}
\\
\hline
\\[-9pt]
\gCb{1}{2}
&
$\begin{array}{lr}
&y_{\iiE{1}12}y_{\iiE{2}34}-y_{\iiE{1}13}y_{\iiE{2}24}+y_{\iiE{1}14}y_{\iiE{2}23} \\
+ & y_{\iiE{1}34}y_{\iiE{2}12}-y_{\iiE{1}24}y_{\iiE{2}13}+y_{\iiE{1}23}y_{\iiE{2}14}
\end{array}$
&
$\INV{\iiE{12}}$
\vspace{2pt}
\\
\hline
\\[-9pt]
\gD{1}{2}{3}{4}{5}{6}%
\hspace{-1mm}%
\gS{-}%
\hspace{-1mm}%
\gD{6}{1}{2}{3}{4}{5}
&
$
\left| \begin{array}{ccc}
y_{\iiE{1}12} & \cdots & y_{\iiE{6}12} \\
\vdots &\ddots& \vdots \\
y_{\iiE{1}34} & \cdots & y_{\iiE{6}34}
\end{array}\right|
$
&
$\INV{\iiE{123456}}$
\vspace{2pt}
\\
\hline
\\[-9pt]
\gE{1}{1}
&
$
x_{\iE{1}1}z_{\iiiE{1}234}-x_{\iE{1}2}z_{\iiiE{1}134}+x_{\iE{1}3}z_{\iiiE{1}124}-x_{\iE{1}4}z_{\iiiE{1}123}
$
&
$\INV{{}_{\iE{1}}{}^{\iiiE{1}}}$
\vspace{2pt}
\\
\hline
\\[-9pt]
\gF{1}{1}{2}
&
$\begin{array}{lr}
&y_{\iiE{1}12}(x_{\iE{1}3}x_{\iE{2}4}\!-\!x_{\iE{1}4}x_{\iE{2}3})+y_{\iiE{1}13}(x_{\iE{1}4}x_{\iE{2}2}\!-\!x_{\iE{1}2}x_{\iE{2}4})\\
+& y_{\iiE{1}14}(x_{\iE{1}2}x_{\iE{2}3}\!-\!x_{\iE{1}3}x_{\iE{2}2})+ y_{\iiE{1}23}(x_{\iE{1}1}x_{\iE{2}4}\!-\!x_{\iE{1}4}x_{\iE{2}1}) \\
+ & y_{\iiE{1}24}(x_{\iE{1}3}x_{\iE{2}1}\!-\!x_{\iE{1}1}x_{\iE{2}3})+y_{\iiE{1}34}(x_{\iE{1}1}x_{\iE{2}2}\!-\!x_{\iE{1}2}x_{\iE{2}1})
\end{array}$
&
$\INV{{}_{\iE{1}}\iiE{1}_{\iE{2}}}$
\vspace{2pt}
\\
\hline
\\[-9pt]
\gGa{1}{2}{3}{1}
&
$
96
$
&
$\INV{{}_{\iE{1}}\iiE{123}_{\iE{1}}}$
\vspace{2pt}
\\
\hline
\\[-9pt]
\gGb{1}{2}{3}{1}{2}
&
$
108
$
&
$\INV{{}_{\iE{1}}\iiE{123}_{\iE{2}}}$
\vspace{2pt}
\\
\hline
\\[-9pt]
\gGc{1}{2}{3}{4}{5}{1}{2}
&
$
972
$
&
$\INV{{}_{\iE{1}}\iiE{12345}_{\iE{2}}}$
\vspace{2pt}
\\
\hline
\\[-9pt]
\gHa{1}{1}{2}
&
$12$
&
$\INV{{}^{\iiiE{1}}\iiE{1}^{\iiiE{2}}}$
\vspace{2pt}
\\
\hline
\\[-9pt]
\gHb{1}{2}{3}{1}
&
$
96
$
&
$\INV{{}^{\iiiE{1}}\iiE{123}^{\iiiE{1}}}$
\vspace{2pt}
\\
\hline
\\[-9pt]
\gHc{1}{2}{3}{1}{2}
&
$
108
$
&
$\INV{{}^{\iiiE{1}}\iiE{123}^{\iiiE{2}}}$
\vspace{2pt}
\\
\hline
\\[-9pt]
\gHd{1}{2}{3}{4}{5}{1}{2}
&
$
972
$
&
$\INV{{}^{\iiiE{1}}\iiE{12345}^{\iiiE{2}}}$
\vspace{2pt}
\\
\hline
\\[-9pt]
\gIa{1}{2}{1}{1}
&
$
36
$
&
$\INV{{}_{\iE{1}}\iiE{12}^{\iiiE{1}}}$
\vspace{2pt}
\\
\hline
\\[-9pt]
\gIb{1}{2}{3}{4}{1}{1}
&
$
324
$
&
$\INV{{}_{\iE{1}}\iiE{1234}^{\iiiE{1}}}$

\end{longtable}

\end{theorem}

\begin{corollary}
All geometric incidence relations between $n_1$ points, $n_2$ lines, and $n_3$ planes in $\mathbb{P}_3(\CC)$ can be expressed by means of the invariants from Theorem~\ref{th:fftSL4}.
\end{corollary}

It follows a generating set for the respective invariants of $\SL_5$. Here the number of graphs coming into consideration turns out to be considerably larger as in the case of $\SL_4$. Thus in the following theorem, we do not give all different colorings of each graph. We did this exemplarily in the special case $n_i=0$ for $i\neq 2$ (for graphs only with $2$-edges in other words), see Proposition~\ref{prop:SL52} at the beginning of Section~\ref{sec:SL5}. 

Moreover, by the duality of $\BigWedge^iV$ and $\BigWedge^{n-i}V$, we have \emph{mirror invariants} and also mirror graphs, where $i$-edges of the one correspond to $(n-i)$-edges of the other. For example, the first two graphs from Theorem~\ref{th:fftSL4} are mirrors, while graphs three to six are their own mirror each. On the invariant side, it is straightforward to get the mirror by replacing $x_{\iE{j}a}$ by $z_{\iiiE{j}bcd}$ and vice versa, as well as $y_{\iiE{j}ab}$ by $y_{\iiE{j}cd}$ - where $\{a,b,c,d\}=\{1,2,3,4\}$ - in the case of $\SL_4$, and analogously for greater $n$. Thus we consider only one graph of each mirror pair in the following theorem. 
Lastly, there are the following types of 'building blocks' that can be attached to some of the graphs:
\begin{center}

\end{center}

\begin{theorem}\label{th:fftSL5}
The following graphs with all possible combinations of building blocks attached and all possible colorings with respect to the conditions:
\begin{itemize}
\item the number of vertices with at least one looping $2$-edge plus the number of non-looping two edges is less than or equal to nine,
\item the number of blocks 
%
plus the number of $3$-edges that are not part of a block 
%
 is less than or equal to nine,
\item the number of $3$-edges is less than or equal to the number of $2$-edges,
\end{itemize}
together with their mirrors constitute a generating set of  $\CC[W]^{\SL_5}$.
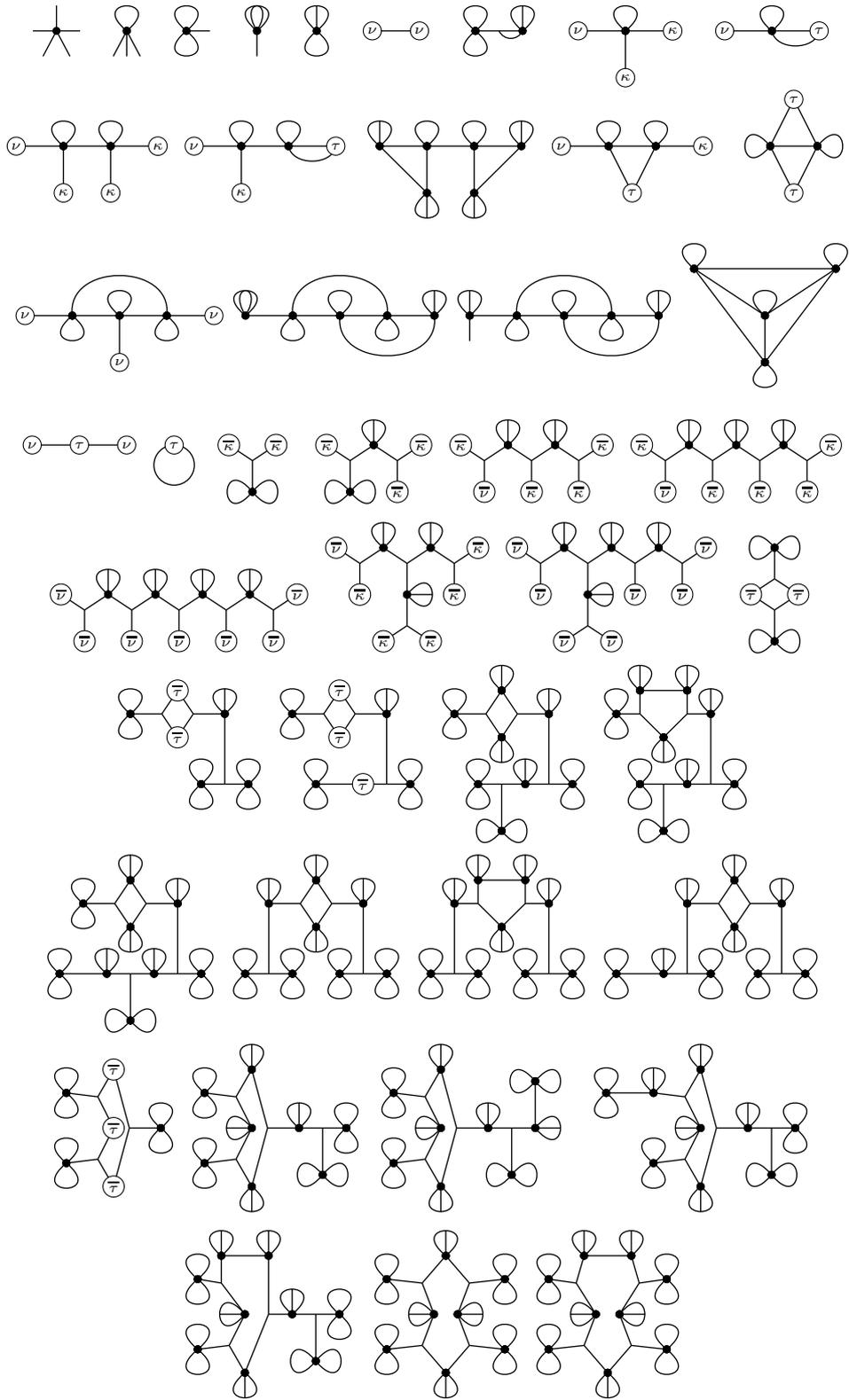
\begin{center}
%
\end{center}
\end{theorem}

\begin{corollary}
All geometric incidence relations between $n_1$ points, $n_2$ lines, $n_3$ planes, and $n_4$ hyperplanes in $\mathbb{P}_4(\CC)$ can be expressed by means of the invariants from Theorem~\ref{th:fftSL5}.
\end{corollary}

Finally, we present relations that hold between the generators of $\CC[W]^{\SL_4}$. Most of them can be generalized to $n\geq 5$. The notation 
$$
(u_1, \ldots ,u_M)\vdash(u)
$$ 
means that we sum over all partitions of the word $u$ in $M$ subwords $u_k=l_{k,1} \cdots l_{k,N_k}$ of the lengths $N_k$. If we have an ordering on the letters of $u$, then we require $u$ and all subwords to be ordered and for every summand define $\sgn(\vdash)$ to be the sign of the underlying permutation of letters $u \mapsto u_1 \cdots u_M$.

\begin{theorem}\label{th:SL4rel}
The following sums of graphs correspond to polynomials in the ideal of relations of $\CC[W]^{\SL_4}$, where the not necessarily connected (sub-)graphs $\Gamma$ and $\Gamma_i$ have to be chosen such that all involved graphs are out of the minimal generating set from Theorem~\ref{th:fftSL4}. 
\begin{longtable}{rl}
\gS{\Upsilon_1=}%
&
\gSu{(\iE{ijkl,m})\vdash(\iE{12345})}{}\gA{i}{j}{k}{l} 
\begin{tikzpicture}[baseline=(A),outer sep=0pt,inner sep=0pt]
\node (A) at (0,0) {};
\onee{0}{m}
\draw[line width=0.5pt, draw=black] (0,0.38)-- (0,0);
\draw[black, fill=black] (0,0) circle [radius=1.5pt];
\gra{0}
\end{tikzpicture}%
\gES{,}
\quad
\gS{\Upsilon_2=}%
\gA{1}{2}{3}{4}%
\begin{tikzpicture}[baseline=(A),outer sep=0pt,inner sep=0pt]
\twoed{0}{2}
\twoe{0}{1}
\draw[black, fill=black] (0,0) circle [radius=1.5pt];
\end{tikzpicture}
\gS{-}%
\gSu{(\iE{ij,kl})\vdash(\iE{1234})}{}%
\gF{1}{i}{j}%
\gF{2}{k}{l}%
\gES{,}
\vspace{7pt} \\
\gS{\Upsilon_3=}%
&
\gA{1}{2}{3}{4}%
\begin{tikzpicture}[baseline=(A),outer sep=0pt,inner sep=0pt]
\node (A) at (0,0) {};
\draw[line width=0.5pt, draw=black] (0.7,0)-- (0,0);
\draw[line width=0.5pt, draw=black] (0.7,0)-- (0.35,-0.4);
\draw[line width=0.5pt, draw=black] (0.35,-0.4)-- (0,0);
\draw[black, fill=black] (0,0) circle [radius=1.5pt];
\draw[black, fill=black] (0.7,0) circle [radius=1.5pt];
\node[font=\tiny,align=center] at (0.35,0.1) {$\iiE{1}$};

\draw[line width=0.5pt, draw=black] (0.7,0)-- (0.35,-0.3);
\draw[line width=0.5pt, draw=black] (0.7,0)-- (0.35,-0.6);
\draw[line width=0.5pt, draw=black] (0.35,-0.3)-- (0,0);
\draw[line width=0.5pt, draw=black] (0.35,-0.6)-- (0,0);

\node[circle, fill=white,draw,inner sep=1pt]  at (0.35,-0.4) {\small{$\Gamma$}};
\end{tikzpicture}
\gS{-}%
\gSu{(\iE{ij,k,l})\vdash(\iE{1234})}{}%
\gF{1}{i}{j}%
\begin{tikzpicture}[baseline=(A),outer sep=0pt,inner sep=0pt]
\oneeu{0}{k}
\oneeu{0.7}{l}
\draw[line width=0.5pt, draw=black] (0.7,0)-- (0.35,-0.4);
\draw[line width=0.5pt, draw=black] (0,0)-- (0.35,-0.4);
\draw[black, fill=black] (0,0) circle [radius=1.5pt];
\draw[black, fill=black] (0.7,0) circle [radius=1.5pt];

\draw[line width=0.5pt, draw=black] (0.7,0)-- (0.35,-0.3);
\draw[line width=0.5pt, draw=black] (0.35,-0.3)-- (0,0);
\draw[line width=0.5pt, draw=black] (0.7,0)-- (0.35,-0.6);
\draw[line width=0.5pt, draw=black] (0.35,-0.6)-- (0,0);

\node[circle, fill=white,draw,inner sep=1pt]  at (0.35,-0.4) {\small{$\Gamma$}};
\end{tikzpicture}%
\gES{,}
\vspace{7pt} \\
\gS{\Upsilon_4=}%
&
\gA{1}{2}{3}{4}%
\begin{tikzpicture}[baseline=(A),outer sep=0pt,inner sep=0pt]
\node (A) at (0,0) {};
\twoe{0}{1}
\draw[line width=0.5pt, draw=black] (0.7,0)-- (0,0);
\draw[line width=0.5pt, draw=black] (0.7,0)-- (0.35,-0.4);
\draw[line width=0.5pt, draw=black] (0.35,-0.4)-- (0,0);
\draw[black, fill=black] (0,0) circle [radius=1.5pt];
\draw[black, fill=black] (0.7,0) circle [radius=1.5pt];
\node[font=\tiny,align=center] at (0.35,0.1) {$\iiE{2}$};

\draw[line width=0.5pt, draw=black] (0.7,0)-- (0.35,-0.3);
\draw[line width=0.5pt, draw=black] (0.7,0)-- (0.35,-0.6);

\node[circle, fill=white,draw,inner sep=1pt]  at (0.35,-0.4) {\small{$\Gamma$}};
\end{tikzpicture}
\gS{+}%
\gSu{(\iE{ij,k,l})\vdash(\iE{1234})}{}%
\gF{1}{i}{j}%
\begin{tikzpicture}[baseline=(A),outer sep=0pt,inner sep=0pt]
\onee{0}{k}
\oneeu{0.7}{l}
\twoe{0}{2}
\draw[line width=0.5pt, draw=black] (0.7,0)--(0.35,-0.4);
\draw[line width=0.5pt, draw=black] (0,0)-- (0.35,-0.4);
\draw[black, fill=black] (0,0) circle [radius=1.5pt];
\draw[black, fill=black] (0.7,0) circle [radius=1.5pt];

\draw[line width=0.5pt, draw=black] (0.7,0)-- (0.35,-0.3);
\draw[line width=0.5pt, draw=black] (0.7,0)-- (0.35,-0.6);

\node[circle, fill=white,draw,inner sep=1pt] at (0.35,-0.4) {\small{$\Gamma$}};
\end{tikzpicture}
\gES{,}
\vspace{7pt} \\
\gS{\Upsilon_5=}%
&
\gA{1}{2}{3}{4}%
\begin{tikzpicture}[baseline=(A),outer sep=0pt,inner sep=0pt]
\thre{0}{1}
\draw[line width=0.5pt, draw=black] (0.7,0)-- (0,0);
\draw[black, fill=black] (0,0) circle [radius=1.5pt];
\draw[black, fill=black] (0.7,0) circle [radius=1.5pt];
\node[font=\tiny,align=center] at (0.35,0.1) {$\iiE{1}$};
\gra{0.7}
\end{tikzpicture}%
\gS{-}%
\gSu{(\iE{i,jk,l})\vdash(\iE{1234})}{}%
\gE{1}{i}%
\gF{1}{j}{k}%
\begin{tikzpicture}[baseline=(A),outer sep=0pt,inner sep=0pt]
\node (A) at (0,0) {};
\onee{0}{l}
\draw[black, fill=black] (0,0) circle [radius=1.5pt];
\draw[line width=0.5pt, draw=black] (0,0)-- (0,0.38);
\gra{0}
\end{tikzpicture}%
\gES{,}
\vspace{7pt} \\
\gS{\Upsilon_6=}%
&
\gSu{(\iE{i},\iE{j})\vdash(\iE{12})}{}%
\gCb{1}{2}
\begin{tikzpicture}[baseline=(A),outer sep=0pt,inner sep=0pt]
\node (A) at (0,0) {};
\onee{0}{i}
\draw[line width=0.5pt, draw=black] (0,0.38)-- (0,0);
\draw[black, fill=black] (0,0) circle [radius=1.5pt];
\grb{0}{\Gamma_1}
\end{tikzpicture}
\begin{tikzpicture}[baseline=(A),outer sep=0pt,inner sep=0pt]
\node (A) at (0,0) {};
\onee{0}{j}
\draw[line width=0.5pt, draw=black] (0,0.38)-- (0,0);
\draw[black, fill=black] (0,0) circle [radius=1.5pt];
\grb{0}{\Gamma_2}
\end{tikzpicture}%
\vspace{4pt} \\
&
\gS{-}%
\gSu{(i,j)\vdash(12)}{}%
\gS{\sgn(\vdash)\Biggl( }%
\begin{tikzpicture}[baseline=(A),outer sep=0pt,inner sep=0pt]
\node (A) at (0,0) {};
\onee{0}{1}
\draw[line width=0.5pt, draw=black] (0,0.38)-- (0,0);
\draw[black, fill=black] (0,0) circle [radius=1.5pt];
\grb{0}{\Gamma_i}
\end{tikzpicture}
\begin{tikzpicture}[baseline=(A),outer sep=0pt,inner sep=0pt]
\node (A) at (0,0) {};
\twoe{0}{1}
\onee{0}{2}
\draw[line width=0.5pt, draw=black] (0.7,0)-- (0,0);
\draw[line width=0.5pt, draw=black] (0.7,0)-- (0.7,0.38);
\draw[black, fill=black] (0,0) circle [radius=1.5pt];
\draw[black, fill=black] (0.7,0) circle [radius=1.5pt];
\node[font=\tiny,align=center] at (0.35,0.1) {$\iiE{2}$};
\grb{0.7}{\Gamma_j}
\end{tikzpicture}%
\gS{+}
\begin{tikzpicture}[baseline=(A),outer sep=0pt,inner sep=0pt]
\node (A) at (0,0) {};
\onee{0}{2}
\draw[line width=0.5pt, draw=black] (0,0.38)-- (0,0);
\draw[black, fill=black] (0,0) circle [radius=1.5pt];
\grb{0}{\Gamma_i}
\end{tikzpicture}
\begin{tikzpicture}[baseline=(A),outer sep=0pt,inner sep=0pt]
\node (A) at (0,0) {};
\twoe{0}{1}
\onee{0}{1}
\draw[line width=0.5pt, draw=black] (0.7,0)-- (0,0);
\draw[line width=0.5pt, draw=black] (0.7,0)-- (0.7,0.38);
\draw[black, fill=black] (0,0) circle [radius=1.5pt];
\draw[black, fill=black] (0.7,0) circle [radius=1.5pt];
\node[font=\tiny,align=center] at (0.35,0.1) {$\iiE{2}$};
\grb{0.7}{\Gamma_j}
\end{tikzpicture}%
\gS{+}%
\gF{i}{1}{2}%
\begin{tikzpicture}[baseline=(A),outer sep=0pt,inner sep=0pt]
\node (A) at (0,0) {};
\draw[line width=0.5pt, draw=black] (0.7,0)-- (0,0);
\draw[black, fill=black] (0,0) circle [radius=1.5pt];
\draw[black, fill=black] (0.7,0) circle [radius=1.5pt];
\node[font=\tiny,align=center] at (0.35,0.1) {$\iiE{j}$};
\draw[line width=0.5pt, draw=black] (0.7,0)-- (0.7,0.38);
\draw[line width=0.5pt, draw=black] (0,0)-- (0,0.38);
\grb{0}{\Gamma_1}
\grb{0.7}{\Gamma_2}
\end{tikzpicture}%
\gS{\Biggr)}
\gES{,}
\vspace{7pt} \\
 \gS{\Upsilon_7=}%
 &
\gSu{(\iE{i,j,k})\vdash(\iE{123})}{}%
\begin{tikzpicture}[baseline=(A),outer sep=0pt,inner sep=0pt]
\node (A) at (0,0) {};
\onee{0}{i}
\draw[line width=0.5pt, draw=black] (0,0.38)-- (0,0);
\draw[black, fill=black] (0,0) circle [radius=1.5pt];
\grb{0}{\Gamma_1}
\end{tikzpicture}
\begin{tikzpicture}[baseline=(A),outer sep=0pt,inner sep=0pt]
\node (A) at (0,0) {};
\onee{0}{j}
\draw[line width=0.5pt, draw=black] (0,0.38)-- (0,0);
\draw[black, fill=black] (0,0) circle [radius=1.5pt];
\grb{0}{\Gamma_2}
\end{tikzpicture}
\begin{tikzpicture}[baseline=(A),outer sep=0pt,inner sep=0pt]
\node (A) at (0,0) {};
\twoe{0}{1}
\onee{0}{k}
\draw[line width=0.5pt, draw=black] (0.7,0)-- (0,0);
\draw[line width=0.5pt, draw=black] (0.7,0)-- (0.7,0.38);
\draw[black, fill=black] (0,0) circle [radius=1.5pt];
\draw[black, fill=black] (0.7,0) circle [radius=1.5pt];
\node[font=\tiny,align=center] at (0.35,0.1) {$\iiE{2}$};
\grb{0.7}{\Gamma_3}
\end{tikzpicture}%
\vspace{4pt} \\
&
\gS{+}%
\gSu{(\iE{i,jk})\vdash(\iE{123})}{}%
\gS{\Biggl(}%
\begin{tikzpicture}[baseline=(A),outer sep=0pt,inner sep=0pt]
\node (A) at (0,0) {};
\onee{0}{i}
\draw[line width=0.5pt, draw=black] (0,0.38)-- (0,0);
\draw[black, fill=black] (0,0) circle [radius=1.5pt];
\grb{0}{\Gamma_1}
\end{tikzpicture}
\gF{1}{j}{k}%
\begin{tikzpicture}[baseline=(A),outer sep=0pt,inner sep=0pt]
\node (A) at (0,0) {};

\draw[line width=0.5pt, draw=black] (0.7,0)-- (0,0);
\draw[black, fill=black] (0,0) circle [radius=1.5pt];
\draw[black, fill=black] (0.7,0) circle [radius=1.5pt];
\node[font=\tiny,align=center] at (0.35,0.1) {$\iiE{2}$};
\draw[line width=0.5pt, draw=black] (0.7,0)-- (0.7,0.38);
\draw[line width=0.5pt, draw=black] (0,0)-- (0,0.38);
\grb{0}{\Gamma_2}
\grb{0.7}{\Gamma_3}
\end{tikzpicture}%
\gS{+}%
\begin{tikzpicture}[baseline=(A),outer sep=0pt,inner sep=0pt]
\node (A) at (0,0) {};
\onee{0}{i}
\draw[line width=0.5pt, draw=black] (0,0.38)-- (0,0);
\draw[black, fill=black] (0,0) circle [radius=1.5pt];
\grb{0}{\Gamma_3}
\end{tikzpicture}
\gF{2}{j}{k}%
\begin{tikzpicture}[baseline=(A),outer sep=0pt,inner sep=0pt]
\node (A) at (0,0) {};

\draw[line width=0.5pt, draw=black] (0.7,0)-- (0,0);
\draw[black, fill=black] (0,0) circle [radius=1.5pt];
\draw[black, fill=black] (0.7,0) circle [radius=1.5pt];
\node[font=\tiny,align=center] at (0.35,0.1) {$\iiE{1}$};
\draw[line width=0.5pt, draw=black] (0.7,0)-- (0.7,0.38);
\draw[line width=0.5pt, draw=black] (0,0)-- (0,0.38);
\grb{0}{\Gamma_1}
\grb{0.7}{\Gamma_2}
\end{tikzpicture}
\gS{+}
\begin{tikzpicture}[baseline=(A),outer sep=0pt,inner sep=0pt]
\node (A) at (0,0) {};
\onee{0}{i}
\draw[line width=0.5pt, draw=black] (0,0.38)-- (0,0);
\draw[black, fill=black] (0,0) circle [radius=1.5pt];
\grb{0}{\Gamma_2}
\end{tikzpicture}
\gF{1}{j}{k}%
\begin{tikzpicture}[baseline=(A),outer sep=0pt,inner sep=0pt]
\node (A) at (0,0) {};

\draw[line width=0.5pt, draw=black] (0.7,0)-- (0,0);
\draw[black, fill=black] (0,0) circle [radius=1.5pt];
\draw[black, fill=black] (0.7,0) circle [radius=1.5pt];
\node[font=\tiny,align=center] at (0.35,0.1) {$\iiE{2}$};
\draw[line width=0.5pt, draw=black] (0.7,0)-- (0.7,0.38);
\draw[line width=0.5pt, draw=black] (0,0)-- (0,0.38);
\grb{0}{\Gamma_1}
\grb{0.7}{\Gamma_3}
\end{tikzpicture}%
\gS{\Biggr)}%
\gES{,}
\vspace{7pt} \\
\gS{\Upsilon_8=}%
&
\gS{2}
\gCb{1}{2}%
\gS{\Biggl(}%
\gD{3}{4}{5}{6}{7}{8}%
\gS{-}%
\gD{8}{3}{4}{5}{6}{7}%
\gS{\Biggr)}%
\vspace{4pt} \\
&
\gS{+}
\gSu{(\iiE{g},\iiE{h})\vdash(\iiE{12})}{}%
\gSu{(\iiE{i},\iiE{jklmn})\vdash(\iiE{345678})}{}%
\gS{(-1)^{i}}%
\gCb{g}{i}%
\gS{\Biggl(}%
\gD{h}{j}{k}{l}{m}{n}%
\gS{-}%
\gD{n}{h}{j}{k}{l}{m}%
\gS{\Biggr)}%
\vspace{4pt} \\
&
\gS{-}%
\gSu{(\iiE{i},\iiE{jk},\iiE{l})\vdash(\iiE{3456})}{}%
\gS{\sgn(\vdash)\Biggl(}%
\gCb{1}{8}%
\gCb{2}{i}%
\gCb{j}{k}%
\gCb{l}{7}%
\gS{+}%
\gCb{2}{8}%
\gCb{1}{i}%
\gCb{j}{k}%
\gCb{l}{7}%
\gS{\Biggr)}%
\gES{,}
\vspace{7pt} \\
\gS{\Upsilon_9=}
&
\gD{i_1}{i_2}{i_3}{i_4}{i_5}{i_6}%
\gD{j_1}{j_2}{j_3}{j_4}{j_5}{j_6}%
\gS{+\det\Biggl(}%
\gCb{i_r}{j_s}
\gS{\Biggr)_{1\leq r,s\leq 6}}
\vspace{7pt} \\
\gS{\Upsilon_{10}=}
&
\gS{\sum\limits_{(i_1\cdots i_6,j_1)\vdash(k_1\ldots k_7)}}
\gD{i_1}{i_2}{i_3}{i_4}{i_5}{i_6}%
\gD{j_1}{j_2}{j_3}{j_4}{j_5}{j_6}
\gES{.}
\end{longtable}
\end{theorem}

These relations emerge from somewhat natural principles that are discussed in the proof of Theorem~\ref{th:SL4rel} in Section~\ref{sec:SL4rel}. This fact together with calculations of the Hilbert series for small values of $n_i$ using~\cite[\S 4.6]{KD} and Xin's  algorithm~\cite{Xin} for MacMahon partition analysis leads us to our final conjecture:

\begin{conjecture}
The graphsums from Theorem~\ref{th:SL4rel} generate the ideal of relations of $\CC[W]^{\SL_4}$.
\end{conjecture}

The paper is organized as follows: in Section~\ref{sec:brainv}, we give the necessary background on the symbolic method from~\cite{GRS}, then in Section~\ref{sec:bragra} we introduce (edge-)colored hypergraphs and give some general statements on the behaviour of these under relations, and restrictions concerning colorings. In Sections~\ref{sec:SL4}  and~\ref{sec:SL5}, we develop techniques in order to find the generating sets of $\CC[W]^{\SL_4}$ and $\CC[W]^{\SL_5}$ of Theorems~\ref{th:fftSL4} and \ref{th:fftSL5} respectively. Most of these techniques are applicable for general $n$. In Section~\ref{sec:SL4rel}, we prove Theorem~\ref{th:SL4rel}. The relation-generating principles are applicable for $n\geq5$ as well.
Finally, Section~\ref{sec:outlook} gives an outlook on possible further applications of the hypergraph method.

\tableofcontents

\section{Invariants and brackets}\label{sec:brainv}

This section is merely a summary of the parts of~\cite{GRS} that are relevant for antisymmetric tensors.
All notation is as close as possible to the one from~\cite{GRS}.
Fix a natural number $n$, a complex vector space $V$ of dimension $n$ and a basis $e_1,\ldots,e_n$ of $V$. Now let the special linear group $\SL_n$ act on $V$ by multiplication from the left. This induces an action of $\SL_n$ on $\BigWedge^iV$ for every $1 \leq i \leq n-1$. Fix some integer $n_i \geq 0$ for every such $i$ and set $V_{i,j}:=\BigWedge^iV$ for every $1 \leq i \leq n-1$ and $ 1 \leq j \leq n_i$. Then the action of $\SL_n$ on $V$ finally induces an action on 
\begin{align*}
W:= W_{(n_1,\ldots,n_{n-1})} =& \bigoplus_{i=1}^{n-1} \bigoplus_{j=1}^{n_i} V_{i,j} 
= \bigoplus_{i=1}^{n-1}\left(\bigwedge\nolimits^iV\right)^{n_i}.
\end{align*}

We can identify the ring of polynomial functions on $W$ with $\CC\left[T_{i,j,\iota_1\cdots\iota_{i}}\right]$, where   $1 \leq j \leq n_i$, $\{\iota_1<\ldots<\iota_{i}\} \subseteq \{1,\ldots,n\} $. We do this by linearly mapping $T_{i,j,\iota_1\cdots\iota_{i}}$ onto a function $f$ so that for an element
$$
t=\sum_{i=1}^{n-1} \sum_{j=1}^{n_i} \sum_{\iota_1 < \ldots < \iota_{\nu_i}} t_{i,j,\iota_1\cdots\iota_{i}} e_{\iota_1} \wedge \ldots \wedge e_{\iota_{i}}
$$
of $W$ we have $f(t)=t_{i,j\cdots\iota_{i}}$. 

Now following~\cite{GRS} we introduce an ordered alphabet $P=\{1,\ldots,n\}$ with $1<\ldots<n$ of  so called \emph{places} and the algebra ${\rm Ext}(P)$, which is the exterior algebra generated by the places. We denote multiplication in ${\rm Ext}(P)$ by juxtaposition.
Moreover, for every $V_{i,j}$, we introduce an infinite number of so called \emph{letters} $a_{i,j,k}$ for all $k \in \mathbb{N}$ forming the alphabet $L$. We set $a_{i_1,j_1,k_1} < a_{i_2,j_2,k_2}$ if either $i_1 < i_2$ or  $i_1=i_2$ and $j_1<j_2$ or $i_1=i_2$ and $j_1=j_2$ and $k_1<k_2$.

\begin{definition}
Let $A$ be an alphabet, then the \emph{divided powers algebra} ${\rm Div}(A)$ is the commutative algebra generated by symbols $a^{(i)}$, where $a \in A$ and $i \in \NN$. We denote multiplication in ${\rm Div}(A)$ by juxtaposition. Moreover, we set $a^{(0)}=1$ and $a^{(1)}=a$ and impose the identity
$$
a^{(i)}a^{(j)} = \binom{i+j}{j} a^{(i+j)}.
$$
We define the length of the word $a^{(i)}$ to be $|a^{(i)}|=i$.
\end{definition}

Now we proceed with the divided powers algebra ${\rm Div}(L)$ generated by the alphabet $L$ of letters and define a third alphabet $\left[ L | P \right]$, the \emph{letterplace alphabet} having as elements pairs $(x | \alpha )$, where $x\in L$, $\alpha \in P$. The algebra ${\rm Ext}(\left[ L | P \right])$ is called the \emph{fourfold algebra}.

\begin{definition}
We define a bilinear form
$$
(*|*): {\rm Div}(L) \times {\rm Ext}(P) \to {\rm Ext}(\left[ L | P \right]),
$$
called the \emph{biproduct}, by the following:
\begin{enumerate}
\item $(w | v )=0$ if $w$ and $v$ are words of different length.
\item $(w | v )= ( x | \alpha)$ if $w=x$ is a letter and $v=\alpha$ is a place, thus the image is a single letterplace.
\item $(1 | 1)=1$.
\item $ (w | v u ) = \sum_{w_1 w_2 = w} \left( w_1 | v \right) \left( w_2 | u \right)$, where the sum ranges over all pairs $w_1, w_2$ of subwords of $w$ such that $w_1w_2=w$.  
\item $ (v u | w) = \sum_{w_1 w_2 = \pm w} (-1)^{\delta(w_1,w_2)} \left(v | w_1 \right) \left( u | w_2 \right)$, where $\delta(w_1,w_2)$ is the number of transpositions needed to obtain the word $w$ from the word $w_1w_2$. 
\end{enumerate}
\end{definition}

We give some examples to clarify these rules.

\begin{example}
Let $a \in L$. We have a look at the image of  $(a^{(2)},12)$ under the biproduct. First we want to use Rule $(iv)$ with $v=1$, $u=2$. Due to Rule $(i)$ we only have to consider pairs of subwords of length one. We have $a^{(2)}=\frac{1}{2}aa$ and we have \emph{two} pairs of possible subwords of length one, since we have to distinguish the two $a$'s. Thus we get
$$
(a^{(2)}|12)=\frac{1}{2}(aa|12)=\frac{1}{2}((a|1)(a|2)+(a|1)(a|2))=(a|1)(a|2).
$$   
Using Rule $(v)$ instead, we compute
$$
(a^{(2)}|12)=\frac{1}{2}(aa|12)=\frac{1}{2}((a|1)(a|2)-(a|2)(a|1))=(a|1)(a|2).
$$
More generally, for arbitrary $i,j,k,l$, we get
$$
\left( a_{i,j,k}^{(l)} | i_1\ldots i_l \right) = \left( a_{i,j,k}| i_1 \right) \cdots \left( a_{i,j,k} | i_l \right).
$$
For different letters $a, b \in L$, we compute
$$
(ab|12)=(a|1)(b|2)+(b|1)(a|2)=(ba|12).
$$ 
\end{example}

Now for letters $a_1,\ldots, a_n$, we define the \emph{bracket} in $a_1,\ldots, a_n$ to be the element
$$
[a_1 \ldots a_n]:=\left(a_1 \cdots a_n | 1 \cdots n \right)
$$
of ${\rm Ext}(\left[ L | P \right])$.
A \emph{bracket monomial} is a product of brackets and a \emph{bracket polynomial} is a linear combination of bracket monomials. We denote the subalgebra of all bracket polynomials of ${\rm Ext}(\left[ L | P \right])$ by ${\rm Br}(L)$.

\begin{lemma}[\cite{GRS}, 'Exchange Lemma', p. 60]
\label{le:exch}
Let $u,v,w$ be words in $\mathrm{Div}(L)$. Then
$$
\sum_{u_1u_2=u} \left[u_1v\right] \left[ u_2 w \right] 
\ = \
(-1)^{n-|w|} \sum_{v_1v_2=v} \left[v_1u\right] \left[ v_2 w \right].
$$
\end{lemma}

\begin{proposition}\label{prop:allid}
All identities among bracket polynomials can be deduced from the identity of Lemma~\ref{le:exch} with $|u|=2$, $|w|=n-1$.
\end{proposition}

\begin{proof}
The fact that all identities can be deduced from Lemma~\ref{le:exch} follows directly from Theorem~8 of~\cite{GRS}. Thus - as was stated in~\cite{GRS} on page $\mathrm{xv}$ - it can be used for an abstract definition of (skew) brackets. The fact that all those identities stem from the ones with ${\rm Length}(u)=2$ is clear.
\end{proof}

\begin{remark}
In Proposition~\ref{prop:allid}, one can replace '$|u|=2$, $|w|=n-1$' with '$|u|=n+1$, $|w|=0$'.
\end{remark}

Finally, we bring together brackets and invariants of the action of $\SL_n$ on $W$ by the following linear map.

\begin{definition}\label{def:umbralop}
Let the linear \emph{umbral operator} 
\begin{align*}
U: {\rm Ext}([L|P]) \to & \CC[W] \\
f \mapsto & \left\langle U,f \right\rangle
\end{align*}
be defined by the following:
\begin{align*}
{\rm (i)~} & \left\langle U, \left( a_{i,j,k}^{(i)} | \iota_1\ldots \iota_i \right) \right\rangle = T_{i,j,\iota_1,\ldots,\iota_{i}},  \\
{\rm (ii)~} & \left\langle U, \left( a_{i,j,k}^{(l)} | \iota_1\ldots \iota_l \right) \right\rangle = 0 & {\rm if~ } l \neq i, \\
{\rm (iii)~} & \left\langle U, \prod_{i,j,k}  \left( a_{i,j,k}^{\left(l_{i,j,k}\right)} | \iota_1\ldots \iota_{l_{i,j,k}} \right) \right\rangle = \prod_{i,j,k} \left\langle U,  \left( a_{i,j,k}^{\left(l_{i,j,k}\right)} | \iota_1\ldots \iota_{l_{i,j,k}} \right) \right\rangle,
\end{align*}
where in (iii), the order of the letterplaces in the word $\prod_{i,j,k}  ( a_{i,j,k}^{(l_{i,j,k})} | \iota_1\ldots \iota_{l_{i,j,k}} )$ must be according to the order of the letters $a_{i,j,k}$.
\end{definition}

\begin{example}
For arbitrary $n$ and any permutation $\sigma \in S_n$, we have
$$
\left\langle U, \left[ a_{1,1,1} \cdots a_{1,n,1} \right] \right\rangle =
\left\langle U, \left[ a_{1,\sigma(1),1} \cdots a_{1,\sigma(n),1} \right] \right\rangle =
 \left| \begin{array}{ccc}
T_{1,1,1} & \cdots & T_{1,n,1} \\
\vdots & \ddots & \vdots \\
T_{1,1,n} & \cdots & T_{1,n,n}
\end{array}
\right|
$$
as $a_{1,1,1} < \ldots < a_{1,n,1}$.
\end{example}

\begin{theorem}[\cite{GRS}, Thm. 18]
The umbral operator $U:{\rm Ext}([L|P]) \to  \CC[W] $ is surjective and its restriction to the bracket polynomials ${\rm Br}(L)$ is onto $\CC[W]^{\SL_n}$.
\end{theorem}

We are only interested in bracket polynomials that are not in the kernel of $U$. Thus in the following, we consider the subalgebra ${\rm Bra(L)}$ of \emph{appropriate} bracket polynomials, where if a letter $a_{i,j,k}$ turns up in an appropriate bracket monomial, it does so exactly $i$ times. Of course, the restriction of $U$ to  ${\rm Bra(L)}$ is still onto $\CC[W]^{\SL_n}$.

\section{Brackets and graphs}\label{sec:bragra}

In this section, we develop the basis of our method:   bracket polynomials are associated with formal sums of colored  hypergraphs.

\begin{definition}
Let $X$ be a set. Then we denote by $\MMM(X)$ the set of nonempty multisets composed of elements of $X$.
\end{definition}

\begin{definition}
Let $m$ be a positive integer.
An undirected $n$-regular \emph{colored hypergraph} $\Gamma$ with $m$ vertices  is a pair $\Gamma=(\Ver,\Edg)$, where $\Ver=\{v_1 < \ldots <v_m\}$ is the ordered set of \emph{vertices} and $\Edg \in \mathcal{P}\left(\mathcal{M}(\Ver) \times \NN \times \NN\right)$ is the set of colored \emph{hyperedges} $e=(e_1,e_2,e_3)$, and for all vertices $v$, we have 
$$
\sum\limits_{e \in E} \#_{v}\left( e_1 \right) = n,
$$  
where $\#_{v}\left( e_1 \right)$ is the \emph{number of connections} of $e$ to $v$, i.e. the number of occurences of the element $v$ in the multiset $e_1$. 
By the \emph{virtual degree} of a vertex $v$, we mean the number 
$$
{\rm vdeg}(v)=n-\sum\limits_{(\{v,\ldots,v\},e_2,e_3) \in E} \left| \{v,\ldots,v\} \right|.
$$
By the \emph{effective graph} $\Gamma_\eff$ of $\Gamma=(\Ver,\Edg)$, we denote the subgraph 
$$
\Gamma_\eff=(\Ver,\Edg \setminus \{ e | e_1=\{v,\ldots,v \}, v \in \Ver\}).
$$ 
If $e=(e_1,e_2,e_3)$ is a hyperedge, then we call  $k=|e_1|$ the \emph{size}, $e_2 \in \NN$ the \emph{color} and $e_3 \in \NN$ the \emph{shading} of $e$. We call $e$ a $k$-edge. We say that $e$ is connected to $v$, if $v \in e_1$. If $e$ is connected to only one $v$, then we call it a \emph{looping edge}.
\end{definition}

Observe that multiple edges and loops are allowed in this definition of a hypergraph, so it is truly a \emph{pseudo-hypergraph}.

\begin{definition}
Now let $G$ be the $\CC$-vector space of formal sums of $\CC$-multiples of $n$-regular colored hypergraphs. On $G$, we define a (non-abelian) multiplication as follows. For $\Gamma_1=(\{v_1 < \ldots <v_m\},\Edg_1)$ and $\Gamma_2=(\{w_1 < \ldots <w_m\},\Edg_2)$ in $G$, we set
$$
\Gamma_1 \Gamma_2 := \Gamma_1 \cdot \Gamma_2 := (\{v_1 < \ldots <v_m<w_1 < \ldots <w_m\},\Edg_1 \cup \Edg_2)
$$
and extend to formal sums of graphs in the obvious way.
This makes $G$ a $\CC$-algebra. We call elements $\Upsilon=\sum a_i \Gamma_i \in G$ \emph{graphsums}.
\end{definition}

To a hypergraph $\Gamma=(\{v_1 < \ldots <v_m\},\Edg_1)$, we associate a bracket monomial $p_\Gamma=b_1 \cdots b_m$ with brackets $b_1,\ldots, b_m$ defined by:
$$
b_i :=\left[\prod\limits_{e \in \Edg} a_{|e_1|,e_2,e_3}^{(\#_{v_i}(e_1))}\right].
$$
\noindent
This gives a linear surjective map $\gamma: G \to {\rm Bra(L)}$ by setting
$$
\gamma: \sum a_i \Gamma_i \mapsto \sum a_i p_{\Gamma_i}.
$$
\noindent
Now we set $\mathcal{G}:=G / {\rm ker}(\gamma)$ and by $\gamma': \mathcal{G} \to {\rm Bra}(L)$ denote the induced isomorphism.

\begin{example} \label{ex:grbr}
We associate the bracket monomial
$$
\brB{a_{3,1,1}^{(3)}a_{3,4,1}}\brB{a_{2,2,3}^{(2)}a_{3,4,1}a_{2,1,1}}\brB{a_{3,4,1}a_{2,1,1}a_{1,7,1}a_{1,2,5}}
$$
\noindent
to the color-and-shading-labeled hypergraph
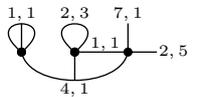
\begin{center}
\begin{tikzpicture}[baseline=(A),outer sep=0pt,inner sep=0pt]

\node (A) at (0,0) {};

\node[font=\tiny,align=center] at (0,-0.5) {$4,1$};

\draw[line width=0.5pt, draw=black] (0,0)-- (0,-0.38);

\draw[black, fill=black] (0,0) circle [radius=1.5pt];

\draw[black, fill=black] (0.7,0) circle [radius=1.5pt];

\draw[black, fill=black] (-0.7,0) circle [radius=1.5pt];

\thre{-0.7}{}
\node[font=\tiny,align=center] at (-0.7,0.5) {$1,1$};

\twoe{0}{}
\node[font=\tiny,align=center] at (0,0.5) {$2,3$};

\draw[line width=0.5pt, draw=black] (0,0)-- (1.08,0);
\draw[line width=0.5pt, draw=black] (0.7,0)-- (0.7,0.38);

\node[font=\tiny,align=center] at (0.4,0.1) {$1,1$};
\node[font=\tiny,align=center] at (0.7,0.5) {$7,1$};
\node[font=\tiny,align=center] at (1.3,0) {$2,5$};

\draw[line width=0.5pt, draw=black]  (0.7,0) .. controls (0.6,-0.5) and (-0.6,-0.5) .. (-0.7,0) ;

\end{tikzpicture}
\gES{.}
\end{center}
\end{example}

\begin{convention}
We will often speak only of \emph{graphs}, when we mean colored hypergraphs. Moreover, we will not number vertices of graphs, but will assume that they are ordered ascending from left to right. We also ignore shading of edges, assuming that all $k$-edges of the same color have different shading.
\end{convention}

\begin{definition}\label{def:red}
We say that two graphsums $\Upsilon_1$ and $\Upsilon_2$ are \emph{equivalent}, writing $\Upsilon_1\simeq\Upsilon_2$ , if  $\Upsilon_1 - \Upsilon_2 \in \ker(U \circ \gamma)$. We call a graphsum $\Upsilon$ \emph{reducible}, if it is equivalent to zero or to some $\sum a_i \Gamma_i$ with all $\Gamma_i$ disconnected. A graphsum is irreducible if it is not reducible. 

Several graphsums $\Upsilon_1,\ldots,\Upsilon_N$ are called \emph{reducibly independent}, if the only reducible linear combination $\sum a_i\Upsilon_i$ is the trivial one. If for two graphsums $\Upsilon_1$, $\Upsilon_2$ the linear combination $\Upsilon_1-\Upsilon_2$ is reducible, we call them \emph{reducibly equivalent} and write $\Upsilon_1\simeq_r\Upsilon_2$. We say that a set of reducibly independent  irreducible graphs has property $(RI)$.
\end{definition}

\begin{remark}
The Exchange Lemma~\ref{le:exch} leads to equivalencies between graphsums.
If two graphsums are equivalent in this way, for any pair of graphs occuring in the two graphsums, there is a color-preserving one-to-one-correspondence between the edges.
\end{remark}

\begin{theorem}
Let $M$ be a maximal set with property $(RI)$. Then $M$ is in one-to-one-correspondence to a minimal generating set of $\CC[W]^{\SL_n}$ by
$
\Gamma \mapsto U \circ \gamma (\Gamma)
$.
\end{theorem}

\begin{proof}
Let $M$ be a maximal set of reducibly independent irreducible graphs in $G$. Let $F:=U \circ \gamma(M)$. Since $U\circ \gamma$ is surjective, for any element $f$ of $\CC[W]^{\SL_n}$, we have a graphsum $\Gamma$ with $U\circ \gamma (\Gamma) =f$. If $\Gamma$ is irreducible, then either $\Gamma \in M$ and thus $f \in F$, or if $\Gamma \notin M$, due to maximality  of $M$, there is a reducible nontrivial linear combination
$$
\Gamma + \sum_{\Gamma' \in M} a_{\Gamma'} \Gamma'.
$$
Due to linearity of $U \circ \gamma$, we can proceed with reducible $\Gamma$. Either $\Gamma=0$, then $f=0$, or $\Gamma$ is equivalent to a graphsum of disconnected graphs. We can assume all connected  subgraphs are irreducible and thus proceed with an irreducible graph, where the number of vertices is strictly less than that of $\Gamma$. Since the number of vertices of graphs is bounded from below, this procedure comes to an end. So $F$ generates $\CC[W]^{\SL_n}$. The minimality of $F$ follows immediately from $M$ being reducibly independent.
\end{proof}

\begin{lemma}\label{le:redvirtdeg}
Let the graph $\Gamma=(\Ver,\Edg)$ have a hyperedge $e'=(e_1',e_2',e_3')$ with $v \in e_1'$. Then $\Gamma$ is equivalent to a graphsum $\sum_i (\Ver,\Edg_i)$ where $\Edg_i \setminus \{e|v \in e_1\} \subseteq \Edg \setminus \{e|v \in e_1\}$ and $(\{v,...,v\},e_2',e_3') \in \Edg_i$ for all $i$.
\end{lemma}

\begin{proof}
We can assume that $v$ is the smallest element of $\Ver$. If $\#_v(e_1')=|e_1'|$, we are done. Thus take $\#_v(e_1')=\kappa<|e_1'|$. We can assume that for the second smallest element $v'$ of $\Ver$ we have $\#_v'(e_1')=\nu \geq 1$. So
$$
\gamma(\Gamma)=\brB{a^{(\kappa)}_{|e_1'|,e_2',e_3'}w_1}\brB{a^{(\nu)}_{|e_1'|,e_2',e_3'}w_2}\mathfrak{m}
$$ 
for some $w_1,w_2 \in L$ and $\mathfrak{m}$ a bracket monomial. Now applying Lemma~\ref{le:exch} with $u=a^{(\kappa+\nu)}_{|e_1'|,e_2',e_3'}$, we see that $\Gamma$ is equivalent to a graphsum $\sum_i (\Ver,\Edg_i)$ where $\Edg_i \setminus \{e|v \in e_1\} \subseteq \Edg \setminus \{e|v \in e_1\}$ and $(e_1'',e_2',e_3') \in \Edg_i$ with $\#_v(e_1'')=\kappa+\nu$. Iterating this gives the desired result.
\end{proof}

\begin{remark}
When searching for a maximal set of reducibly independent irreducible graphs, Lemma~\ref{le:redvirtdeg} can be used to simplify the respective effective graphs.
\end{remark}

\begin{proposition}\label{prop:maxcolors}
Let $\Gamma$ be an irreducible graph with arbitrary coloring which is not reducibly equivalent to $c\Upsilon_{\det(J)}$ for any $J\subseteq \{1,\ldots,n_k\}$ and $c \in \CC$. Let $M$ be the set of graphs of the same form as $\Gamma$ but with $k$-edges of only $\tbinom{n}{k}-1$ different colors.
Then at least one element of $M$ is irreducible.
Moreover, if $n_k \geq \tbinom{n}{k}$ and $J \subseteq \{1,\ldots,n_k\}$ has cardinality $\tbinom{n}{k}$, then $\det:\oplus_{j \in J} V_{k,j}\to \CC$ equals  $U \circ \gamma(\Upsilon_{\det(J)})$ for some irreducible graphsum $\Upsilon_{\det(J)}$.
\end{proposition}

\begin{proof}
This follows directly from Theorem 9.2 of~\cite{PV}.
\end{proof}


\section{Invariants of  $\SL_4$}
\label{sec:SL4}

In the case of $\SL_4$, it turns out that elementary graph theoretical and combinatorial considerations suffice to determine a minimal generating set of $\CC[W]^{\SL_4}$. This means that the present section is almost totally self-contained.

\begin{proposition}\label{prop:maxset}
The following graphs constitute a maximal set of reducibly independent irreducible graphs for the action of $\SL_4$ on $W$.

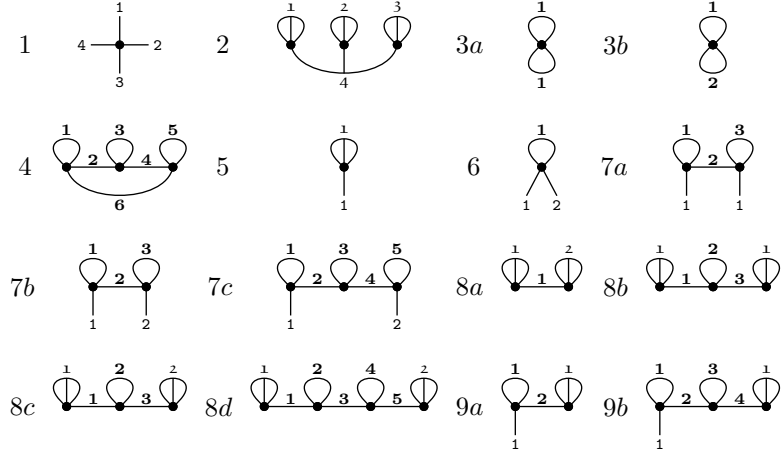
\setlength{\tabcolsep}{1mm}

\begin{longtable}{rcrcrcrc}
\gS{1}
&
\gA{1}{2}{3}{4}
&
\gS{2}
&
\gB{1}{2}{3}{4}
&
\gS{3a}
&
\gCa{1}
&
\gS{3b}
&
\gCb{1}{2}
\\ \ \\
\gS{4}
&
\gD{1}{2}{3}{4}{5}{6}
&
\gS{5}
&
\gE{1}{1}
&
\gS{6}
&
\gF{1}{1}{2}
&
\gS{7a}
&
\gGa{1}{2}{3}{1}
\\ \ \\
\gS{7b}
&
\gGb{1}{2}{3}{1}{2}
&
\gS{7c}
&
\gGc{1}{2}{3}{4}{5}{1}{2}
&
\gS{8a}
&
\gHa{1}{1}{2}
&
\gS{8b}
&
\gHb{1}{2}{3}{1}
\\ \ \\
\gS{8c}
&
\gHc{1}{2}{3}{1}{2}
&
\gS{8d}
&
\gHd{1}{2}{3}{4}{5}{1}{2}
&
\gS{9a}
&
\gIa{1}{2}{1}{1}
&
\gS{9b}
&
\gIb{1}{2}{3}{4}{1}{1}
\end{longtable}
\end{proposition}

\begin{remark}\label{re:4det}
According to Proposition~\ref{prop:maxcolors}, some multiple of graph no. 4 must be reducibly equivalent to some graphsum $\Upsilon$ with $U\circ\gamma(\Upsilon)=\det$. In fact,
\begin{center}
\gS{\Upsilon=}%
\gD{1}{2}{3}{4}{5}{6}
\gS{-}
\gD{6}{1}{2}{3}{4}{5}%
\gS{\simeq_r}
\gS{2}%
\gD{1}{2}{3}{4}{5}{6}%
\end{center} 
has this property, as it is alternating in the colors, i.e. $\Upsilon$ is not only reducibly, but truly equivalent to $\sgn(\sigma)\Upsilon_\sigma$. The more pleasant display has of course graph no. 4, while for some purposes like for example finding relations, $\Upsilon$ will do better. The corresponding bracket polynomial, as well as those of graphs no. 3a and 3b, turns also up in~\cite{mcmillan, RStu, White}, while no attempt is made there to show that these give a minimal generating set of $\CC[W_{(0,n_2,0,0)}]^{\SL_4}$.
\end{remark}

\begin{lemma}\label{le:redgraphs}
Every irreducible graph for the action of $\SL_4$ on $W$ is reducibly  equivalent to a graphsum $\sum a_i \Gamma_i$, where all $\Gamma_i$ are of the same form. This form is one of the following:
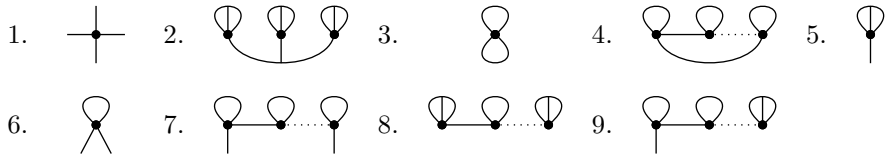
\begin{longtable}{rcrcrcrcrc}
\gS{1.}
&
\gA{}{}{}{}
&
\gS{2.}
&
\gB{}{}{}{}
&
\gS{3.}
&
\gCa{}
&
\gS{4.}
&
\gDA{}{}{}{}{}{}
&
\gS{5.}
&
\gE{}{}
\\
\gS{6.}
&
\gF{}{}{}
&
\gS{7.}
&
\gGA{}{}{}{}{}{}{}
&
\gS{8.}
&
\gHA{}{}{}{}{}

&
\gS{9.}
&
\gIA{}{}{}{}{}{}

\end{longtable}

\end{lemma}

\begin{proof}
Let $\Gamma$ be an irreducible graph. First assume $\Gamma$ has only $1$-edges. Then $\Gamma$ is connected and thus irreducible only if it has one vertex. We are in Case $1$. 

Now assume $\Gamma$ has only $3$-edges. Then with Lemma~\ref{le:redvirtdeg}, it is reducibly equivalent to a sum of graphs $\Gamma_i$ having only vertices of virtual degree one. Now take one vertex $v_1$ of an arbitrary $\Gamma_i$. There must be a non-looping $3$-edge $e$ with one connection to $v_1$. Since all vertices have virtual degree one, there must be two vertices $v_2$ and $v_3$ with $e=\{v_1,v_2,v_3\}$. Thus $v_1$, $v_2$ and $v_3$ together with $e$ and their respective looping $3$-edge are a connected component of $\Gamma_i$ and since $\Gamma_i$ is connected, we are in Case $2$.

If now $\Gamma$ has only $2$-edges and it has only one vertex, we are in Case $3$. Assume it has more than one vertex. Again by Lemma~\ref{le:redvirtdeg}, $\Gamma$ is reducibly equivalent to a sum of connected graphs $\Gamma_i$ where for each of them every vertex has one (and only one) looping edge and is thus of virtual degree two. 
The effective graph of an arbitrary $\Gamma_i$ must be a \emph{connected simple 2-regular graph}, i.e. a cycle, and we are in Case $4$.

We come to the cases where $\Gamma$ has hyperedges of two different sizes. Let us begin assuming there are $1$- and $3$-edges. Again using Lemma~\ref{le:redvirtdeg}, we can move on to some $\Gamma_i$ with vertices $v_1, \ldots, v_N$ each with a looping $3$-edge. If there is an additional vertex $v_{N+1}$ with a connection to a $3$-edge, we can move on to a graph having $N+1$ vertices with a looping $3$-edge. If there is an additional vertex with no connection to a $3$-edge, it constitutes a connected component like in Case 1. Thus we can assume all vertices of $\Gamma_i$ have a looping $3$-edge. Now if any vertex despite for its looping $3$-edge has a connection to another $3$-edge, we have a connected component like in Case 2. Thus each vertex must have a connection to a $1$-edge. Since $\Gamma_i$ is connected we are in Case $5$.

Now let $\Gamma$ have $1$- and $2$-edges. We use Lemma~\ref{le:redvirtdeg}, move on to some $\Gamma_i$ and can with the same argumentation as in the previous case assume that every vertex has a looping $2$-edge. If there is a vertex connected to two $1$-edges, we are in Case $6$. If not, there are vertices of virtual degree one and two. Thus the effective graph of $\Gamma_i$  must be a \emph{chain} and we are in Case $7$.

The same argumentation goes through if $\Gamma$ has $2$- and $3$-edges. By Lemma~\ref{le:redvirtdeg}, we have vertices of virtual degree one and two and the effective graph of $\Gamma_i$ must be a chain, giving Case $8$.

Finally let $\Gamma$ have $1$-, $2$-, and $3$-edges. Once more, we have vertices of virtual degree one and two by Lemma~\ref{le:redvirtdeg} and the effective graph of $\Gamma_i$ must be a chain, leading to Case $9$.
\end{proof}

When we say '$\Gamma$ is of the form $n$' in the following, we mean that the graph $\Gamma$ falls under Case $n$ of Lemma~\ref{le:redgraphs}, while we use the term 'graph no. n' for the colored graphs from Proposition~\ref{prop:maxset}.

\begin{lemma}\label{le:tricks}
Let $\Gamma$ be one of the graphs from Lemma~\ref{le:redgraphs} with two or more vertices and arbitrary coloring, let $\sigma$ be a permutation of the colors of the $2$-edges and $\Gamma_\sigma$ be the graph with permuted colors. Then $\Gamma\simeq_r\sgn(\sigma)\Gamma_\sigma$.
\end{lemma}

\begin{proof}
We show that $\Gamma\simeq_r-\Gamma_\sigma$, where $\sigma$ swaps the colors of a looping and a non-looping $2$-edge connected with the same vertex. From this the assertion follows immediately.

So let $v \neq v'$ be vertices of $\Gamma$ and $(\{v,v\},j_1,k_1)$, $(\{v,v'\},j_2,k_2)$ the respective two edges. We use Lemma~\ref{le:exch} with $u=a_{2,j_1,k_1}^{(2)}a_{2,j_2,k_2}^{(2)}$ and get
$
\Gamma+\Gamma_\sigma\simeq-\Gamma',
$
where $\Gamma'$ has the looping $2$-edges $(\{v,v\},j_1,k_1)$, $(\{v,v\},j_2,k_2)$ and is thus reducible. So $\Gamma\simeq_r-\Gamma_\sigma$.
\end{proof}

\begin{proof}[Proof of Proposition~\ref{prop:maxset}]
For graphs $\Gamma$ of the forms $1$, $2$ and $5$, the corresponding invariants are 'determinants' and 'dot products', in classical terms 'invariants of systems of vectors and linear forms', see for example~\cite[p. 254]{PV}. Graphs of the form $3$ either give the 'Pfaffian' if both edges have the same color, or a variation of it. Graphs $\Gamma$ of the form $6$ neither are disconnected nor $U \circ \gamma(\Gamma)=0$ unless both $1$-edges have the same color. 

The remaining types to check for irreducibility are $4$, $7$, $8$, $9$. By the duality of $V$ and $\BigWedge^{n-1}V$, we can reduce form 8 to form 7. Thus $4$, $7$, $9$ remain.  In all these cases we require all $2$-edges to be of pairwise different color, otherwise Lemma~\ref{le:tricks} with $\sigma$ swapping two edges of the same color would result in $\Gamma \simeq_r (\Gamma + \Gamma_\sigma)/2$, which is reducible.

Now first we show that if a graph of one of these types has four or more vertices, it is reducible. We observe graphs $\Gamma_i$ of the forms:
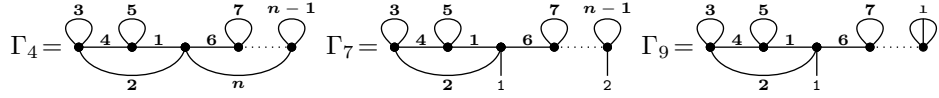
\begin{center}
\gS{\Gamma_4\!=}%
\begin{tikzpicture}[baseline=(A),outer sep=0pt,inner sep=0pt]
\node (A) at (0,0) {};
\twoe{0}{3};
\twoe{0.7}{5};
\twoe{2.1}{7};
\twoe{2.8}{};
\node[font=\tiny,align=center]  at (2.8,0.5) {$\iiE{n}-\iiE{1}$};
\twoeAB{0}{1.4}{2};
\twoeAB{1.4}{2.8}{n};
\twoeABS{0}{0.7}{4};
\twoeABS{0.7}{1.4}{1};
\twoeABS{1.4}{2.1}{6};
\draw[line width=0.5pt, draw=black, dotted] (2.1,0)-- (2.8,0);
\end{tikzpicture}
\gS{\Gamma_7\!=}%
\begin{tikzpicture}[baseline=(A),outer sep=0pt,inner sep=0pt]
\node (A) at (0,0) {};
\twoe{0}{3};
\twoe{0.7}{5};
\twoe{2.1}{7};
\twoe{2.8}{};
\node[font=\tiny,align=center]  at (2.8,0.5) {$\iiE{n}-\iiE{1}$};
\twoeAB{0}{1.4}{2};
\onee{1.4}{1};
\onee{2.8}{2};
\twoeABS{0}{0.7}{4};
\twoeABS{0.7}{1.4}{1};
\twoeABS{1.4}{2.1}{6};
\draw[line width=0.5pt, draw=black, dotted] (2.1,0)-- (2.8,0);
\end{tikzpicture}
\gS{\Gamma_9\!=}%
\begin{tikzpicture}[baseline=(A),outer sep=0pt,inner sep=0pt]
\node (A) at (0,0) {};
\twoe{0}{3};
\twoe{0.7}{5};
\twoe{2.1}{7};
\thre{2.8}{1};
\twoeAB{0}{1.4}{2};
\onee{1.4}{1};
\twoeABS{0}{0.7}{4};
\twoeABS{0.7}{1.4}{1};
\twoeABS{1.4}{2.1}{6};
\draw[line width=0.5pt, draw=black, dotted] (2.1,0)-- (2.8,0);
\end{tikzpicture}
\end{center}

\noindent
We proceed exemplarily with $\Gamma_4$. Applying Lemma~\ref{le:exch} with $u$ the word corresponding to the $2$-edge of color $\iiE{1}$, we get that 
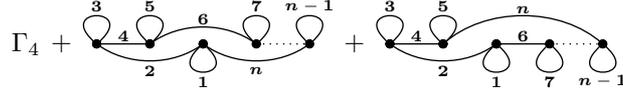
\begin{center}
\gS{\Gamma_4}
\gS{+}
\begin{tikzpicture}[baseline=(A),outer sep=0pt,inner sep=0pt]
\node (A) at (0,0) {};
\twoe{0}{3};
\twoe{0.7}{5};
\twoe{2.1}{7};
\twoe{2.8}{};
\node[font=\tiny,align=center]  at (2.8,0.5) {$\iiE{n}-\iiE{1}$};
\twoeABfd{0}{1.4}{2}{0.3};
\twoeABfd{1.4}{2.8}{n}{0.3};
\twoeABS{0}{0.7}{4};
\twoed{1.4}{1};
\twoeABfu{0.7}{2.1}{6}{0.3};
\draw[line width=0.5pt, draw=black, dotted] (2.1,0)-- (2.8,0);
\end{tikzpicture}
\gS{+}
\begin{tikzpicture}[baseline=(A),outer sep=0pt,inner sep=0pt]
\node (A) at (0,0) {};
\twoe{0}{3};
\twoe{0.7}{5};
\twoed{2.1}{7};
\twoed{2.8}{};
\node[font=\tiny,align=center]  at (2.8,-0.5) {$\iiE{n}-\iiE{1}$};
\twoeABfd{0}{1.4}{2}{0.3};
\twoeABfu{0.7}{2.8}{n}{0.5};
\twoeABS{0}{0.7}{4};
\twoed{1.4}{1};
\twoeABS{1.4}{2.1}{6};
\draw[line width=0.5pt, draw=black, dotted] (2.1,0)-- (2.8,0);
\end{tikzpicture}
\end{center}

\noindent
is reducible.
Let us call the second and third graph in this sum $\Gamma_{4,1}$ and $\Gamma_{4,2}$ respectively.
Since graphs with permuted vertices are equivalent, by swapping colors $1$ and $5$ as well as $2$ and $4$, we get  $\Gamma_{4,2}\simeq_r\Gamma_{4,1}$, so $\Gamma_4\simeq_r -2\Gamma_{4,1}$. Applying Lemma~\ref{le:exch} again, now with $u$ the word corresponding to the $2$-edge of color $6$, we get that

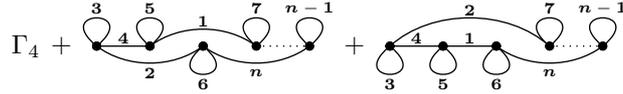
\begin{center}
\gS{\Gamma_4}
\gS{+}
\begin{tikzpicture}[baseline=(A),outer sep=0pt,inner sep=0pt]
\node (A) at (0,0) {};
\twoe{0}{3};
\twoe{0.7}{5};
\twoe{2.1}{7};
\twoe{2.8}{};
\node[font=\tiny,align=center]  at (2.8,0.5) {$\iiE{n}-\iiE{1}$};
\twoeABfd{0}{1.4}{2}{0.3};
\twoeABfd{1.4}{2.8}{n}{0.3};
\twoeABS{0}{0.7}{4};
\twoed{1.4}{6};
\twoeABfu{0.7}{2.1}{1}{0.3};
\draw[line width=0.5pt, draw=black, dotted] (2.1,0)-- (2.8,0);
\end{tikzpicture}
\gS{+}
\begin{tikzpicture}[baseline=(A),outer sep=0pt,inner sep=0pt]
\node (A) at (0,0) {};
\twoed{0}{3};
\twoed{0.7}{5};
\twoe{2.1}{7};
\twoe{2.8}{};
\node[font=\tiny,align=center]  at (2.8,0.5) {$\iiE{n}-\iiE{1}$};
\twoeABfd{1.4}{2.8}{n}{0.3};
\twoeABfu{0}{2.1}{2}{0.5};
\twoeABS{0}{0.7}{4};
\twoed{1.4}{6};
\twoeABS{0.7}{1.4}{1};
\draw[line width=0.5pt, draw=black, dotted] (2.1,0)-- (2.8,0);
\end{tikzpicture}
\end{center}

\noindent
is reducible.
Calling the second and third graph in this sum $\Gamma_{4,3}$ and $\Gamma_{4,4}$ respectively, by symmetry reasons, we have $\Gamma_{4,3}\simeq_r\Gamma_{4,4}$ and by swapping the edges of colors $1$ and $6$ in $\Gamma_{4,3}$ we get $\Gamma_{4,3}\simeq_r-\Gamma_{4,1}$, thus $2\Gamma_{4,1}\simeq_r-2\Gamma_{4,1}$ and $\Gamma_{4,1}$ must be reducible. Exactly the same procedure, namely two times applying Lemma~\ref{le:exch}, one time on the edge of color 1, one time on that of color 6, leads to reducibility of graphs of the forms 4, 7, and 9 with four or more vertices. We observe these with three or less vertices in the following.

\medskip

\noindent
\textbf{Case 1: $\Gamma$ is of the form 7.} Here $\Gamma$ has either two or three vertices.  
We first check these with two vertices. All of the graphs with $2$-edges of pairwise different colors $\iiE{1}$, $\iiE{2}$, $\iiE{3}$, and $1$-edges of possibly non-different colors $\iE{1}$ and $\iE{2}$ are either reducible or of the form
$\Gamma_\sigma$, where $\sigma$ permutes colors of $2$-edges and 
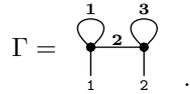
\begin{center}
\gS{\Gamma=}
\gGb{1}{2}{3}{1}{2}
\gES{.}
\end{center}

\noindent
We define a map $\phi:G\to\CC$ by setting $\phi(c\Gamma_\sigma)=\sgn(\sigma)c$ and $\phi(\Gamma')=0$ for graphs $\Gamma'$ of other forms. Since $\Gamma \notin \ker(U\circ\gamma)$ and all relations from Lemma~\ref{le:exch} with $|u|=2$ involving $\Gamma$ are compatible with $\phi$ in the sense that $\phi(\Upsilon_1)=\phi(\Upsilon_2)$ if $\gamma(\Upsilon_1)=\gamma(\Upsilon_2)$ for two graphsums $\Upsilon_1=\Gamma+\Upsilon_1'$ and $\Upsilon_2$, $\phi$ induces a well-defined map $\phi':G/\ker(U\circ\gamma)\to\CC$ and $\Gamma$ is irreducible.

Now if $\Gamma$ is of type 7 with three vertices and the two $1$-edges are of the same color, we have $U\circ\gamma(\Gamma)=-U\circ\gamma(\Gamma_\tau)$, where $\tau$ interchanges the two $1$-edges. On the other hand, we have 
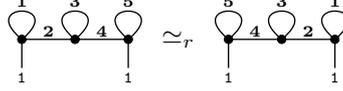
\begin{center}
\gGc{1}{2}{3}{4}{5}{1}{1}
\gS{\simeq_{r}}
\gGc{5}{4}{3}{2}{1}{1}{1}
\end{center}

\noindent
by swapping colors $\iiE{1}$ and $\iiE{5}$ as well as $\iiE{2}$ and $\iiE{4}$ of $2$-edges and see that $\Gamma$ is reducible.  Now all graphs with three vertices and with $2$-edges of pairwise different colors $\iiE{1}$-$\iiE{5}$ and $1$-edges of different colors $\iE{1}$ and $\iiE{2}$ are either reducible or of the form $\Gamma_\sigma$, $\Gamma'_\sigma$, where $\sigma$ permutes $2$-edges and 

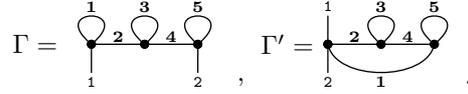
\begin{center}
\gS{\Gamma=}
\gGc{1}{2}{3}{4}{5}{1}{2}
\gES{,}
\gS{\Gamma'=}
\begin{tikzpicture}[baseline=(A),outer sep=0pt,inner sep=0pt]
\node (A) at (0,0) {};
\node[font=\tiny,align=center] at (0,-0.5) {$\iiE{1}$};
\node[font=\tiny,align=center] at (-0.35,0.1) {$\iiE{2}$};
\node[font=\tiny,align=center] at (0.35,0.1) {$\iiE{4}$};
\draw[line width=0.5pt, draw=black] (-0.7,0)-- (0.7,0);
\draw[black, fill=black] (0,0) circle [radius=1.5pt];
\draw[black, fill=black] (0.7,0) circle [radius=1.5pt];
\draw[black, fill=black] (-0.7,0) circle [radius=1.5pt];
\oneeu{-0.7}{1}
\onee{-0.7}{2}
\twoe{0}{3}
\twoe{0.7}{5}
\draw[line width=0.5pt, draw=black]  (0.7,0) .. controls (0.6,-0.5) and (-0.6,-0.5) .. (-0.7,0) ;
\end{tikzpicture}
\gES{.}
\end{center}

\noindent
By applying Lemma~\ref{le:exch} on the $2$-edge of $\Gamma'$ of  color $\iiE{1}$, we see that $\Gamma'\simeq_r -2\Gamma$. Similar as we did before, we define a map $\phi:G\to\CC$ by setting $\phi(c\Gamma_\sigma)=\sgn(\sigma)c$, $\phi(c\Gamma'_\sigma)=-\sgn(\sigma)2c$ and $\phi(\Gamma^*)=0$ for graphs $\Gamma^*$ of other forms. As before, $\phi$ induces a well-defined map $\phi':G/\ker(U\circ\gamma)\to\CC$ and $\Gamma$ thus is irreducible.

\medskip

\noindent
\textbf{Case 2: $\Gamma$ is of the form 4.}
If such a graph has two vertices, it has two edges $(\{v_1,v_2\},j_1,k_1)$ and $(\{v_1,v_2\},j_2,k_2)$.  Then by Lemma~\ref{le:exch} with $u=a_{2,j_1,k_1}^{(2)}a_{2,j_2,k_2}^{(2)}$, it is reducible. The case of three vertices remains.
All graphs with three vertices and six $2$-edges of different colors  $\iiE{1}$-$\iiE{6}$ are either reducible or of the form $\Gamma_\sigma$, where

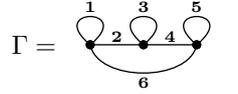
\begin{center}
\gS{\Gamma=}
\gD{1}{2}{3}{4}{5}{6}
\gES{.}
\end{center}

Again the map $\phi:G\to\CC$ defined by $\phi(c\Gamma_\sigma)=\sgn(\sigma)c$ and $\phi(\Gamma^*)=0$ for graphs $\Gamma^*$ of other forms induces a well-defined map $\phi':G/\ker(U\circ\gamma)\to\CC$ and $\Gamma$ thus is irreducible.

\medskip

\noindent
\textbf{Case 3: $\Gamma$ is of the form 9.} Here $\Gamma$ can have two or three vertices. We begin with the case of two. All graphs with  one $1$-edge of color $\iE{1}$, two $2$-edges of colors $\iiE{1}$ and $\iiE{2}$, and one $3$-edge of color $\iiiE{1}$ are either reducible or of the form $(\Gamma_{i})_{\sigma}$ with

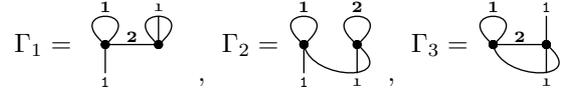
\begin{center}
\gS{\Gamma_1=}
\gIa{1}{2}{1}{1}
\gES{,}
\gS{\Gamma_2=}
\begin{tikzpicture}[baseline=(A),outer sep=0pt,inner sep=0pt]
\node (A) at (0,0) {};
\twoe{0}{1}
\twoe{0.7}{2}
\onee{0}{1}
\threAB{0}{0.7}{1}
\end{tikzpicture}
\gES{,}
\gS{\Gamma_3=}
\begin{tikzpicture}[baseline=(A),outer sep=0pt,inner sep=0pt]
\node (A) at (0,0) {};
\twoe{0}{1}
\twoeABS{0}{0.7}{2}
\oneeu{0.7}{1}
\threAB{0}{0.7}{1}
\end{tikzpicture}
\gES{.}
\end{center}

Applying Lemma~\ref{le:exch} with $u$ the word corresponding to the $3$-edge, we see $\Gamma_3\simeq_r\Gamma_2\simeq_r\Gamma_1$. The map $\phi:G\to\CC$ defined by $\phi(c (\Gamma_i)_\sigma)=\sgn(\sigma) c$ and $\phi(\Gamma^*)=0$ for graphs $\Gamma'$ of other forms induces a well-defined map $\phi':G/\ker(U\circ\gamma)\to\CC$ and $\Gamma_1$ thus is irreducible.

We come to those graphs with three vertices. All graphs with  one $1$-edge of color $\iE{1}$, four $2$-edges of pairwise different colors $\iiE{1}$-$\iiE{4}$, and one $3$-edge of color $\iiiE{1}$ are either reducible or of the form $(\Gamma_{i})_{\sigma}$ with

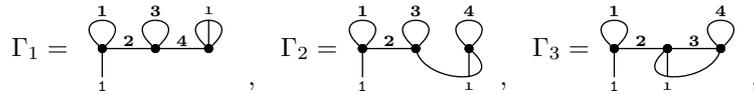
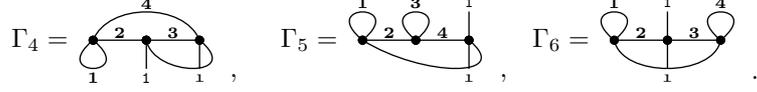
\begin{longtable}{lll}
\gS{\Gamma_1=}
\gIb{1}{2}{3}{4}{1}{1}
\gES{,}
&
\gS{\Gamma_2=}
\begin{tikzpicture}[baseline=(A),outer sep=0pt,inner sep=0pt]
\node (A) at (0,0) {};
\twoe{-0.7}{1}
\twoeABS{-0.7}{0}{2}
\twoe{0}{3}
\twoe{0.7}{4}
\onee{-0.7}{1}
\threAB{0}{0.7}{1}
\end{tikzpicture}
\gES{,}
&
\gS{\Gamma_3=}
\begin{tikzpicture}[baseline=(A),outer sep=0pt,inner sep=0pt]
\node (A) at (0,0) {};
\twoe{-0.7}{1}
\twoeABS{-0.7}{0}{2}
\twoeABS{0}{0.7}{3}
\twoe{0.7}{4}
\onee{-0.7}{1}
\threBA{0}{0.7}{1}
\end{tikzpicture}
\gES{,}
\\ \ \\
\gS{\Gamma_4=}
\begin{tikzpicture}[baseline=(A),outer sep=0pt,inner sep=0pt]
\node (A) at (0,0) {};
\twoed{-0.7}{1}
\twoeABS{-0.7}{0}{2}
\twoeABS{0}{0.7}{3}
\twoeABfu{-0.7}{0.7}{4}{0.5}
\onee{0}{1}
\threAB{0}{0.7}{1}
\end{tikzpicture}
\gES{,}
&
\gS{\Gamma_5=}
\begin{tikzpicture}[baseline=(A),outer sep=0pt,inner sep=0pt]
\node (A) at (0,0) {};
\twoe{-0.7}{1}
\twoeABS{-0.7}{0}{2}
\twoeABS{0}{0.7}{4}
\twoe{0}{3}
\oneeu{0.7}{1}
\threAB{-0.7}{0.7}{1}
\end{tikzpicture}
\gES{,}
&
\gS{\Gamma_6=}
\begin{tikzpicture}[baseline=(A),outer sep=0pt,inner sep=0pt]
\node (A) at (0,0) {};
\twoe{-0.7}{1}
\twoeABS{-0.7}{0}{2}
\twoeABS{0}{0.7}{3}
\twoe{0.7}{4}
\oneeu{0}{1}
\draw[line width=0.5pt, draw=black]  (0,0) -- (0,-0.38) ;

\draw[line width=0.5pt, draw=black]  (0.7,0) .. controls (0.6,-0.5) and (-0.6,-0.5) .. (-0.7,0) ;
\node[font=\tiny,align=center] at (0,-0.5) {$\iiiE{1}$};
\end{tikzpicture}
\gES{.}
\end{longtable}

We apply Lemma~\ref{le:exch} with $u$ the word corresponding to the $3$-edge (to two ends of the $3$-edge in the case of $\Gamma_6$) and get $\Gamma_1\simeq_r\Gamma_3\simeq_r-\Gamma_2\simeq_r-\Gamma_5$ and $\Gamma_6\simeq_r\Gamma_4\simeq_r-2\Gamma_1$. The map $\phi:G\to\CC$ defined in the usual form induces a well-defined map $\phi':G/\ker(U\circ\gamma)\to\CC$ and $\Gamma_1$ thus is irreducible. The proof is complete.

\end{proof}

\begin{remark}
To show that the graphs with four (five in the case of those of form 8 respectively) or more vertices are reducible, we also could use Proposition~\ref{prop:maxcolors} together with the fact that graphs with two $2$-edges of the same color are reducible. We preferred the more self contained version here, because it provides more insight \emph{why} this is so from our combinatorial viewpoint. We want to stress that in our opinion, this reducibility is almost impossible to see without the graph notation, which might explain why Huang in~\cite{Hu} could not reduce her generating set of cycles to a minimal one.

Moreover, to show the irreducibility of the remaining graphs, one can also compute the Hilbert series for small but sufficiently large values of respective $n_i$'s using~\cite[\S 4.6]{KD} and Xin's  algorithm~\cite{Xin} for MacMahon partition analysis. In fact, Xin's algorithm performs very good for such small values in our case.
\end{remark}

\begin{proof}[Proof of Theorem~\ref{th:fftSL4}]
Proposition~\ref{prop:maxset} provides us with a maximal set of reducibly independent irreducible graphs. The corresponding invariants can be computed according to the rules from Definition~\ref{def:umbralop} or by computing the complete contractions given in the introduction. The author used the \emph{DifferentialGeometry} package of \textsc{Maple} for these computations.
\end{proof}

\section{Invariants of $\SL_5$}
\label{sec:SL5}

In order to prove Theorem~\ref{th:fftSL5}, we need some more techniques than the ones we developed for $\SL_4$. The duality of $\BigWedge^k V$ and $\BigWedge^{n-k} V$ becomes very important and we introduce a new reducibility notion that is essential (and will be even more in higher dimensions). As Theorem~\ref{th:fftSL5} does not provide any colorings, we give these exemplarily for $W_{(0,n_2,0,0)}$ in the following.

\begin{proposition}\label{prop:SL52}
The following graphs constitute a maximal set of reducibly independent irreducible graphs for the action of $\SL_5$ on $W_{(0,n_2,0,0)}$, where in each case $1\leq iiE{1}<\iiE{2}<\ldots\leq n_2$  are pairwise different colors of the $2$-edges. Moreover, 
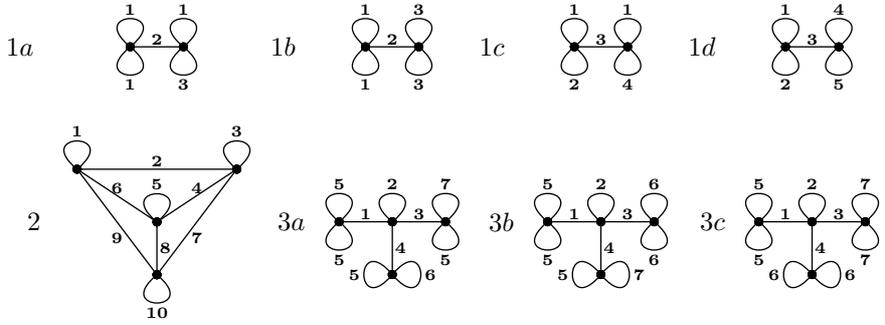
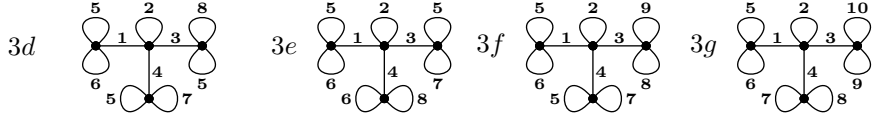
\begin{longtable}{rcrcrcrc}
\gS{1a}
&
\begin{tikzpicture}[baseline=(A),outer sep=0pt,inner sep=0pt]
\node (A) at (0,0) {};
\twoe{0}{1}
\twoed{0}{1}
\twoeABS{0}{0.7}{2}
\twoe{0.7}{1}
\twoed{0.7}{3}
\end{tikzpicture}
&
\gS{1b}
&
\begin{tikzpicture}[baseline=(A),outer sep=0pt,inner sep=0pt]
\node (A) at (0,0) {};
\twoe{0}{1}
\twoed{0}{1}
\twoeABS{0}{0.7}{2}
\twoe{0.7}{3}
\twoed{0.7}{3}
\end{tikzpicture}
&
\gS{1c}
&
\begin{tikzpicture}[baseline=(A),outer sep=0pt,inner sep=0pt]
\node (A) at (0,0) {};
\twoe{0}{1}
\twoed{0}{2}
\twoeABS{0}{0.7}{3}
\twoe{0.7}{1}
\twoed{0.7}{4}
\end{tikzpicture}
&
\gS{1d}
&
\begin{tikzpicture}[baseline=(A),outer sep=0pt,inner sep=0pt]
\node (A) at (0,0) {};
\twoe{0}{1}
\twoed{0}{2}
\twoeABS{0}{0.7}{3}
\twoe{0.7}{4}
\twoed{0.7}{5}
\end{tikzpicture}
\\ \ \\
\gS{2}%
&
\begin{tikzpicture}[baseline=(A),outer sep=0pt,inner sep=0pt]
\node (A) at (0,0.7) {};
\twoeARB{-1.05}{1.4}{1.05}{1.4}{2}{above=1pt};
\twoeARB{1.05}{1.4}{0}{0}{7}{below=4pt};
\twoeARB{-1.05}{1.4}{0}{0}{9}{below=4pt};
\twoeARB{1.05}{1.4}{0}{0.7}{4}{above=1pt};
\twoeARB{-1.05}{1.4}{0}{0.7}{6}{above=1pt};
\twoeARB{0}{0.7}{0}{0}{8}{right=1pt};
\twoep{-1.05}{1.4}{1}
\twoep{1.05}{1.4}{3}
\twoep{0}{0.7}{5}
\twoed{0}{10}
\end{tikzpicture}
&
\gS{3a}%
&
\begin{tikzpicture}[baseline=(A),outer sep=0pt,inner sep=0pt]
\node (A) at (0,0) {};
\twoeARB{0}{0}{0.7}{0}{1}{above=1pt};
\twoeARB{0.7}{0}{1.4}{0}{3}{above=1pt};
\twoeARB{0.7}{0}{0.7}{-0.7}{4}{right=1pt};
\twoetr{0.7}{-0.7}{6}
\twoetl{0.7}{-0.7}{5}
\twoe{0}{5}
\twoed{0}{5}

\twoe{0.7}{2}
\twoe{1.4}{7}
\twoed{1.4}{5}
\end{tikzpicture}
&
\gS{3b}%
&
\begin{tikzpicture}[baseline=(A),outer sep=0pt,inner sep=0pt]
\node (A) at (0,0) {};
\twoeARB{0}{0}{0.7}{0}{1}{above=1pt};
\twoeARB{0.7}{0}{1.4}{0}{3}{above=1pt};
\twoeARB{0.7}{0}{0.7}{-0.7}{4}{right=1pt};
\twoetr{0.7}{-0.7}{7}
\twoetl{0.7}{-0.7}{5}
\twoe{0}{5}
\twoed{0}{5}

\twoe{0.7}{2}
\twoe{1.4}{6}
\twoed{1.4}{6}
\end{tikzpicture}
&
\gS{3c}%
&
\begin{tikzpicture}[baseline=(A),outer sep=0pt,inner sep=0pt]
\node (A) at (0,0) {};
\twoeARB{0}{0}{0.7}{0}{1}{above=1pt};
\twoeARB{0.7}{0}{1.4}{0}{3}{above=1pt};
\twoeARB{0.7}{0}{0.7}{-0.7}{4}{right=1pt};
\twoetr{0.7}{-0.7}{6}
\twoetl{0.7}{-0.7}{6}
\twoe{0}{5}
\twoed{0}{5}

\twoe{0.7}{2}
\twoe{1.4}{7}
\twoed{1.4}{7}
\end{tikzpicture}
\\ \ \\
\gS{3d}%
&
\begin{tikzpicture}[baseline=(A),outer sep=0pt,inner sep=0pt]
\node (A) at (0,0) {};
\twoeARB{0}{0}{0.7}{0}{1}{above=1pt};
\twoeARB{0.7}{0}{1.4}{0}{3}{above=1pt};
\twoeARB{0.7}{0}{0.7}{-0.7}{4}{right=1pt};
\twoetr{0.7}{-0.7}{7}
\twoetl{0.7}{-0.7}{5}
\twoe{0}{5}
\twoed{0}{6}

\twoe{0.7}{2}
\twoe{1.4}{8}
\twoed{1.4}{5}
\end{tikzpicture}
&
\gS{3e}%
&
\begin{tikzpicture}[baseline=(A),outer sep=0pt,inner sep=0pt]
\node (A) at (0,0) {};
\twoeARB{0}{0}{0.7}{0}{1}{above=1pt};
\twoeARB{0.7}{0}{1.4}{0}{3}{above=1pt};
\twoeARB{0.7}{0}{0.7}{-0.7}{4}{right=1pt};
\twoetr{0.7}{-0.7}{8}
\twoetl{0.7}{-0.7}{6}
\twoe{0}{5}
\twoed{0}{6}

\twoe{0.7}{2}
\twoe{1.4}{5}
\twoed{1.4}{7}
\end{tikzpicture}
&
\gS{3f}%
&
\begin{tikzpicture}[baseline=(A),outer sep=0pt,inner sep=0pt]
\node (A) at (0,0) {};
\twoeARB{0}{0}{0.7}{0}{1}{above=1pt};
\twoeARB{0.7}{0}{1.4}{0}{3}{above=1pt};
\twoeARB{0.7}{0}{0.7}{-0.7}{4}{right=1pt};
\twoetr{0.7}{-0.7}{7}
\twoetl{0.7}{-0.7}{5}
\twoe{0}{5}
\twoed{0}{6}

\twoe{0.7}{2}
\twoe{1.4}{9}
\twoed{1.4}{8}
\end{tikzpicture}
&
\gS{3g}%
&
\begin{tikzpicture}[baseline=(A),outer sep=0pt,inner sep=0pt]
\node (A) at (0,0) {};
\twoeARB{0}{0}{0.7}{0}{1}{above=1pt};
\twoeARB{0.7}{0}{1.4}{0}{3}{above=1pt};
\twoeARB{0.7}{0}{0.7}{-0.7}{4}{right=1pt};
\twoetr{0.7}{-0.7}{8}
\twoetl{0.7}{-0.7}{7}
\twoe{0}{5}
\twoed{0}{6}

\twoe{0.7}{2}
\twoe{1.4}{10}
\twoed{1.4}{9}
\end{tikzpicture}
\end{longtable}
\end{proposition}

\begin{definition}
We say that a vertex is of type $\Ver_{k_1\cdots k_r}^{l_1\cdots l_s}$, if it has a looping $k_i$-edge for every $1\leq i\leq r$ and has one connection to a $l_j$-edge  for every $1\leq j\leq s$.
\end{definition}

\begin{definition}
The \emph{virtual degree type} $d(\Gamma)$ of a graph $\Gamma$ with $k$ vertices is the descending sequence $(d_1,\ldots,d_k)$ of virtual degrees of vertices of $\Gamma$. We  define a partial order on the set of graphs for the action of $\SL_n$ with $k$ vertices by setting 
\begin{align*}
\Gamma < \Gamma'
&:\Leftrightarrow
(d_1,\ldots,d_k)=d(\Gamma)< d(\Gamma')=(d'_1,\ldots,d'_k)
\\
&:\Leftrightarrow
d_1 \leq d_1',\ldots,d_{k-1} \leq d_{k-1}',d_k<d_k'.
\end{align*} 
We call a graphsum $\sum \Gamma_i$ \emph{degree-reducible}, if it is reducibly equivalent either to $0$ or to a graphsum $\sum \Gamma'_j$ with $d(\Gamma'_j) < d(\Gamma_i)$ for all $i$, $j$. 
Moreover, in analogy to Definition~\ref{def:red}, we say that graphsums $\Upsilon_1,\ldots,\Upsilon_N$ are \emph{degree-reducibly independent}, if a linear combination $\sum a_i\Upsilon_i$ is degree-reducible only if all $a_i$ are equal to zero. If for two graphsums $\Upsilon_1$, $\Upsilon_2$ the linear combination $\Upsilon_1-\Upsilon_2$ is degree-reducible, we call them \emph{degree-reducibly equivalent} and write $\Upsilon_1\simeq_d\Upsilon_2$. 
 We say that a set of degree-reducibly independent degree-irreducible graphs has property $(DI)$.
\end{definition}

\begin{lemma}\label{le:DIRI}
A maximal set with property $(DI)$ is also a maximal set with property $(RI)$, i.e. a set of reducibly independent irreducible graphs.
\end{lemma}

\begin{proof}
Let $M$ be a maximal set with property $(DI)$. Of course, $M$ has property $(RI)$. Assume $M$ is not maximal with that property. Then there is a graph $\Gamma$, so that $M'=M \cup \{\Gamma\}$ still has property $(RI)$, but not $(DI)$. Thus  $\Gamma+ \sum a_i \Gamma_i$ is degree-reducible for some $\Gamma_i\in M$.  So $\Gamma+ \sum a_i \Gamma_i \simeq_r \sum b_j \Gamma_j'$ with reducibly independent irreducible $\Gamma_j'$ so that $d(\Gamma_j')<d(\Gamma)$ for all $j$. 
Since $M'$ has property $(RI)$, not all $\Gamma_j'$ can be elements of $M'$. 
Take $\Gamma_k' \notin M'$ and assume $M \cup \{\Gamma_k'\}$ does not have property $(RI)$. 
Since $\Gamma_k'$ is irreducible, there must be a reducible sum $\sum c_i \Gamma_i + \Gamma_k'$ with not all $c_i$ equal to zero. So we find
$$
\Gamma+ \sum (a_i+b_ic_i) \Gamma_i  \simeq_r \sum_{j\neq k} b_j \Gamma_j'
$$
Thus there must be some $\Gamma_l'$ so that  $M''=M \cup \{\Gamma_k'\}$ has property $(RI)$ and $d(\Gamma_k')<d(\Gamma)$. Since $(0,\ldots,0)$ is a lower bound for the virtual degree type, iterating this procedure gives a contradiction.
\end{proof}

\begin{proof}[Proof of Proposition~\ref{prop:SL52}]
Due to Lemma~\ref{le:DIRI}, we only have to consider degree-irreducible graphs. Due to Lemma~\ref{le:redvirtdeg}, such a graph can be assumed to have two types of vertices: such with one - type $\Ver_{2}$ - and such with two looping $2$-edges - type $\Ver_{22}$ - , being of virtual degree three and one respectively.
A graph with a multiple edge is not necessarily reducible but degree-reducible, due to
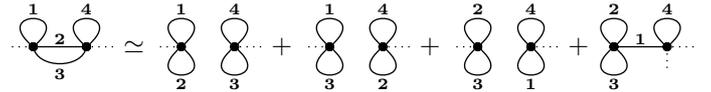
\begin{center}
\begin{tikzpicture}[baseline=(A),outer sep=0pt,inner sep=0pt]
\node (A) at (0,0) {};
\twoe{0.7}{4};
\twoe{0}{1};
\twoeABS{0}{0.7}{2};
\twoeABfd{0}{0.7}{3}{0.3};
\draw[line width=0.5pt, draw=black,dotted] (0,0)-- (-0.35,0);
\draw[line width=0.5pt, draw=black, dotted] (1.05,0)-- (0.7,0);
\end{tikzpicture}
\gS{\simeq}
\begin{tikzpicture}[baseline=(A),outer sep=0pt,inner sep=0pt]
\node (A) at (0,0) {};
\twoe{0.7}{4};
\twoed{0.7}{3};
\twoe{0}{1};
\twoed{0}{2};
\draw[line width=0.5pt, draw=black,dotted] (0,0)-- (-0.35,0);
\draw[line width=0.5pt, draw=black, dotted] (1.05,0)-- (0.7,0);
\end{tikzpicture}
\gS{+}
\begin{tikzpicture}[baseline=(A),outer sep=0pt,inner sep=0pt]
\node (A) at (0,0) {};
\twoe{0.7}{4};
\twoed{0.7}{2};
\twoe{0}{1};
\twoed{0}{3};
\draw[line width=0.5pt, draw=black,dotted] (0,0)-- (-0.35,0);
\draw[line width=0.5pt, draw=black, dotted] (1.05,0)-- (0.7,0);
\end{tikzpicture}
\gS{+}
\begin{tikzpicture}[baseline=(A),outer sep=0pt,inner sep=0pt]
\node (A) at (0,0) {};
\twoe{0.7}{4};
\twoed{0.7}{1};
\twoe{0}{2};
\twoed{0}{3};
\draw[line width=0.5pt, draw=black,dotted] (0,0)-- (-0.35,0);
\draw[line width=0.5pt, draw=black, dotted] (1.05,0)-- (0.7,0);
\end{tikzpicture}
\gS{+}
\begin{tikzpicture}[baseline=(A),outer sep=0pt,inner sep=0pt]
\node (A) at (0,0) {};
\twoe{0.7}{4};
\twoe{0}{2};
\twoed{0}{3};
\twoeABS{0}{0.7}{1};
\draw[line width=0.5pt, draw=black,dotted] (0.7,0)-- (0.7,-0.35);
\draw[line width=0.5pt, draw=black, dotted] (1.05,0)-- (0.7,0);
\end{tikzpicture}
\gES{.}
\end{center}

\noindent
So we can exclude such graphs as well. We call vertices with two looping edges \emph{black holes}, because they 'absorb colors' in the sense that we can not interchange the colors of  the two looping edges with other edges' colors in the way we are used to from the $\SL_4$-case. Colors can only be extracted if two of the adjacent edges have the same color:
\begin{center}
\begin{tikzpicture}[baseline=(A),outer sep=0pt,inner sep=0pt]
\node (A) at (0,0) {};
\twoe{0}{1};
\twoed{0}{1};
\twoeABS{0}{0.7}{2};

\draw[line width=0.5pt, draw=black, dotted] (1.05,0.35)-- (0.7,0);
\draw[line width=0.5pt, draw=black, dotted] (1.05,-0.35)-- (0.7,0);

\draw[line width=0.5pt, draw=black, dotted] (0.7,-0.35)-- (0.7,0);
\draw[line width=0.5pt, draw=black, dotted] (0.7,0.35)-- (0.7,0);
\end{tikzpicture}
\gS{+}
\gS{2}
\begin{tikzpicture}[baseline=(A),outer sep=0pt,inner sep=0pt]
\node (A) at (0,0) {};
\twoe{0}{1};
\twoed{0}{2};
\twoeABS{0}{0.7}{1};

\draw[line width=0.5pt, draw=black, dotted] (1.05,0.35)-- (0.7,0);
\draw[line width=0.5pt, draw=black, dotted] (1.05,-0.35)-- (0.7,0);

\draw[line width=0.5pt, draw=black, dotted] (0.7,-0.35)-- (0.7,0);
\draw[line width=0.5pt, draw=black, dotted] (0.7,0.35)-- (0.7,0);
\end{tikzpicture}
\gS{\simeq 0}
\end{center}
On the other hand, if all three adjacent edges have the same color, the graph evaluates to zero under $U\circ \gamma$. So we exclude this case as well and first let $\Gamma$ have two vertices, then it clearly is of the form $1$ from the proposition and we get the relevant colorings by evaluating all other non-equivalent colorings to zero.

Let now $\Gamma$ have four or more vertices and $\Gamma_2$ be the subgraph consisting of vertices of type $\mathcal{V}_2$ and all edges with a connection to one of these vertices. Let $\Gamma_{\sigma,2}$ be the graph $\Gamma$ with the colors inside $\Gamma_2$ permuted by $\sigma$. Then similar as in the $\SL_4$-case, but now with degree-reducibly equivalence, we get $\Gamma\simeq_d\sgn(\sigma)\Gamma_{\sigma,2}$, due to
\begin{center}
\begin{tikzpicture}[baseline=(A),outer sep=0pt,inner sep=0pt]
\node (A) at (0,0) {};
\twoe{0}{1};
\twoeABS{0}{0.7}{2};
\draw[line width=0.5pt, draw=black,dotted] (0,0)-- (-0.35,0);
\draw[line width=0.5pt, draw=black, dotted] (1.05,0.35)-- (0.7,0);
\draw[line width=0.5pt, draw=black, dotted] (1.05,-0.35)-- (0.7,0);
\draw[line width=0.5pt, draw=black,dotted] (0,0)-- (0,-0.35);
\draw[line width=0.5pt, draw=black, dotted] (0.7,-0.35)-- (0.7,0);
\draw[line width=0.5pt, draw=black, dotted] (0.7,0.35)-- (0.7,0);
\end{tikzpicture}
\gS{+}
\begin{tikzpicture}[baseline=(A),outer sep=0pt,inner sep=0pt]
\node (A) at (0,0) {};
\twoe{0}{2};
\twoeABS{0}{0.7}{1};
\draw[line width=0.5pt, draw=black,dotted] (0,0)-- (-0.35,0);
\draw[line width=0.5pt, draw=black, dotted] (1.05,0.35)-- (0.7,0);
\draw[line width=0.5pt, draw=black, dotted] (1.05,-0.35)-- (0.7,0);
\draw[line width=0.5pt, draw=black,dotted] (0,0)-- (0,-0.35);
\draw[line width=0.5pt, draw=black, dotted] (0.7,-0.35)-- (0.7,0);
\draw[line width=0.5pt, draw=black, dotted] (0.7,0.35)-- (0.7,0);
\end{tikzpicture}
\gS{+}
\gS{\sum}
\begin{tikzpicture}[baseline=(A),outer sep=0pt,inner sep=0pt]
\node (A) at (0,0) {};
\twoe{0}{1};
\twoed{0}{2};
\draw[line width=0.5pt, draw=black,dotted] (0,0)-- (-0.35,0);
\draw[line width=0.5pt, draw=black, dotted] (1.05,0.35)-- (0.7,0);
\draw[line width=0.5pt, draw=black, dotted] (1.05,-0.35)-- (0.7,0);
\draw[line width=0.5pt, draw=black, dotted] (0.35,0)-- (0.7,0);
\draw[line width=0.5pt, draw=black, dotted] (0.7,-0.35)-- (0.7,0);
\draw[line width=0.5pt, draw=black, dotted] (0.7,0.35)-- (0.7,0);
\draw[black, fill=black] (0.7,0) circle [radius=1.5pt];
\end{tikzpicture}
\gS{\simeq 0}
\gES{.}
\end{center}

\noindent
Thus we can use Proposition~\ref{prop:maxcolors} to conclude that either  $\Gamma$ has four vertices and $\Gamma_2$ ten edges or $\Gamma_2$ has  at most $9=\tbinom{5}{2}-1$ edges. In the first case, $\Gamma_{\eff}$ is the simple cubic connected graph $K_4$ and we find Graph $2$ from the proposition. 
In the second case, we distinguish between the number of vertices of $\Gamma_2$:

\medskip

\noindent
\textbf{Case 1: $\Gamma_2$ has one vertex.} Here $\Gamma$ must be of the form:
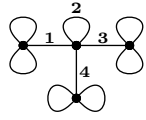
\begin{center}
 \begin{tikzpicture}[baseline=(A),outer sep=0pt,inner sep=0pt]
\node (A) at (0,0) {};
\twoeARB{0}{0}{0.7}{0}{1}{above=1pt};
\twoeARB{0.7}{0}{1.4}{0}{3}{above=1pt};
\twoeARB{0.7}{0}{0.7}{-0.7}{4}{right=1pt};
\twoetr{0.7}{-0.7}{}
\twoetl{0.7}{-0.7}{}
\twoe{0}{}
\twoed{0}{}
\twoe{0.7}{2}
\twoe{1.4}{}
\twoed{1.4}{}
\end{tikzpicture}
\end{center}

\noindent
If this graph is reducible for any coloring, it must be reducible for a coloring where the remaining looping edges are of colors $\iiE{1},\ldots,\iiE{9}$. If the two looping edges of one black hole are $\iiE{1},\ldots,\iiE{4}$-colored, the graph is reducible. If one looping edge of a black hole is $\iiE{1},\ldots,\iiE{4}$-colored, say $\iiE{1}$, and the other $\iiE{5},\ldots,\iiE{9}$-colored, by moving the $\iiE{1}$-colored edge of $\Gamma_2$ to the black hole, this graph is reducibly equivalent to the respective one with two looping edges of color $\iiE{1}$ at the black hole and one edge of color $\iiE{5}$ in $\Gamma_2$. So by swapping colors $\iiE{1}$ and $\iiE{5}$, we can assume that the looping edges of black holes are $\iiE{5},\ldots,\iiE{9}$-colored.
If the four looping edges of two black holes are colored with only one color, the graph is reducible by moving one of the colored edges of the first to the second black hole. If two black holes each have colored their looping edges  with the same two colors, say $\iiE{5}$ and $\iiE{6}$, then we get
\begin{longtable}{rl}
\begin{tikzpicture}[baseline=(A),outer sep=0pt,inner sep=0pt]
\node (A) at (0,0) {};
\twoe{0}{5};
\twoed{0}{6};
\twoe{0.7}{2};
\twoeARB{0.7}{0}{0.7}{-0.7}{4}{right=1pt};
\twoetr{0.7}{-0.7}{5}
\twoetl{0.7}{-0.7}{6}
\twoeABS{0}{0.7}{1}
\draw[line width=0.5pt, draw=black, dotted] (1.05,0)-- (0.7,0);
\end{tikzpicture}
&
\gS{\simeq_r}
\gS{-}
\begin{tikzpicture}[baseline=(A),outer sep=0pt,inner sep=0pt]
\node (A) at (0,0) {};
\twoe{0}{1};
\twoed{0}{6};
\twoe{0.7}{2};
\twoeARB{0.7}{0}{0.7}{-0.7}{4}{right=1pt};
\twoetr{0.7}{-0.7}{5}
\twoetl{0.7}{-0.7}{6}
\twoeABS{0}{0.7}{5}
\draw[line width=0.5pt, draw=black, dotted] (1.05,0)-- (0.7,0);
\end{tikzpicture}
\gS{-}
\begin{tikzpicture}[baseline=(A),outer sep=0pt,inner sep=0pt]
\node (A) at (0,0) {};
\twoe{0}{1};
\twoed{0}{5};
\twoe{0.7}{2};
\twoeARB{0.7}{0}{0.7}{-0.7}{4}{right=1pt};
\twoetr{0.7}{-0.7}{5}
\twoetl{0.7}{-0.7}{6}
\twoeABS{0}{0.7}{6}
\draw[line width=0.5pt, draw=black, dotted] (1.05,0)-- (0.7,0);
\end{tikzpicture}
\gS{\simeq_r}
\gS{\frac{1}{2}}
\begin{tikzpicture}[baseline=(A),outer sep=0pt,inner sep=0pt]
\node (A) at (0,0) {};
\twoe{0}{1};
\twoed{0}{6};
\twoe{0.7}{4};
\twoeARB{0.7}{0}{0.7}{-0.7}{6}{right=1pt};
\twoetr{0.7}{-0.7}{5}
\twoetl{0.7}{-0.7}{5}
\twoeABS{0}{0.7}{2}
\draw[line width=0.5pt, draw=black, dotted] (1.05,0)-- (0.7,0);
\end{tikzpicture}
\gS{+ \frac{1}{2}}
\begin{tikzpicture}[baseline=(A),outer sep=0pt,inner sep=0pt]
\node (A) at (0,0) {};
\twoe{0}{1};
\twoed{0}{5};
\twoe{0.7}{4};
\twoeARB{0.7}{0}{0.7}{-0.7}{5}{right=1pt};
\twoetr{0.7}{-0.7}{6}
\twoetl{0.7}{-0.7}{6}
\twoeABS{0}{0.7}{2}
\draw[line width=0.5pt, draw=black, dotted] (1.05,0)-- (0.7,0);
\end{tikzpicture}
\\ \ \\
&
\gS{\simeq_r}
\gS{\frac{1}{4}}
\begin{tikzpicture}[baseline=(A),outer sep=0pt,inner sep=0pt]
\node (A) at (0,0) {};
\twoe{0}{6};
\twoed{0}{6};
\twoe{0.7}{2};
\twoeARB{0.7}{0}{0.7}{-0.7}{4}{right=1pt};
\twoetr{0.7}{-0.7}{5}
\twoetl{0.7}{-0.7}{5}
\twoeABS{0}{0.7}{1}
\draw[line width=0.5pt, draw=black, dotted] (1.05,0)-- (0.7,0);
\end{tikzpicture}
\gS{+ \frac{1}{4}}
\begin{tikzpicture}[baseline=(A),outer sep=0pt,inner sep=0pt]
\node (A) at (0,0) {};
\twoe{0}{5};
\twoed{0}{5};
\twoe{0.7}{2};
\twoeARB{0.7}{0}{0.7}{-0.7}{4}{right=1pt};
\twoetr{0.7}{-0.7}{6}
\twoetl{0.7}{-0.7}{6}
\twoeABS{0}{0.7}{1}
\draw[line width=0.5pt, draw=black, dotted] (1.05,0)-- (0.7,0);
\end{tikzpicture}
\gES{,}
\end{longtable}
\noindent
and any such graph is reducible. The only remaining possible form for a $\iiE{1},\ldots,\iiE{6}$-colored graph is reducible as well:
\begin{center}
 \begin{tikzpicture}[baseline=(A),outer sep=0pt,inner sep=0pt]
\node (A) at (0,0) {};
\twoeARB{0}{0}{0.7}{0}{1}{above=1pt};
\twoeARB{0.7}{0}{1.4}{0}{3}{above=1pt};
\twoeARB{0.7}{0}{0.7}{-0.7}{4}{right=1pt};
\twoetr{0.7}{-0.7}{5}
\twoetl{0.7}{-0.7}{6}
\twoe{0}{5}
\twoed{0}{5}
\twoe{0.7}{2}
\twoe{1.4}{6}
\twoed{1.4}{6}
\end{tikzpicture}
\gS{\simeq}
\gS{-}
\begin{tikzpicture}[baseline=(A),outer sep=0pt,inner sep=0pt]
\node (A) at (0,0) {};
\twoeARB{0}{0}{0.7}{0}{1}{above=1pt};
\twoeARB{0.7}{0}{1.4}{0}{3}{above=1pt};
\twoeARB{0.7}{0}{0.7}{-0.7}{5}{right=1pt};
\twoetr{0.7}{-0.7}{4}
\twoetl{0.7}{-0.7}{6}
\twoe{0}{5}
\twoed{0}{5}
\twoe{0.7}{2}
\twoe{1.4}{6}
\twoed{1.4}{6}
\end{tikzpicture}
\gS{-}
\begin{tikzpicture}[baseline=(A),outer sep=0pt,inner sep=0pt]
\node (A) at (0,0) {};
\twoeARB{0}{0}{0.7}{0}{1}{above=1pt};
\twoeARB{0.7}{0}{1.4}{0}{3}{above=1pt};
\twoeARB{0.7}{0}{0.7}{-0.7}{6}{right=1pt};
\twoetr{0.7}{-0.7}{4}
\twoetl{0.7}{-0.7}{5}
\twoe{0}{5}
\twoed{0}{5}
\twoe{0.7}{2}
\twoe{1.4}{6}
\twoed{1.4}{6}
\end{tikzpicture}
\end{center}
\noindent
For graphs with more than six colors, we get the reducibly independent possibilities $3a$-$3g$.

\medskip

\noindent
\textbf{Case 2: $\Gamma_2$ has two vertices.} Then $\Gamma$ must have six and is of the form 
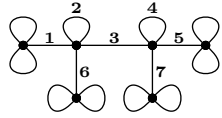
\begin{center}
\begin{tikzpicture}[baseline=(A),outer sep=0pt,inner sep=0pt]
\node (A) at (0,0) {};
\twoe{0}{};
\twoed{0}{};
\twoe{2.4}{};
\twoed{2.4}{};
\twoe{0.7}{2};
\twoe{1.7}{4};
\twoeARB{0.7}{0}{0.7}{-0.7}{6}{right=1pt};
\twoetr{0.7}{-0.7}{}
\twoetl{0.7}{-0.7}{}
\twoeARB{1.7}{0}{1.7}{-0.7}{7}{right=1pt};
\twoetr{1.7}{-0.7}{}
\twoetl{1.7}{-0.7}{}
\twoeABS{0.7}{1.7}{3}
\twoeABS{0}{0.7}{1}
\twoeABS{1.7}{2.4}{5}
\end{tikzpicture}
\end{center}
\noindent
Due to the considerations from the previous case, this graph must be irreducible for some coloring of the remaining edges with colors $\iiE{8},\iiE{9}$. But any such graph is reducible. 

\medskip

\noindent
\textbf{Case 3: $\Gamma_2$ has three vertices.} It either contains more than nine edges, or $\Gamma$ must be of the form
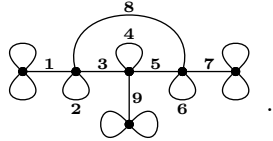
\begin{center}
\begin{tikzpicture}[baseline=(A),outer sep=0pt,inner sep=0pt]
\node (A) at (0,0) {};
\twoe{0}{};
\twoed{0}{};
\twoed{0.7}{2};
\twoeABS{0}{0.7}{1}
\twoeABS{0.7}{1.4}{3}
\twoe{1.4}{4};
\twoed{2.1}{6};
\twoeABS{1.4}{2.1}{5}
\twoeABS{2.1}{2.8}{7}
\twoe{2.8}{};
\twoed{2.8}{};
\twoeABfu{0.7}{2.1}{8}{1}
\twoeARB{1.4}{0}{1.4}{-0.7}{9}{right=1pt};
\twoetr{1.4}{-0.7}{}
\twoetl{1.4}{-0.7}{}
\end{tikzpicture}
\gES{.}
\end{center}
\noindent
But this graph is reducible for all the remaining edges colored with colors $\iiE{1},\ldots,\iiE{9}$. So it is reducible for any coloring.

\medskip

\noindent
\textbf{Case 4: $\Gamma_2$  has four or more vertices.} Then it contains more than nine edges, which is a contradiction. The proof is complete.
\end{proof}

\begin{proof}[Proof of Theorem~\ref{th:fftSL5}]
We only consider degree-irreducible graphs and do not list explicit colorings. We also can assume  that the number of $2$-edges is greater or equal to the number of $3$-edges due to the duality of $\BigWedge^k V$ and $\BigWedge^{n-k} V$.
First of all, we have vertices of the type $\Ver_{11111}$, these are a connected component.
Besides, there are vertices of virtual degree one of the types $\Ver_4$, $\Ver_{22}$, $\Ver_{13}$, and $\Ver_{112}$. We have vertices of virtual degree two of the types $\Ver_3$ and $\Ver_{21}$. Lastly, we have vertices of virtual degree three of the type $\Ver_2$. We have no multiple $2$-edges due to degree-irreducibility.

All possible graphs with one vertex are irreducible for a suitable choice of colors. Those with two either have a non-looping $2$-edge and any combination of the types $\Ver_4$, $\Ver_{22}$, $\Ver_{13}$, and $\Ver_{112}$, or they have two $2$-edges and two $3$-edges and due to degree-irreducibility, the form of such graph is unique.

So let $\Gamma$ have three or more vertices.
We can assume that there is no non-looping $4$-edge, since if there is one and it is only connected to vertices with looping $4$-edges, this constitues a connected component with mirror the graph with one vertex of the type $\Ver_{11111}$, on the other hand, if it is connected to a vertex without a looping $4$-edge, we can pull it over to this vertex. 

If there is a vertex of type $\Ver_2$, $\Ver_{12}$ or $\Ver_{112}$, then there is no non-looping $3$-edge due to degree-irreducibility. All $2$-edges but the looping ones of  black holes $\Ver_{22}$ can be permuted and we have $\Gamma_\sigma\simeq_d\sgn(\sigma)\Gamma$ as usual. 

\medskip

\noindent
\textbf{Case 1: $\Gamma$ has a vertex of type $\Ver_2$.}
Here in principle, all graphs stem from those from the proof of Proposition~\ref{prop:SL52}, with three possible modifications. Firstly, vertices of type $\Ver_{22}$ can be replaced by such of types $\Ver_{4}$, $\Ver_{13}$ or $\Ver_{112}$. Secondly, arms or cycles can be prolonged by inserting vertices of type $\Ver_3$ and $\Ver_{12}$, and lastly, two arms can be connected to a cycle by replacing the two 'end-vertices' with one vertex of type $\Ver_3$ or $\Ver_{12}$. The number of $2$-vertices here  is always bounded by Proposition~\ref{prop:maxcolors}. We can assume that vertices of types $\Ver_3$ and $\Ver_{12}$ only are on two of three sides of a vertex of type $\Ver_2$ by the following:
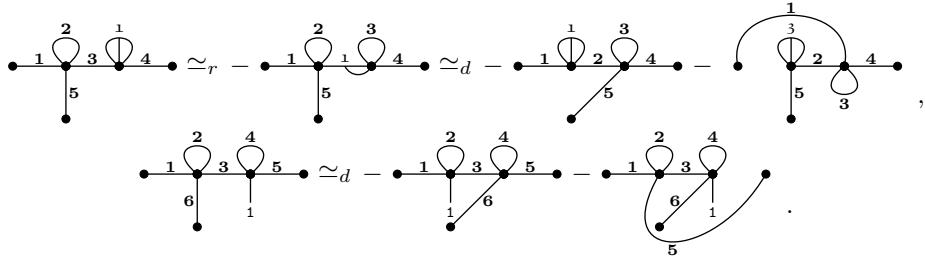
\begin{center}
\begin{tikzpicture}[baseline=(A),outer sep=0pt,inner sep=0pt]
\node (A) at (0,0) {};
\twoeABS{0}{0.7}{1};
\twoeABS{0.7}{1.4}{3};
\twoeABS{1.4}{2.1}{4};
\twoe{0.7}{2};
\thre{1.4}{1};
\twoeARB{0.7}{0}{0.7}{-0.7}{5}{right=1pt};
\end{tikzpicture}
\gS{\simeq_r}
\gS{-}
\begin{tikzpicture}[baseline=(A),outer sep=0pt,inner sep=0pt]
\node (A) at (0,0) {};
\twoeABS{0}{0.7}{1};
\twoeABS{1.4}{2.1}{4};
\twoe{0.7}{2};
\threABS{0.7}{1.4}{1}
\twoe{1.4}{3};
\twoeARB{0.7}{0}{0.7}{-0.7}{5}{right=1pt};
\end{tikzpicture}
\gS{\simeq_d}
\gS{-}
\begin{tikzpicture}[baseline=(A),outer sep=0pt,inner sep=0pt]
\node (A) at (0,0) {};
\twoeABS{0}{0.7}{1};
\twoeABS{1.4}{2.1}{4};
\thre{0.7}{1};
\twoeABS{0.7}{1.4}{2}
\twoe{1.4}{3};
\twoeARB{1.4}{0}{0.7}{-0.7}{5}{right=2pt};
\end{tikzpicture}
\gS{-}
\begin{tikzpicture}[baseline=(A),outer sep=0pt,inner sep=0pt]
\node (A) at (0,0) {};
\twoeABS{1.4}{2.1}{4};
\thre{0.7}{3};
\twoeABS{0.7}{1.4}{2}
\twoed{1.4}{3};
\twoeARB{0.7}{0}{0.7}{-0.7}{5}{right=1pt};
\twoeABfu{0}{1.4}{1}{0.9}
\end{tikzpicture}
\gES{,}
\\ \ \\ 
\begin{tikzpicture}[baseline=(A),outer sep=0pt,inner sep=0pt]
\node (A) at (0,0) {};
\twoeABS{0}{0.7}{1};
\twoeABS{0.7}{1.4}{3};
\twoeABS{1.4}{2.1}{5};
\twoe{0.7}{2};
\twoe{1.4}{4};
\onee{1.4}{1};
\twoeARB{0.7}{0}{0.7}{-0.7}{6}{left=1pt};
\end{tikzpicture}
\gS{\simeq_d}
\gS{-}
\begin{tikzpicture}[baseline=(A),outer sep=0pt,inner sep=0pt]
\node (A) at (0,0) {};
\twoeABS{0}{0.7}{1};
\twoeABS{0.7}{1.4}{3};
\twoeABS{1.4}{2.1}{5};
\twoe{0.7}{2};
\twoe{1.4}{4};
\onee{0.7}{1};
\twoeARB{1.4}{0}{0.7}{-0.7}{6}{right=2pt};
\end{tikzpicture}
\gS{-}
\begin{tikzpicture}[baseline=(A),outer sep=0pt,inner sep=0pt]
\node (A) at (0,0) {};
\twoeABS{0}{0.7}{1};
\twoeABS{0.7}{1.4}{3};
\draw[black, fill=black] (2.1,0) circle [radius=1.5pt];
\draw[line width=0.5pt, draw=black]  (2.1,0) .. controls (1.4,-1.2) and (0,-1.2) .. (0.7,0) node[font=\tiny, midway, below=1pt] {$\iiE{5}$} ;
\twoe{0.7}{2};
\twoe{1.4}{4};
\onee{1.4}{1};
\twoeARB{1.4}{0}{0.7}{-0.7}{6}{left=2pt};
\end{tikzpicture}
\gES{.}
\end{center}
Moreover, a graph with two vertices of types $\Ver_3$ and $\Ver_{12}$ joined by a $2$-edge and the graph with these two vertices swapped differ degree-reducibly by a graph with an additional vertex of type $\Ver_2$:
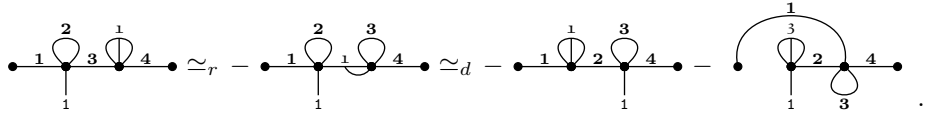
\begin{center}
\begin{tikzpicture}[baseline=(A),outer sep=0pt,inner sep=0pt]
\node (A) at (0,0) {};
\twoeABS{0}{0.7}{1};
\twoeABS{0.7}{1.4}{3};
\twoeABS{1.4}{2.1}{4};
\twoe{0.7}{2};
\onee{0.7}{1}
\thre{1.4}{1};
\end{tikzpicture}
\gS{\simeq_r}
\gS{-}
\begin{tikzpicture}[baseline=(A),outer sep=0pt,inner sep=0pt]
\node (A) at (0,0) {};
\twoeABS{0}{0.7}{1};
\twoeABS{1.4}{2.1}{4};
\twoe{0.7}{2};
\onee{0.7}{1}
\threABS{0.7}{1.4}{1}
\twoe{1.4}{3};
\end{tikzpicture}
\gS{\simeq_d}
\gS{-}
\begin{tikzpicture}[baseline=(A),outer sep=0pt,inner sep=0pt]
\node (A) at (0,0) {};
\twoeABS{0}{0.7}{1};
\twoeABS{1.4}{2.1}{4};
\thre{0.7}{1};
\twoeABS{0.7}{1.4}{2}
\twoe{1.4}{3};
\onee{1.4}{1}
\end{tikzpicture}
\gS{-}
\begin{tikzpicture}[baseline=(A),outer sep=0pt,inner sep=0pt]
\node (A) at (0,0) {};
\twoeABS{1.4}{2.1}{4};
\thre{0.7}{3};
\twoeABS{0.7}{1.4}{2}
\twoed{1.4}{3};
\onee{0.7}{1}
\twoeABfu{0}{1.4}{1}{0.9}
\end{tikzpicture}
\gES{.}
\end{center}
\noindent
This graph is either reducible or is considered in the list of generators as well, so we can in fact swap  two such vertices.
This directly leads to the graphs from the Theorem.

\medskip

\noindent
\textbf{Case 2: $\Gamma$ has no vertex of type $\Ver_2$ and no non-looping $3$-edge.}
In this case, the only non-looping edges are still those of size two. But now, we only have vertices of virtual degree one and two. Thus we have two types: chains and cycles.

\medskip

\noindent
\textbf{Case 3: $\Gamma$ has a non-looping $3$-edge.} We have no vertices of types $\Ver_2$, $\Ver_{12}$ or $\Ver_{112}$. Assume a non-looping $3$-edge of $\Gamma$ has two connections to one vertex, then at this vertex due to degree-irreducibility, there is a looping $3$-edge. At the second vertex connected to the non-looping $3$-edge, there must be a looping $3$-edge as well. Thus this part of the graph must have the looks
\begin{center}
\begin{tikzpicture}[baseline=(A),outer sep=0pt,inner sep=0pt]
\node (A) at (0,0) {};
\thre{0.7}{};
\thre{0}{};
\threABS{0}{0.7}{}
\draw[line width=0.5pt, draw=black, dotted] (-0.35,0)-- (0,0);
\end{tikzpicture}
\gES{.}
\end{center}
If $\Gamma$ has more than one cycle, the number of $3$-edges exceeds the number of $2$-edges. So first assume $\Gamma$ has no cycle. By
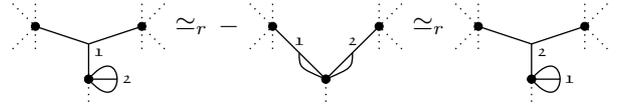
\begin{center}
\begin{tikzpicture}[baseline=(A),outer sep=0pt,inner sep=0pt]
\node (A) at (0,0) {};
\thretr{0.7}{-0.7}{2}
\threABC{0}{0}{1.4}{0}{0.7}{-0.7}{1}
\draw[line width=0.5pt, draw=black, dotted] (1.7,0.3)-- (1.4,0);
\draw[line width=0.5pt, draw=black, dotted] (1.7,-0.3)-- (1.4,0);
\draw[line width=0.5pt, draw=black, dotted] (1.4,-0.3)-- (1.4,0);
\draw[line width=0.5pt, draw=black, dotted] (1.4,0.3)-- (1.4,0);
\draw[line width=0.5pt, draw=black, dotted] (-0.3,0.3)-- (0,0);
\draw[line width=0.5pt, draw=black, dotted] (-0.3,-0.3)-- (0,0);
\draw[line width=0.5pt, draw=black, dotted] (0,-0.3)-- (0,0);
\draw[line width=0.5pt, draw=black, dotted] (0,0.3)-- (0,0);
\draw[line width=0.5pt, draw=black, dotted] (0.7,-1)-- (0.7,-0.7);
\end{tikzpicture}
\gS{\simeq_r}
\gS{-}
\begin{tikzpicture}[baseline=(A),outer sep=0pt,inner sep=0pt]
\node (A) at (0,0) {};
\threABARB{1.4}{0}{0.7}{-0.7}{2}{0}{-0.2}
\threABARB{0}{0}{0.7}{-0.7}{1}{0}{-0.2}
\draw[line width=0.5pt, draw=black, dotted] (1.7,0.3)-- (1.4,0);
\draw[line width=0.5pt, draw=black, dotted] (1.7,-0.3)-- (1.4,0);
\draw[line width=0.5pt, draw=black, dotted] (1.4,-0.3)-- (1.4,0);
\draw[line width=0.5pt, draw=black, dotted] (1.4,0.3)-- (1.4,0);
\draw[line width=0.5pt, draw=black, dotted] (-0.3,0.3)-- (0,0);
\draw[line width=0.5pt, draw=black, dotted] (-0.3,-0.3)-- (0,0);
\draw[line width=0.5pt, draw=black, dotted] (0,-0.3)-- (0,0);
\draw[line width=0.5pt, draw=black, dotted] (0,0.3)-- (0,0);
\draw[line width=0.5pt, draw=black, dotted] (0.7,-1)-- (0.7,-0.7);
\end{tikzpicture}
\gS{\simeq_r}
\begin{tikzpicture}[baseline=(A),outer sep=0pt,inner sep=0pt]
\node (A) at (0,0) {};
\thretr{0.7}{-0.7}{1}
\threABC{0}{0}{1.4}{0}{0.7}{-0.7}{2}
\draw[line width=0.5pt, draw=black, dotted] (1.7,0.3)-- (1.4,0);
\draw[line width=0.5pt, draw=black, dotted] (1.7,-0.3)-- (1.4,0);
\draw[line width=0.5pt, draw=black, dotted] (1.4,-0.3)-- (1.4,0);
\draw[line width=0.5pt, draw=black, dotted] (1.4,0.3)-- (1.4,0);
\draw[line width=0.5pt, draw=black, dotted] (-0.3,0.3)-- (0,0);
\draw[line width=0.5pt, draw=black, dotted] (-0.3,-0.3)-- (0,0);
\draw[line width=0.5pt, draw=black, dotted] (0,-0.3)-- (0,0);
\draw[line width=0.5pt, draw=black, dotted] (0,0.3)-- (0,0);
\draw[line width=0.5pt, draw=black, dotted] (0.7,-1)-- (0.7,-0.7);
\end{tikzpicture}
\gES{,}
\end{center}
\noindent 
we can swap  $3$-edges adjacent to a vertex. In fact, we see that if both have the same color, the graph evaluates to zero since changing shadings of two $k$-edges for odd $k$ results in reversed sign. Moreover, we can assume $2$-edges to be only on two sides of a non-looping $3$-edge due to 
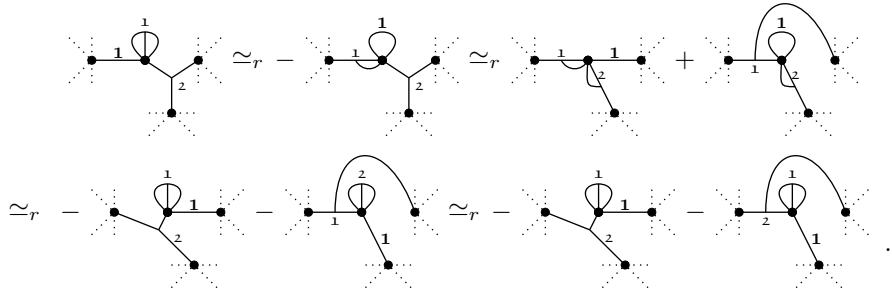
\begin{longtable}{rl}
&
\begin{tikzpicture}[baseline=(A),outer sep=0pt,inner sep=0pt]
\node (A) at (0,0) {};
\thre{0}{1}
\twoeABS{-0.7}{0}{1}
\threABC{0}{0}{0.7}{0}{0.35}{-0.7}{2}
\draw[line width=0.5pt, draw=black, dotted] (1,0.3)-- (0.7,0);
\draw[line width=0.5pt, draw=black, dotted] (1,-0.3)-- (0.7,0);
\draw[line width=0.5pt, draw=black, dotted] (0.7,-0.3)-- (0.7,0);
\draw[line width=0.5pt, draw=black, dotted] (0.7,0.3)-- (0.7,0);
\draw[line width=0.5pt, draw=black, dotted] (-1,0.3)-- (-0.7,0);
\draw[line width=0.5pt, draw=black, dotted] (-1,-0.3)-- (-0.7,0);
\draw[line width=0.5pt, draw=black, dotted] (-0.7,-0.3)-- (-0.7,0);
\draw[line width=0.5pt, draw=black, dotted] (-0.7,0.3)-- (-0.7,0);

\draw[line width=0.5pt, draw=black, dotted] (0.05,-0.7)-- (0.35,-0.7);
\draw[line width=0.5pt, draw=black, dotted] (0.65,-0.7)-- (0.35,-0.7);
\draw[line width=0.5pt, draw=black, dotted] (0.05,-1)-- (0.35,-0.7);
\draw[line width=0.5pt, draw=black, dotted] (0.65,-1)-- (0.35,-0.7);
\end{tikzpicture}
\gS{\simeq_r}
\gS{-}
\begin{tikzpicture}[baseline=(A),outer sep=0pt,inner sep=0pt]
\node (A) at (0,0) {};
\twoe{0}{1}
\threABS{-0.7}{0}{1}
\threABC{0}{0}{0.7}{0}{0.35}{-0.7}{2}
\draw[line width=0.5pt, draw=black, dotted] (1,0.3)-- (0.7,0);
\draw[line width=0.5pt, draw=black, dotted] (1,-0.3)-- (0.7,0);
\draw[line width=0.5pt, draw=black, dotted] (0.7,-0.3)-- (0.7,0);
\draw[line width=0.5pt, draw=black, dotted] (0.7,0.3)-- (0.7,0);
\draw[line width=0.5pt, draw=black, dotted] (-1,0.3)-- (-0.7,0);
\draw[line width=0.5pt, draw=black, dotted] (-1,-0.3)-- (-0.7,0);
\draw[line width=0.5pt, draw=black, dotted] (-0.7,-0.3)-- (-0.7,0);
\draw[line width=0.5pt, draw=black, dotted] (-0.7,0.3)-- (-0.7,0);

\draw[line width=0.5pt, draw=black, dotted] (0.05,-0.7)-- (0.35,-0.7);
\draw[line width=0.5pt, draw=black, dotted] (0.65,-0.7)-- (0.35,-0.7);
\draw[line width=0.5pt, draw=black, dotted] (0.05,-1)-- (0.35,-0.7);
\draw[line width=0.5pt, draw=black, dotted] (0.65,-1)-- (0.35,-0.7);
\end{tikzpicture}
\gS{\simeq_r}
\begin{tikzpicture}[baseline=(A),outer sep=0pt,inner sep=0pt]
\node (A) at (0,0) {};
\twoeABS{0}{0.7}{1}
\threABS{-0.7}{0}{1}
\threABARB{0.35}{-0.7}{0}{0}{2}{-0.2}{0}
\draw[line width=0.5pt, draw=black, dotted] (1,0.3)-- (0.7,0);
\draw[line width=0.5pt, draw=black, dotted] (1,-0.3)-- (0.7,0);
\draw[line width=0.5pt, draw=black, dotted] (0.7,-0.3)-- (0.7,0);
\draw[line width=0.5pt, draw=black, dotted] (0.7,0.3)-- (0.7,0);
\draw[line width=0.5pt, draw=black, dotted] (-1,0.3)-- (-0.7,0);
\draw[line width=0.5pt, draw=black, dotted] (-1,-0.3)-- (-0.7,0);
\draw[line width=0.5pt, draw=black, dotted] (-0.7,-0.3)-- (-0.7,0);
\draw[line width=0.5pt, draw=black, dotted] (-0.7,0.3)-- (-0.7,0);

\draw[line width=0.5pt, draw=black, dotted] (0.05,-0.7)-- (0.35,-0.7);
\draw[line width=0.5pt, draw=black, dotted] (0.65,-0.7)-- (0.35,-0.7);
\draw[line width=0.5pt, draw=black, dotted] (0.05,-1)-- (0.35,-0.7);
\draw[line width=0.5pt, draw=black, dotted] (0.65,-1)-- (0.35,-0.7);
\end{tikzpicture}
\gS{+}
\begin{tikzpicture}[baseline=(A),outer sep=0pt,inner sep=0pt]
\node (A) at (0,0) {};
\twoe{0}{1}
\draw[line width=0.5pt, draw=black]  (-0.35,0) .. controls (-0.35,1) and (0.35,1) .. (0.7,0);
\draw[line width=0.5pt, draw=black] (-0.7,0)-- (0,0) node[font=\tiny, midway, below=2pt] {$\iiiE{1}$};
\draw[black, fill=black] (0.7,0) circle [radius=1.5pt];
\draw[black, fill=black] (-0.7,0) circle [radius=1.5pt];
\threABARB{0.35}{-0.7}{0}{0}{2}{-0.2}{0}
\draw[line width=0.5pt, draw=black, dotted] (1,0.3)-- (0.7,0);
\draw[line width=0.5pt, draw=black, dotted] (1,-0.3)-- (0.7,0);
\draw[line width=0.5pt, draw=black, dotted] (0.7,-0.3)-- (0.7,0);
\draw[line width=0.5pt, draw=black, dotted] (0.7,0.3)-- (0.7,0);
\draw[line width=0.5pt, draw=black, dotted] (-1,0.3)-- (-0.7,0);
\draw[line width=0.5pt, draw=black, dotted] (-1,-0.3)-- (-0.7,0);
\draw[line width=0.5pt, draw=black, dotted] (-0.7,-0.3)-- (-0.7,0);
\draw[line width=0.5pt, draw=black, dotted] (-0.7,0.3)-- (-0.7,0);

\draw[line width=0.5pt, draw=black, dotted] (0.05,-0.7)-- (0.35,-0.7);
\draw[line width=0.5pt, draw=black, dotted] (0.65,-0.7)-- (0.35,-0.7);
\draw[line width=0.5pt, draw=black, dotted] (0.05,-1)-- (0.35,-0.7);
\draw[line width=0.5pt, draw=black, dotted] (0.65,-1)-- (0.35,-0.7);
\end{tikzpicture}
\\ 
\gS{\simeq_r}
&
\gS{-}
\begin{tikzpicture}[baseline=(A),outer sep=0pt,inner sep=0pt]
\node (A) at (0,0) {};
\thre{0}{1}
\twoeABS{0}{0.7}{1}
\threABC{-0.7}{0}{0}{0}{0.35}{-0.7}{2}
\draw[line width=0.5pt, draw=black, dotted] (1,0.3)-- (0.7,0);
\draw[line width=0.5pt, draw=black, dotted] (1,-0.3)-- (0.7,0);
\draw[line width=0.5pt, draw=black, dotted] (0.7,-0.3)-- (0.7,0);
\draw[line width=0.5pt, draw=black, dotted] (0.7,0.3)-- (0.7,0);
\draw[line width=0.5pt, draw=black, dotted] (-1,0.3)-- (-0.7,0);
\draw[line width=0.5pt, draw=black, dotted] (-1,-0.3)-- (-0.7,0);
\draw[line width=0.5pt, draw=black, dotted] (-0.7,-0.3)-- (-0.7,0);
\draw[line width=0.5pt, draw=black, dotted] (-0.7,0.3)-- (-0.7,0);

\draw[line width=0.5pt, draw=black, dotted] (0.05,-0.7)-- (0.35,-0.7);
\draw[line width=0.5pt, draw=black, dotted] (0.65,-0.7)-- (0.35,-0.7);
\draw[line width=0.5pt, draw=black, dotted] (0.05,-1)-- (0.35,-0.7);
\draw[line width=0.5pt, draw=black, dotted] (0.65,-1)-- (0.35,-0.7);
\end{tikzpicture}
\gS{-}
\begin{tikzpicture}[baseline=(A),outer sep=0pt,inner sep=0pt]
\node (A) at (0,0) {};
\thre{0}{2}
\draw[line width=0.5pt, draw=black]  (-0.35,0) .. controls (-0.35,1) and (0.35,1) .. (0.7,0);
\draw[line width=0.5pt, draw=black] (-0.7,0)-- (0,0) node[font=\tiny, midway, below=2pt] {$\iiiE{1}$};
\draw[black, fill=black] (0.7,0) circle [radius=1.5pt];
\draw[black, fill=black] (-0.7,0) circle [radius=1.5pt];
\twoeARB{0.35}{-0.7}{0}{0}{1}{right=2pt}
\draw[line width=0.5pt, draw=black, dotted] (1,0.3)-- (0.7,0);
\draw[line width=0.5pt, draw=black, dotted] (1,-0.3)-- (0.7,0);
\draw[line width=0.5pt, draw=black, dotted] (0.7,-0.3)-- (0.7,0);
\draw[line width=0.5pt, draw=black, dotted] (0.7,0.3)-- (0.7,0);
\draw[line width=0.5pt, draw=black, dotted] (-1,0.3)-- (-0.7,0);
\draw[line width=0.5pt, draw=black, dotted] (-1,-0.3)-- (-0.7,0);
\draw[line width=0.5pt, draw=black, dotted] (-0.7,-0.3)-- (-0.7,0);
\draw[line width=0.5pt, draw=black, dotted] (-0.7,0.3)-- (-0.7,0);

\draw[line width=0.5pt, draw=black, dotted] (0.05,-0.7)-- (0.35,-0.7);
\draw[line width=0.5pt, draw=black, dotted] (0.65,-0.7)-- (0.35,-0.7);
\draw[line width=0.5pt, draw=black, dotted] (0.05,-1)-- (0.35,-0.7);
\draw[line width=0.5pt, draw=black, dotted] (0.65,-1)-- (0.35,-0.7);
\end{tikzpicture}
\gS{\simeq_r}
\gS{-}
\begin{tikzpicture}[baseline=(A),outer sep=0pt,inner sep=0pt]
\node (A) at (0,0) {};
\thre{0}{1}
\twoeABS{0}{0.7}{1}
\threABC{-0.7}{0}{0}{0}{0.35}{-0.7}{2}
\draw[line width=0.5pt, draw=black, dotted] (1,0.3)-- (0.7,0);
\draw[line width=0.5pt, draw=black, dotted] (1,-0.3)-- (0.7,0);
\draw[line width=0.5pt, draw=black, dotted] (0.7,-0.3)-- (0.7,0);
\draw[line width=0.5pt, draw=black, dotted] (0.7,0.3)-- (0.7,0);
\draw[line width=0.5pt, draw=black, dotted] (-1,0.3)-- (-0.7,0);
\draw[line width=0.5pt, draw=black, dotted] (-1,-0.3)-- (-0.7,0);
\draw[line width=0.5pt, draw=black, dotted] (-0.7,-0.3)-- (-0.7,0);
\draw[line width=0.5pt, draw=black, dotted] (-0.7,0.3)-- (-0.7,0);

\draw[line width=0.5pt, draw=black, dotted] (0.05,-0.7)-- (0.35,-0.7);
\draw[line width=0.5pt, draw=black, dotted] (0.65,-0.7)-- (0.35,-0.7);
\draw[line width=0.5pt, draw=black, dotted] (0.05,-1)-- (0.35,-0.7);
\draw[line width=0.5pt, draw=black, dotted] (0.65,-1)-- (0.35,-0.7);
\end{tikzpicture}
\gS{-}
\begin{tikzpicture}[baseline=(A),outer sep=0pt,inner sep=0pt]
\node (A) at (0,0) {};
\thre{0}{1}
\draw[line width=0.5pt, draw=black]  (-0.35,0) .. controls (-0.35,1) and (0.35,1) .. (0.7,0);
\draw[line width=0.5pt, draw=black] (-0.7,0)-- (0,0) node[font=\tiny, midway, below=2pt] {$\iiiE{2}$};
\draw[black, fill=black] (0.7,0) circle [radius=1.5pt];
\draw[black, fill=black] (-0.7,0) circle [radius=1.5pt];
\twoeARB{0.35}{-0.7}{0}{0}{1}{right=2pt}
\draw[line width=0.5pt, draw=black, dotted] (1,0.3)-- (0.7,0);
\draw[line width=0.5pt, draw=black, dotted] (1,-0.3)-- (0.7,0);
\draw[line width=0.5pt, draw=black, dotted] (0.7,-0.3)-- (0.7,0);
\draw[line width=0.5pt, draw=black, dotted] (0.7,0.3)-- (0.7,0);
\draw[line width=0.5pt, draw=black, dotted] (-1,0.3)-- (-0.7,0);
\draw[line width=0.5pt, draw=black, dotted] (-1,-0.3)-- (-0.7,0);
\draw[line width=0.5pt, draw=black, dotted] (-0.7,-0.3)-- (-0.7,0);
\draw[line width=0.5pt, draw=black, dotted] (-0.7,0.3)-- (-0.7,0);

\draw[line width=0.5pt, draw=black, dotted] (0.05,-0.7)-- (0.35,-0.7);
\draw[line width=0.5pt, draw=black, dotted] (0.65,-0.7)-- (0.35,-0.7);
\draw[line width=0.5pt, draw=black, dotted] (0.05,-1)-- (0.35,-0.7);
\draw[line width=0.5pt, draw=black, dotted] (0.65,-1)-- (0.35,-0.7);
\end{tikzpicture}
\gES{.}
\end{longtable} 
\noindent
This is the mirrored version of the first equation from Case 1. The mirrored version of vertices $\Ver_{22}$ are the blocks 
\begin{tikzpicture}[baseline=(A),outer sep=0pt,inner sep=0pt]
\node (A) at (0,-0.1) {};
\vera{0}{0}{\alpha}
\end{tikzpicture}
from Theorem~\ref{th:fftSL5}. We conclude that the number of blocks 
\begin{tikzpicture}[baseline=(A),outer sep=0pt,inner sep=0pt]
\node (A) at (0,-0.1) {};
\vera{0}{0}{\alpha}
\end{tikzpicture}
plus the number of $3$-edges that are not part of a block 
\begin{tikzpicture}[baseline=(A),outer sep=0pt,inner sep=0pt]
\node (A) at (0,-0.1) {};
\vera{0}{0}{\alpha}
\end{tikzpicture}
 is less than or equal to nine. Only for one non-looping $3$-edge of such graph there can be non-looping $3$-edges on three sides, i.e. we have a 'star'. Here only three or four non-looping $3$-edges are possible. If we have a 'chain', up to five non-looping $3$-edges are possible.

Now assume $\Gamma$ has a cycle (of non-looping $3$-edges). Then all vertices of virtual degree one must be of type $\Ver_{22}$, otherwise the number of $3$-edges would exceed the number of $2$-edges. The cycle can be made up by two, three, or four $3$-edges, where the total number of non-looping $3$-edges is smaller or equal to four. We get the remaining graphs from the theorem. 

Finally consider the restrictions on numbers of edges. The last one - number of $3$-edges less than or equal number of $2$-edges - can be made since we are considering the mirror graphs as well. For all graphs with no non-looping $3$-edges the first two restrictions are due to the previous observations in Cases 1 and 2, and due to the last restriction as well as the form of the graphs. For graphs with looping $3$-edges, the same arguments hold with $2$- and $3$-edges interchanged.
\end{proof}

\section{Relations of  $\SL_4$ }
\label{sec:SL4rel}

\begin{example}\label{ex:VV*}
Consider graphs no. 1 and 2 from Proposition~\ref{prop:maxset}. By applying the Pl\"ucker relation from Lemma~\ref{le:exch} three times, we can \emph{pull over} the $3$-edge of color $\iiiE{4}$ and get the well known - see~\cite[p. 255]{PV} - relation:
\begin{longtable}{rl}
&
\gA{1}{2}{3}{4}\gB{1}{2}{3}{4}%
\gEq%
\gS{-}%
\gSu{(\iE{ijk,l})\vdash(\iE{1234})}{}%
\begin{tikzpicture}[baseline=(A),outer sep=0pt,inner sep=0pt]
\node (A) at (0,0) {};
\node[font=\tiny,align=center] at (0,-0.5) {$\iiiE{4}$};
\draw[line width=0.5pt, draw=black] (0,0)-- (0,-0.38);
\draw[black, fill=black] (0,0) circle [radius=1.5pt];
\draw[black, fill=black] (0.7,0) circle [radius=1.5pt];
\draw[black, fill=black] (-0.7,0) circle [radius=1.5pt];
\thre{0}{1}
\thre{0.7}{2}
\thre{1.4}{3}
\draw[line width=0.5pt, draw=black]  (0.7,0) .. controls (0.6,-0.5) and (-0.6,-0.5) .. (-0.7,0) ;
\node[font=\tiny,align=center] at (-0.7,0.5) {$\iE{i}$};
\node[font=\tiny,align=right] at (-1.2,0) {$\iE{j}$};
\draw[line width=0.5pt, draw=black] (-1.08,0)-- (-0.7,0);
\node[font=\tiny,align=center] at (-0.7,-0.5) {$\iE{k}$};
\draw[line width=0.5pt, draw=black] (0,0.38)-- (0,-0.38);
\draw[line width=0.5pt, draw=black] (-0.7,0.38)-- (-0.7,-0.38);
\draw[line width=0.5pt, draw=black] (1.4,0)-- (1.4,-0.38);
\node[font=\tiny,align=center] at (1.4,-0.5) {$\iE{l}$};
\draw[black, fill=black] (1.4,0) circle [radius=1.5pt];
\end{tikzpicture}
\\ \ \\
\end{longtable}
\end{example}

\begin{proof}[Proof of Theorem~\ref{th:SL4rel}]
We identify four somewhat \emph{natural principles} generating relations between invariants. All relations from Theorem~\ref{th:SL4rel} stem from these principles.

\medskip

\noindent
\textbf{Principle 1a: permuting five $1$-edges.} This comes from applying the Pl\"ucker relation from Lemma~\ref{le:exch} on five $1$-edges \emph{once}. Let $\Gamma$ be an arbitrary graph, connected to a $1$-edge, then
\begin{center}
\gSu{(\iE{ijkl,m})\vdash(\iE{12345})}{}\gA{i}{j}{k}{l} 
%
%
\gEq\gS{0}%
\gES{.}
\end{center}

\medskip

\noindent
\textbf{Principle 1b: pulling over a $k$-edge.}
If we have a product of graph no. $1$ and a second graph with a $2$- or $3$-edge, we can \emph{pull this $k$-edge over to graph no. $1$ and distribute a total of k $1$-edges to the second graph}. There are at most $k$ applications of Lemma~\ref{le:exch} necessary. 
In the case of a non-looping $2$-edge, we get
\begin{center}
\gA{1}{2}{3}{4}%
%
%
\gES{,}
\end{center}
where $\Gamma$ must not be connected. In the case of a looping $2$-edge, we have
\begin{center}
\gA{1}{2}{3}{4}%
%
%
\gES{,}
\end{center}
which is fine if $\Gamma$ is a looping $2$-edge. If not, then the vertex connected to the $2$-edge of color $\iiE{1}$ is connected to another, but non-looping, $2$-edge - say of color $\iiE{2}$ - and 
\begin{center}
\gA{1}{2}{3}{4}%
%
\end{center}
holds. 
Now in the case of a $3$-edge \emph{pulled over} to graph no. 1, Example~\ref{ex:VV*} shows what happens for the non-looping $3$-edge of graph no. 2. Since we can interchange the edges of graph no. 2 as we want by a simple application of Lemma~\ref{le:exch}, we get the same relation for one of the looping $3$-edges.  Thus we can exclude graph no. 2 in the following, which means that we only have to consider graphs with a looping $3$-edge  on a vertex connected to a $2$-edge (of color $\iiE{1}$). Here we get 
\begin{longtable}{rl}
&
\gA{1}{2}{3}{4}%
\begin{tikzpicture}[baseline=(A),outer sep=0pt,inner sep=0pt]
\thre{0}{1}
\draw[line width=0.5pt, draw=black] (0.7,0)-- (0,0);
\draw[black, fill=black] (0,0) circle [radius=1.5pt];
\draw[black, fill=black] (0.7,0) circle [radius=1.5pt];
\node[font=\tiny,align=center] at (0.35,0.1) {$\iiE{1}$};
\gra{0.7}
\end{tikzpicture}%
\gEq%
\gS{-\frac{1}{3}}%
\gSu{(\iE{ijk,l})\vdash(\iE{1234})}{}%
\begin{tikzpicture}[baseline=(A),outer sep=0pt,inner sep=0pt]
\node (A) at (0,0) {};
\node[font=\tiny,align=center] at (0,0.5) {$\iE{i}$};
\node[font=\tiny,align=right] at (-0.5,0) {$\iE{k}$};
\node[font=\tiny,align=right] at (0.7,-0.5) {$\iE{l}$};
\draw[line width=0.5pt, draw=black] (0,0)-- (-0.38,0);
\draw[line width=0.5pt, draw=black] (0.7,0)-- (1.4,0);

\node[font=\tiny,align=center] at (0,-0.5) {$\iE{j}$};
\draw[line width=0.5pt, draw=black] (0,0.38)-- (0,-0.38);
\draw[line width=0.5pt, draw=black] (0.7,0)-- (0.7,-0.38);

\draw[line width=0.5pt, draw=black] (1.4,0)-- (1.4,0.38);
\draw[black, fill=black] (0,0) circle [radius=1.5pt];
\draw[black, fill=black] (0.7,0) circle [radius=1.5pt];
\draw[black, fill=black] (1.4,0) circle [radius=1.5pt];
\node[font=\tiny,align=center] at (1.05,0.1) {$\iiE{1}$};
\gra{1.4}

\draw[line width=0.5pt, draw=black] (0.7,0)-- (0.7,0.38);
\draw[line width=0.5pt, draw=black]  (0,0) .. controls (0.1,0.5) and (1.3,0.5) .. (0.7,0) ;
\node[font=\tiny,align=center] at (0.7,0.5) {$\iiiE{1}$};
\end{tikzpicture}%
\ \\
\end{longtable}

\noindent
Observe that Principle 1a can be seen as a special case of Principle 1b, namely pulling over a $1$-edge.

\medskip

\noindent
\textbf{Principle 2: bringing together two $1$-edges (of two different graphs).} This only works if one of the $1$-edges is connected to a vertex with one looping $2$-edge and one non-looping $2$-edge. So any combination of graphs no. 1, 5, and 6 gives no or no new relation. If one of the graphs is graph no. 1, we can reduce to Principle 1 or 2. So we exclude this as well and get
\begin{longtable}{rl}
&
\begin{tikzpicture}[baseline=(A),outer sep=0pt,inner sep=0pt]
\node (A) at (0,0) {};
\onee{0}{1}
\draw[line width=0.5pt, draw=black] (0,0.38)-- (0,0);
\draw[black, fill=black] (0,0) circle [radius=1.5pt];
\grb{0}{\Gamma_1}
\end{tikzpicture}
\begin{tikzpicture}[baseline=(A),outer sep=0pt,inner sep=0pt]
\node (A) at (0,0) {};
\twoe{0}{1}
\onee{0}{2}
\draw[line width=0.5pt, draw=black] (0.7,0)-- (0,0);
\draw[line width=0.5pt, draw=black] (0.7,0)-- (0.7,0.38);
\draw[black, fill=black] (0,0) circle [radius=1.5pt];
\draw[black, fill=black] (0.7,0) circle [radius=1.5pt];
\node[font=\tiny,align=center] at (0.35,0.1) {$\iiE{2}$};
\grb{0.7}{\Gamma_2}
\end{tikzpicture}%
\gS{+}
\begin{tikzpicture}[baseline=(A),outer sep=0pt,inner sep=0pt]
\node (A) at (0,0) {};
\onee{0}{2}
\draw[line width=0.5pt, draw=black] (0,0.38)-- (0,0);
\draw[black, fill=black] (0,0) circle [radius=1.5pt];
\grb{0}{\Gamma_1}
\end{tikzpicture}
\begin{tikzpicture}[baseline=(A),outer sep=0pt,inner sep=0pt]
\node (A) at (0,0) {};
\twoe{0}{1}
\onee{0}{1}
\draw[line width=0.5pt, draw=black] (0.7,0)-- (0,0);
\draw[line width=0.5pt, draw=black] (0.7,0)-- (0.7,0.38);
\draw[black, fill=black] (0,0) circle [radius=1.5pt];
\draw[black, fill=black] (0.7,0) circle [radius=1.5pt];
\node[font=\tiny,align=center] at (0.35,0.1) {$\iiE{2}$};
\grb{0.7}{\Gamma_2}
\end{tikzpicture}%
\gS{+}%
\gF{1}{1}{2}%
\begin{tikzpicture}[baseline=(A),outer sep=0pt,inner sep=0pt]
\node (A) at (0,0) {};

\draw[line width=0.5pt, draw=black] (0.7,0)-- (0,0);
\draw[black, fill=black] (0,0) circle [radius=1.5pt];
\draw[black, fill=black] (0.7,0) circle [radius=1.5pt];
\node[font=\tiny,align=center] at (0.35,0.1) {$\iiE{2}$};
\draw[line width=0.5pt, draw=black] (0.7,0)-- (0.7,0.38);
\draw[line width=0.5pt, draw=black] (0,0)-- (0,0.38);
\grb{0}{\Gamma_1}
\grb{0.7}{\Gamma_2}
\end{tikzpicture}%
\gEq%
\gS{-}%
\begin{tikzpicture}[baseline=(A),outer sep=0pt,inner sep=0pt]
\node (A) at (0,0) {};
\oneeu{0.7}{1}
\onee{0.7}{2}
\draw[line width=0.5pt, draw=black] (1.4,0)-- (0,0);
\draw[black, fill=black] (0,0) circle [radius=1.5pt];
\draw[black, fill=black] (0.7,0) circle [radius=1.5pt];
\draw[black, fill=black] (1.4,0) circle [radius=1.5pt];
\node[font=\tiny,align=center] at (0.35,0.1) {$\iiE{1}$};
\node[font=\tiny,align=center] at (1.05,0.1) {$\iiE{2}$};
\draw[line width=0.5pt, draw=black] (0,0)-- (0,0.38);
\draw[line width=0.5pt, draw=black] (1.4,0)-- (1.4,0.38);
\grb{0}{\Gamma_1}
\grb{1.4}{\Gamma_2}
\end{tikzpicture}%
\\
\end{longtable}

\medskip

\noindent
\textbf{Principle 3: bringing together three $1$-edges.}
Here as well one of the $1$-edges must be connected to a vertex with a looping and a non-looping $2$-edge and none of the graphs must be no. 1. We have
\begin{longtable}{rl}
&
\gSu{(\iE{i,j,k})\vdash(\iE{123})}{}%
\begin{tikzpicture}[baseline=(A),outer sep=0pt,inner sep=0pt]
\node (A) at (0,0) {};
\onee{0}{i}
\draw[line width=0.5pt, draw=black] (0,0.38)-- (0,0);
\draw[black, fill=black] (0,0) circle [radius=1.5pt];
\grb{0}{\Gamma_1}
\end{tikzpicture}
\begin{tikzpicture}[baseline=(A),outer sep=0pt,inner sep=0pt]
\node (A) at (0,0) {};
\onee{0}{j}
\draw[line width=0.5pt, draw=black] (0,0.38)-- (0,0);
\draw[black, fill=black] (0,0) circle [radius=1.5pt];
\grb{0}{\Gamma_2}
\end{tikzpicture}
\begin{tikzpicture}[baseline=(A),outer sep=0pt,inner sep=0pt]
\node (A) at (0,0) {};
\twoe{0}{1}
\onee{0}{k}
\draw[line width=0.5pt, draw=black] (0.7,0)-- (0,0);
\draw[line width=0.5pt, draw=black] (0.7,0)-- (0.7,0.38);
\draw[black, fill=black] (0,0) circle [radius=1.5pt];
\draw[black, fill=black] (0.7,0) circle [radius=1.5pt];
\node[font=\tiny,align=center] at (0.35,0.1) {$\iiE{2}$};
\grb{0.7}{\Gamma_3}
\end{tikzpicture}%
\gS{+}%
\gSu{(\iE{i,jk})\vdash(\iE{123})}{}%
\begin{tikzpicture}[baseline=(A),outer sep=0pt,inner sep=0pt]
\node (A) at (0,0) {};
\onee{0}{i}
\draw[line width=0.5pt, draw=black] (0,0.38)-- (0,0);
\draw[black, fill=black] (0,0) circle [radius=1.5pt];
\grb{0}{\Gamma_1}
\end{tikzpicture}
\gF{1}{j}{k}%
\begin{tikzpicture}[baseline=(A),outer sep=0pt,inner sep=0pt]
\node (A) at (0,0) {};

\draw[line width=0.5pt, draw=black] (0.7,0)-- (0,0);
\draw[black, fill=black] (0,0) circle [radius=1.5pt];
\draw[black, fill=black] (0.7,0) circle [radius=1.5pt];
\node[font=\tiny,align=center] at (0.35,0.1) {$\iiE{2}$};
\draw[line width=0.5pt, draw=black] (0.7,0)-- (0.7,0.38);
\draw[line width=0.5pt, draw=black] (0,0)-- (0,0.38);
\grb{0}{\Gamma_2}
\grb{0.7}{\Gamma_3}
\end{tikzpicture}%
\\ \ \\
\end{longtable}

\medskip

\noindent
\textbf{Principle 4a: going around circular graphs.}
We get the second identity in the following by going around the circular graph with four vertices first with the $2$-edge of color $\iiE{1}$ and then with the $2$-edge of color $\iiE{2}$:
\begin{longtable}{rl}
&
\gCb{1}{2}%
\gD{3}{4}{5}{6}{7}{8}%
\gS{+}%
\gDf{1}{2}{3}{4}{5}{6}{7}{8}%
\gEq%
\gS{-}%
\gDf{2}{1}{3}{4}{5}{6}{7}{8}%
\\ \ \\
\gEq
&
\gSu{(\iiE{i},\iiE{jklmn})\vdash(\iiE{345678})}{}%
\gS{\Biggl((-1)^{i+1}}%
\gCb{1}{i}%
\gD{2}{j}{k}{l}{m}{n}%
\gS{\Biggr)+}%
\gCb{1}{2}%
\gD{8}{3}{4}{5}{6}{7}%
\\ \ \\
&
\gS{+}
\gSu{(\iiE{i},\iiE{jklmn})\vdash(\iiE{345678})}{}%
\gS{\Biggl((-1)^{i+1}}%
\gCb{2}{i}%
\gD{1}{j}{k}{l}{m}{n}%
\gS{\Biggr)+}%
\gDf{1}{2}{3}{4}{5}{6}{7}{8}%
\gES{.}
\end{longtable}
\noindent
Thus we get
\begin{longtable}{rl}
&
\gS{2}
\gCb{1}{2}%
\gS{\Biggl(}%
\gD{3}{4}{5}{6}{7}{8}%
\gS{-}%
\gD{8}{3}{4}{5}{6}{7}%
\gS{\Biggr)}%
\\ \ \\
\gEq
&
\gSu{(\iiE{i},\iiE{jklmn})\vdash(\iiE{345678})}{}%
\gS{(-1)^{i+1}\Biggl(}%
\gCb{1}{i}%
\gD{2}{j}{k}{l}{m}{n}%
\gS{+}%
\gCb{1}{i}%
\gD{2}{n}{j}{k}{l}{m}%
\\ \ \\
&
\gS{+}%
\gCb{2}{i}%
\gD{1}{j}{k}{l}{m}{n}%
\gS{+}
\gCb{2}{i}%
\gD{1}{n}{j}{k}{l}{m}%
\gS{\Biggr)}
\\ \ \\
\gEq
&
\gSu{(\iiE{g},\iiE{h})\vdash(\iiE{12})}{}%
\gSu{(\iiE{i},\iiE{jklmn})\vdash(\iiE{345678})}{}%
\gS{(-1)^{i+1}}%
\gCb{g}{i}%
\gS{\Biggl(}%
\gD{h}{j}{k}{l}{m}{n}%
\gS{-}%
\gD{n}{h}{j}{k}{l}{m}%
\gS{\Biggr)}%
\\ \ \\
&
\gS{+}%
\gSu{(\iiE{i},\iiE{jk},\iiE{l})\vdash(\iiE{3456})}{}%
\gS{\sgn(\vdash)\Biggl(}%
\gCb{1}{8}%
\gCb{2}{i}%
\gCb{j}{k}%
\gCb{l}{7}%
\gS{+}%
\gCb{2}{8}%
\gCb{1}{i}%
\gCb{j}{k}%
\gCb{l}{7}%
\gS{\Biggr)}%
\gES{.}
\end{longtable}

\medskip

\noindent
\textbf{Principle 4b: determinantal relations.}
Consider the matrices 
\renewcommand{\arraystretch}{0.85} 
\begin{align*}
B:=&
%
.
\end{align*}
We have $\det(A)=\INV{\iiE{i_1\cdots i_6}}$, $\det(B)=-\INV{\iiE{j_1\cdots j_6}}$ and  $A^\mathrm{T}B=C$. Thus 
$$\det(C)+ \INV{\iiE{i_1\cdots i_6}}\INV{\iiE{i_1\cdots i_6}}=0$$ 
holds. Moreover, we have the standard Pl\"ucker identity for determinants of matrices of the form $A$:
$$
\sum\limits_{(i_1\cdots i_6,j_1)\vdash(k_1\ldots k_7)} \INV{\iiE{i_1\cdots i_6}}\INV{\iiE{j_1\cdots j_6}}=0.
$$
Both of the above identities could also be achieved via 'going around circular graphs', which turns out to be a lot harder as this approach. Since we found all relations from Theorem~\ref{th:SL4rel}, the proof is complete.
\end{proof}

\section{An outlook}\label{sec:outlook}

In this section, we want to give a short compendium of possible further applications of our method. 
First of all, at least for small $n$, some $n_i=0$ and aid of computers, determination of generating sets for $\CC[W]^{\SL_n}$ seems to be possible. 

On the other hand, as Section~\ref{sec:SL4rel} shows, at least our method provides some intuitive processes to generate relations, while showing (in general) that these generate the ideal of relations requires more and possibly totally different considerations.

Of course our method is not restricted to antisymmetric tensors. As we mentioned in the introduction, related methods have been applied to binary forms. On the other hand, the symbolic method elaborated by Grosshans, Rota and Stein is able to deal with combinations of symmetric and antisymmetric tensors, so it seems likely to apply our graph method to such combinations. In this case, edges of different 'behaviour' would correspond to symmetric or antisymmetric tensors respectively. The presumably easiest nontrivial case would be that of $\SL_3$ acting on symmetric and antisymmetric $2$-tensors. 

Another direction for generalization is that of changing the acting group. The classical groups $\mathrm{SO}_n$ and $\mathrm{Sp}_n$ for example  have a \emph{principal tensor} $g$ besides $\det$. This tensor is an (anti-)symmetric bilinear form, see~\cite[\S 9.5]{PV}. It could be represented by vertices of degree two that now behave differently than the vertices of degree $n$ corresponding to $\det$.



\begin{thebibliography}{10}

\bibitem{app}
Lukas Braun, \textit{Appendix to 'Invariant rings of sums of fundamental representations of $\SL_n$ and colored hypergraphs'}.

\bibitem{BP1}
A. E. Brouwer, M. I. Popoviciu,
\textit{The invariants of the binary nonic}, J. Symbolic Comput. \textbf{45} (2010), no. 6, 709-720.

\bibitem{BP2}
A. E. Brouwer, M. I. Popoviciu,
\textit{The invariants of the binary decimic},  J. Symbolic Comput. 45 (2010), no. 8, 837-843.

\bibitem{chan}
Wendy Chan, 
\textit{Classification of trivectors in 6-D space}, Mathematical essays in honor of Gian-Carlo Rota (Cambridge, MA, 1996), 63-110, Progr. Math. \textbf{161}, Birkh\"auser Boston, Boston, MA, 1998.

\bibitem{crapo}
Henry H. Crapo,
\textit{Ten abandoned gold mines}, Algebraic combinatorics and computer science, 3-22, Springer Italia, Milan, 2001.

\bibitem{KD}
H. Derksen and G. Kemper, \textit{Computational invariant theory}, Invariant Theory and Algebraic Transformation Groups I, Encyclopaedia of Mathematical Sciences \textbf{130}, Springer-Verlag, Berlin, 2002.

\bibitem{dolga}
Igor Dolgachev, 
\textit{Lectures on invariant theory},
London Mathematical Society Lecture Note Series, \textbf{296}, Cambridge University Press, Cambridge, 2003.

\bibitem{MD}
 Mihaela I. Popoviciu Draisma, \textit{Invariants of binary forms}, Ph.D. Thesis, Universit\"at Basel, 2014, available at https://edoc.unibas.ch/33424/1/thesis\_popoviciudraisma.pdf.

\bibitem{GW}
R. Goodman and N. R. Wallach, \textit{Symmetry, representations, and invariants}, Graduate Texts in Mathematics \textbf{255}, Springer, Dordrecht, 2009.

\bibitem{GRS}
F. D Grosshans, G. C.  Rota, and J. A. Stein,  \textit{Invariant theory and superalgebras}, CBMS Regional Conference Series in Mathematics \textbf{69}, American Mathematical Society, Providence, RI, 1987.

\bibitem{Gur}
Grigorij B. Gurevich, \textit{Foundations of the theory of algebraic invariants}, Noordhoff, Groningen, 1964.

\bibitem{HMSV1}
B. Howard, J. Millson, A. Snowden, and R. Vakil, \textit{The equations for the moduli space of n points on the line}, Duke Math. J. \textbf{146} (2009), no. 2,  175-226.

\bibitem{HMSV2}
B. Howard, J. Millson, A. Snowden, and R. Vakil, \textit{The relations among invariants of points on the projective line},  C. R. Math. Acad. Sci. Paris \textbf{347} (2009), no. 19-20, 1177-1182.

\bibitem{HMSV3}
B. Howard, J. Millson, A. Snowden, and R. Vakil, \textit{The ideal of relations for the ring of invariants of n points on the line},   J. Eur. Math. Soc. (JEMS) \textbf{14} (2012), no. 1, 1-60.

\bibitem{HMSV4}
B. Howard, J. Millson, A. Snowden, and R. Vakil, \textit{The ideal of relations for the ring of invariants of n points on the line: integrality results}, Comm. Algebra \textbf{40} (2012), no. 10, 3884-3902.

\bibitem{HH}
R. Howe, R. Huang,
\textit{Projective invariants of four subspaces}, 
Adv. Math. \textbf{118} (1996), no. 2, 295-336.


\bibitem{Hu1}
Rosa Huang,
\textit{Invariants of sets of linear varieties},
Proc. Nat. Acad. Sci. U.S.A. \textbf{87} (1990), no. 12, 4557-4560. 

\bibitem{Hu}
Rosa Huang,
\textit{Invariants of sets of lines in projective 3-space},
J. Algebra \textbf{143} (1991), no. 1, 208-218.


\bibitem{kempe}
A. Kempe, \textit{On regular difference terms}, Proc. London Math. Soc. \textbf{25} (1894), 343-350.





\bibitem{KP}
H.~Kraft, C.~Procesi, \textit{Classical Invariant Theory. A Primer}, 1996, available at  
https://www2.bc.edu/benjamin-howard/MATH8845/classical\_invariant\_theory.pdf.


\bibitem{LO}
R. Lercier, M. Olive,
\textit{Covariant algebra of the binary nonic and the binary decimic}, Contemp. Math. \textbf{686} (2017),  65-91.

\bibitem{mcmillan}
Timothy R. McMillan III, \textit{Invariants of antisymmetric tensors}, Ph.D. Thesis, University of Florida, 1990, available at https://archive.org/details/invariantsofanti00mcmi.

\bibitem{olive}
Marc Olive, \textit{About Gordan's algorithm for binary forms}, Found. Comput. Math. \textbf{17} (2017), no. 6, 1407-1466.

\bibitem{PO}
Peter J. Olver,
\textit{Classical invariant theory},
 London Mathematical Society Student Texts \textbf{44}, Cambridge University Press, Cambridge, 1999.
 
\bibitem{Ograph}
P. Olver and Ch. Shakiban,
\textit{Graph theory and classical invariant theory}, Adv. Math. \textbf{75} (1989), no. 2, 212-245.

\bibitem{PV}
V. Popov and E. Vinberg,
\textit{Invariant theory},
Algebraic geometry IV (1994), Springer Berlin Heidelberg, 123-278.

\bibitem{Pro}
Claudio Procesi, \textit{Lie groups: an approach through invariants and representations}, Springer, New York, 2007.

\bibitem{RS}
G. C. Rota, J. A. Stein, \textit{Symbolic method in invariant theory},  Proc. Nat. Acad. Sci. U.S.A. \textbf{83} (1986), no. 4, 844-847.

\bibitem{RStu}
G. C. Rota, B. Sturmfels, \textit{Introduction to invariant theory in superalgebras}, Invariant theory and tableaux (Minneapolis, MN, 1988), 1-35, IMA Vol. Math. Appl.  \textbf{19}, Springer, New York, 1990. 

\bibitem{Stu}
Bernd Sturmfels,
\textit{Algorithms in invariant theory},
Second edition, Texts and Monographs in Symbolic Computation, SpringerWienNewYork, Vienna, 2008.

\bibitem{vazzdiss}
Dana R. Vazzana, \textit{Invariants and projections of lines in projective space}, , Ph.D. Thesis, University of Michigan, 1998 available at https://search.proquest.com/docview/304467090.

\bibitem{vazz}
Dana R. Vazzana, \textit{Invariants and projections of six lines in projective space}, Trans. Amer. Math. Soc. \textbf{353} (2001), no. 7, 2673-2688.

\bibitem{Wei1}
Roland Weitzenb\"ock, \textit{Komplex-Symbolik} (German), Teubner, Leipzig, 1908.

\bibitem{Wei2}
Roland Weitzenb\"ock, \textit{Invariantentheorie} (German), Noordhoff, Groningen, 1923.

\bibitem{Weyl}
Hermann Weyl, \textit{The Classical Groups, Their Invariants and Representations}, Princeton Univ. Press, Princeton, NJ, 1946.

\bibitem{White}
Neil L. White, \textit{The bracket of $2$-extensors}, Proceedings of the fourteenth Southeastern conference on combinatorics, graph theory and computing (Boca Raton, Fla., 1983), Congr. Numer. \textbf{40} (1983), 419-428.

\bibitem{Xin}
Guoce Xin, \textit{A fast algorithm for MacMahon's partition analysis}, Electron. J. Combin. \textbf{11} (2004), no. 1, Research Paper 58, 20 pp.

\end{thebibliography}
\end{document}